\documentclass[11pt]{amsart}

\usepackage[utf8]{inputenc}

\usepackage[T1]{fontenc}
\usepackage{ae,aecompl} 

\usepackage{hyperref}

\usepackage{a4wide}
\usepackage{nicematrix}

\usepackage{amsmath}
\usepackage{amssymb}
\usepackage{amsthm}
\usepackage{mathtools}
\usepackage{latexsym}
\usepackage{centernot}
\usepackage{mathdots}
\usepackage{apptools}
\usepackage{stmaryrd}

\usepackage{enumerate} 
\usepackage{color}

\usepackage{mathrsfs}
\usepackage{tikz}
\usetikzlibrary{matrix,arrows,decorations.pathmorphing}
\usepackage{tikz-cd}

\theoremstyle{plain}
\newtheorem{theorem}{Theorem}[section]
\newtheorem{proposition}[theorem]{Proposition}
\newtheorem{lemma}[theorem]{Lemma}
\newtheorem*{theorem*}{Theorem}
\newtheorem{mainthm}{Theorem}

\theoremstyle{definition}
\newtheorem{definition}[theorem]{Definition}

\theoremstyle{remark}
\newtheorem{remark}[theorem]{Remark}

\DeclareMathOperator{\GL}{GL}

\DeclareMathOperator{\Hom}{Hom}
\DeclareMathOperator{\End}{End}

\DeclareMathOperator{\id}{id}

\DeclareMathOperator{\Gal}{Gal}

\DeclareMathOperator{\Ind}{Ind}
\DeclareMathOperator{\cInd}{c-Ind}

\numberwithin{equation}{section}

\usepackage{etoolbox}
\patchcmd{\thebibliography}{\chapter*}{\section*}{}{}
\begin{document}
\title{The Twisted Doubling Method in Algebraic Families}
\author{Johannes Girsch}
\address{Johannes Girsch, School of Mathematics and Statistics, University of Sheffield, Sheffield, S3 7RH, United Kingdom}
\email{j.girsch@sheffield.ac.uk}
\begin{abstract}
 We define the twisted doubling zeta integrals of Cai-Friedberg-Ginzburg-Kaplan in the setting of algebraic families. We then prove a rationality result and a functional equation for these zeta integrals. This allows us to define an unnormalized $\gamma$-factor associated to certain families of representation of a classical group times a general linear group.
\end{abstract}
\maketitle

\section{Introduction}
A fundamental topic in the study of smooth complex representations of reductive $p$-adic groups is the study of their associated $L$-, $\epsilon$- and $\gamma$-factors. Constructing these local constants is rather subtle and there are various different methods, for example the Rankin-Selberg, the Langlands-Shahidi or the doubling method. These constructions all have their different advantages and caveats. For example the Rankin-Selberg and Langlands-Shahidi method, although hugely successful for many problems, only work for representations with explicit models (e.g.\ generic representations) and for general reductive $p$-adic groups not all supercuspidal representations are generic. The doubling method of Piateski-Shapiro and Rallis can be used to associate local factors to any irreducible complex representation of a classical group $G$, in particular there are no genericity assumptions. However, by using this method, they were not able to construct local constants of representations of $G\times\GL_k$ ( for $k>1$), which is fundamental, for example to state the local Langlands correspondence for $G$. Recently, in a series of a papers (\cite{cai2021twisted},\cite{cai2019doubling},\cite{caidoubling},\cite{gourevitch2022multiplicity}), Cai-Ginzburg-Friedberg-Kaplan were able to achieve this by developing an extension of the doubling method, which we call the twisted doubling method.

In recent years there has been increased interest in representations of reductive $p$-adic groups with coefficients different than $\mathbb C$. Substantial work has been done on $\ell$-modular representations, i.e.\ representations with coefficients in an algebraically closed field of characteristic $\ell\not=p$, but more generally, one can study families of representations, i.e.\ representations with coefficients in some general $\mathbb Z[1/p]$-algebra $A$. It is then a natural question if analogues of local constants can be constructed for such representations. 

For $\ell$-modular representations of $\GL_n$ this is achieved in work of Minguez by using the Godement-Jacquet method (\cite{minguez2012fonctions}) and later by Kurinczuk-Matringe for $\GL_n\times\GL_m$ by using the Rankin-Selberg method (\cite{kurinczuk2017rankin}). For families of representations $L$- and $\epsilon$-factors do not seem to behave well, however Moss was able to construct $\gamma$-factors for ``co-Whittaker'' families of $\GL_n\times\GL_m$ (\cite{moss2015gamma},\cite{moss2016interpolating}). Moreover, he was able to prove a converse theorem which was pivotal in the proof of the local Langlands correspondence for families by Helm-Moss in \cite{helm2018converse}. 

Recently, the author has used the doubling method to construct $\gamma$-factors for families of representations of a classical group. The aim of this paper is to develop the twisted doubling method of Cai-Friedberg-Ginzburg-Kaplan in the setting of families. If one would hope that an approach to the local Langlands correspondence in families via a converse theorem (like in the proof of Helm-Moss for $\GL_n$) could work for classical groups, constructing the relevant $\gamma$-factors of pairs provides a step in that direction. 

We give a quick overview of the results of Cai-Friedberg-Ginzburg-Kaplan that are of relevance to us. Fix a finite extension $F$ of $\mathbb Q_p$ with residue cardinality $q$. In the following we assume for simplicity that $G=\operatorname{Sp}_{2n}(F)$. Let $\pi$ be an irreducible complex representation of $G$ and $\theta$ a generic complex $\GL_k(F)$-representation. Cai-Friedberg-Ginzburg-Kaplan consider the group $G^{\Box,k}=\operatorname{Sp}_{4kn}(F)$ and construct a certain embedding $\iota\colon G\times G\hookrightarrow G^{\Box,k}$. Moreover they associate to $\theta$ a generalized Speh representation $\rho_{2n}(\theta)$ of $\GL_{2kn}(F)$ that has special degeneracy properties. By using this representation $\rho_{2n}(\theta)$ they define for each $s\in\mathbb C$ a space of test functions $I(s,\theta)$ on $G^{\Box,k}$. In \cite{cai2021twisted},\cite{cai2019doubling},\cite{caidoubling},\cite{gourevitch2022multiplicity} the following statements are proven.
\begin{itemize}
    \item For any matrix coefficient $\varphi$ of $\pi$ and any $f\in I(s,\theta)$ the twisted doubling zeta integral
	$$Z(s,f,\varphi)=\int_G\int_{N^\circ}\varphi(g)f(u\iota(g,1))\psi^\bullet (u)dudg$$
	converges if the real part of $s$ is large enough. Here $N^\circ$ is a subgroup of a unipotent subgroup $N^\bullet$ of $G^{\Box,k}$ and $\psi^\bullet$ is a $\mathbb C^\times$-valued character of $N^\bullet$.
    \item Moreover, for any entire section $f$ of $I(s,\theta)$, i.e. a function $\mathbb C\times G^{\Box,k}\to\mathbb C$ such that $f(s,.)\in I(s,\theta)$ and $f(.,g)$ is entire for all $g\in G^{\Box,k}$, the integral $Z(s,f,\varphi)$ is a rational function in $q^{-s}$.
    \item The twisted doubling zeta integral satisfies a functional equation
	$$Z(s,M(s)f,\varphi)=\Gamma(s,\pi\times\theta,\psi)Z(s,f,\varphi)$$
	for a certain integral operator $M(s)$ and an appropriate scalar valued function $\Gamma(s,\pi\times\theta,\psi)$.
\end{itemize}  

Note that the intertwining operator $M(s)$ depends on a choice of measure of some unipotent subgroup and hence $\Gamma(s,\pi,\theta,\psi)$ is not uniquely defined. One can remedy this by constructing a normalizing factor for $M(s)$ (which depends on $\theta$ and $\psi$, but not on $\pi$) and then $\Gamma(s,\pi,\theta,\psi)$ divided by this normalizing factor yields the standard $\gamma$-factor associated to $\pi\times\theta$. However, this normalization factor does not depend on the representation of the classical group and hence is irrelevant for the formulation of a converse theorem. 

The proofs of the above results of Cai-Friedberg-Ginzburg-Kaplan is roughly along the following lines. First, one shows a certain multiplicity at most one statement, namely that 
\begin{equation}\label{eq:multone}
    \dim_{\mathbb C}\left(\Hom_{G\times G}(J_{N^\bullet,{\psi^\bullet}^{-1}}(I(s,\theta)),\widetilde{\pi}\otimes_\mathbb C\pi\right)\leq 1
\end{equation}
outside a discrete subset of $s\in\mathbb C$, where $J_{N^\bullet,{\psi^\bullet}^{-1}}$ is the twisted Jacquet functor. Then one proves that the twisted doubling zeta integral converges for $\operatorname{Re}(s)$ large enough, in which case it defines an element of the aforementioned $\Hom$-space. From this, Bernstein's continuation principle (\cite{banks1998corollary}) yields the rationality and the functional equation of the twisted doubling zeta integral.

In the setting of families we are able to prove the following results. We again assume for simplicity that $G=\operatorname{Sp}_{2n}(F)$ (see Section \ref{seclass} for a definition of the classical groups we prove our results for) and that the residue characteristic of $F$ is odd. Let $R$ be a commutative $\mathbb Z[1/p,\mu_{p^\infty}]$-algebra, where $\mathbb Z[1/p,\mu_{p^\infty}]$ is obtained from $\mathbb Z[1/p]$ by adjoining all $p$-th power roots of unity. Moreover, let $A$ and $B$ be commutative noetherian $R$-algebras. We fix a nontrivial additive character $\psi\colon F\to\mathbb Z[1/p,\mu_{p^\infty}]^\times$ and a smooth $A[G]$-module $\pi$. In the setting of families the theory of generalized Speh representations is not developed yet, so let $\theta$ be a smooth $B[\GL_{2kn}(F)]$-representation of type $(k,n)$ (see Definition \ref{defkn}), i.e.\ $\theta$ satisfies the analogues of the degeneracy conditions that are satisfied by the generalized Speh representations used by Cai-Friedberg-Ginzburg-Kaplan. Then we construct a space of test functions $I(X,\theta)$, which are a class of functions on $G^{\Box,k}$ valued in $\theta\otimes_BB[X^{\pm1}]$. Since we assume that $\theta$ has a certain degenerate Whittaker model we have an evaluation at the identity map $\operatorname{ev}_1\colon\theta\to B$. In Proposition \ref{compactness} we prove that for any $f\in I(X,\theta)$ the sets
$$\mathcal C_j=\{(g,u)\in G\times N^\circ\mid\operatorname{Coeff}_{X^j}(f(u\iota(g,1)))\not=0\}$$
are compact for all integers $j$ and empty for $j$ small enough. This allows us to define for $v\in\pi$, an element $\lambda$ of the contragredient $\widetilde{\pi}$ of $\pi$ and  and a test function $f\in I(X,\theta)$ a Laurent series
$$Z(X,\varphi,f)=\sum_{j=-\infty}^\infty X^j\int_{G\times N^\circ}
\lambda(\pi(g)v)\otimes\operatorname{Coeff}_{X^j}\left(\operatorname{ev}_1(f(u\iota(g,1)))\psi^\bullet (u)\right)d(g,u)$$
in $(A\otimes_R B)[[X]][X^{-1}]$, which we call the \emph{twisted doubling zeta integral} associated to $\varphi$ and $f$.

By using this definition of twisted zeta integrals we are in a position to prove the following rationality result. Let $S_{A\otimes_R B}$ be the multiplicative subset of $(A\otimes_R B)[X^{\pm1}]$ consisting of Laurent polynomials with leading and trailing coefficient a unit. 
\begin{mainthm}[Theorem \ref{rattheo}]
If the contragredient $\widetilde{\pi}$ of $\pi$ is admissible and finitely generated as an $A[G]$-module there is a polynomial $\mathcal P\in S_{A\otimes_R B}$ such that 
$$\mathcal P\cdot Z(X,v,\lambda,f)\in (A\otimes_R B)[X^{\pm1}]$$
for any $v\in\pi,\lambda\in\widetilde{\pi}$ and $f\in I(X,\theta)$. 
\end{mainthm}
Our proof of this result is motivated by the proof of the aforementioned multiplicity at most one theorem (Equation (\ref{eq:multone})). Namely, we use a certain filtration on $I(X,\theta)$ that arises, similarly to the geometric Lemma, from the action of $\iota(G\times G)N^\bullet$ on the flag variety $P\backslash G^{\Box,k}$ and then via support of elements in $I(X,\theta)$ on different orbits. We give a quick idea of the proof:

First note that the twisted doubling zeta integral descends to a map in
$$\Hom_{(A\otimes_R B)[X^{\pm1}]}\left(\mathcal M,(A\otimes_R B)[[X]][X^{-1}]\right),$$
where $$\mathcal M=(\pi\otimes_A\widetilde{\pi})\otimes_{R[G\times G]}J_{N^\bullet,{\psi^\bullet}^{-1}}(I(X,\theta)).$$
The above described filtration on $I(X,\theta)$ then gives rise to a filtration of $\mathcal M$. We show that on the bottom piece $\mathcal M_0$ of this filtration, the twisted doubling zeta integral has values in Laurent \emph{polynomials}. As a next step we prove that the successive quotients and hence also $\mathcal M/\mathcal M_0$ are annihilated by Laurent polynomials in $S_{A\otimes_R B}$, which by the $(A\otimes_R B)[X^{\pm1}]$-linearity of the twisted doubling zeta integral implies the result. To prove this we have to use the recent deep result of Dat-Helm-Kurinczuk-Moss (\cite{dat2022finiteness}) that the Jacquet functor preserves admissibility of representations over general noetherian rings.

Next we state the functional equation for the twisted doubling zeta integrals.
\begin{mainthm}[Theorem \ref{funcequ}]
Suppose that 
\begin{itemize}
    \item the contragredient $\widetilde{\pi}$ of $\pi$ is admissible, finitely generated as an $A[G]$-module and the canonical trace map $\pi\otimes\widetilde{\pi}\to A$ is surjective,
    \item $\theta$ is of type $(k,n)$ and the evaluation at the identity of a certain degenerate Whittaker model of $\theta$ is surjective, 
    \item $\Hom_{(A\otimes_R B)[G}(\pi\otimes_A\widetilde{\pi}\otimes_R B,A\otimes_R B)\cong A\otimes_R B,$
    where $G$ acts diagonally on $\pi\otimes_A\widetilde{\pi}$.
\end{itemize}
    Then there exists a unique element $\Gamma(X,\pi,\theta,\psi)\in S_{A\otimes_R B}^{-1}\cdot (A\otimes_R B)[X^{\pm1}]$ such that 
    $$\Gamma(X,\pi,\theta,\psi)Z(X,v,\lambda,f)=Z(X^{-1},v,\lambda,\mathbb M(f))$$
    for all $v\in \pi,\lambda\in\widetilde{\pi}$ and $f\in I(X,\theta)$.
\end{mainthm}
 To prove this result we first construct by following ideas of Dat (\cite{dat2005nu}) a certain intertwining operator $\mathbb M$ (see Proposition \ref{intertwining}). Then we show that under our conditions 
 $$\Hom_{(A\otimes_R B)[X^{\pm1}]}(\mathcal M_0,(A\otimes_R B)[X^{\pm1}]),$$
 is free of rank one as an $(A\otimes_R B)[X^{\pm1}]$-module and the twisted doubling zeta integral is a generator. If one multiplies $Z\circ\mathbb M$ with a certain element of $S_{A\otimes_R B}$ one also gets an element of the above $\Hom$-space and the result follows.
 
When setting $R=A=B=\mathbb C$ and $G$ split our proofs basically reduce to the proofs of Cai-Friedberg-Ginzburg-Kaplan, although we do not use Bernstein's continuation principle (in particular our proofs do not rely on the corresponding statements over the complex numbers). Hence we also obtain new results over the complex numbers, i.e.\ we extend the multiplicity at most one statement of \cite{gourevitch2022multiplicity} to the setting of \cite{cai2021twisted}, see Remark \ref{resultoverc}.

The $\Gamma$-factor constructed above behaves nicely under base change and interpolates the $\Gamma$-factors of Cai-Friedberg-Ginzburg-Kaplan over $\mathbb C$. For example, suppose that $R=B=\mathbb C$ and that the residue field $\kappa(\mathfrak p)$ of a prime ideal $\mathfrak p\in\operatorname{Spec}(A)$ is isomorphic to the complex numbers. Then the reduction of the $\Gamma$-factor $\Gamma(X,\pi,\theta,\psi)\in S_A^{-1}\cdot A[X^{\pm1}]$ modulo $\mathfrak p$ equals the $\Gamma$-factor $\Gamma(X,\pi\otimes\kappa(\mathfrak p),\theta,\psi)$ constructed by Cai-Friedberg-Ginzburg-Kaplan (when fixing a $\mathbb Z[1/p,\mu_{p^\infty}]$-valued measure on a certain unipotent, which is used to define the intertwining operator $\mathbb M$). In future work we plan to study generalised Speh representations in families and construct a normalising factor for the intertwining operator $\mathbb M$, which, together with the results of this article, yields a $\gamma$-factor in families that interpolates the standard $\gamma$-factor of pairs of representations of $G\times\GL_{k}(F)$.


\subsection*{Acknowledgments}
This article is part of my PhD thesis which was written under the excellent supervision of David Helm. I am very grateful for his guidance and support. I would also like to thank my PhD examiners Toby Gee and Shaun Stevens for carefully reading my thesis and many suggestions which vastly improved this article. Moreover, I would like to thank Eyal Kaplan and Robert Kurinczuk for helpful discussions. This work was supported by the Engineering and Physical Sciences Research Council [EP/V001930/1] and [EP/L015234/1], the EPSRC Centre for Doctoral Training in Geometry and Number Theory (The London School of Geometry and Number Theory), University College London and Imperial College London.

\section{Notation and Conventions}
We fix a prime number $p$ and let $F$ be a finite extension of $\mathbb Q_p$. Moreover, $E$ will always be a finite extension of $F$, with ring of integers $\mathcal O_E$, uniformizer $\varpi_E$, and its residue field has cardinality $q_E$. We denote by $D$ a central division algebra over $E$. For a free right $D$-module $W$ let $\operatorname{Nrd}_W$ (respectively $\operatorname{Trd}_W$) be the reduced norm (respectively the reduced trace) of the central simple $E$-algebra $\End_D(W)$. Moreover, we assume that $D$ comes equipped with an involution $\rho\colon D\to D$. For any $m\times n$ matrix $Y\in M_{m\times n}(D)$ we will write $Y^*$ for the $n\times m$ matrix that is the transpose of $Y$ and has the involution $\rho$ applied to all of its entries. We will denote the identity matrix of $M_n(D)$ by $1_n$ and the zero matrix by $0_n$.

We write $\mathbb Z[1/p,\mu_{p^\infty}]$ for the ring that originates when adjoining all $p$-th power roots of unity to $\mathbb Z[1/p]$. In the following $R$ will always be a commutative $\mathbb Z[1/p,\mu_{p^\infty}]$-algebra. Usually when we take tensor products over $R$, we will drop the $R$ from $\otimes_R$. Moreover, we fix a nontrivial additive character $\psi_F\colon F\to\mathbb Z[1/p,\mu_{p^\infty}]^\times$ and we set $\psi=\psi_F\circ\operatorname{tr_{E/F}}$. 

Suppose that $G$ is a locally profinite group. Then a smooth representation of $G$ over a noetherian commutative ring $A$ is an $A[G]$-module, where every element is stabilized by an open compact subgroup of $G$. Sometimes we will denote such a representation by a pair $(\pi,V)$ (where $V$ is an $A$-module and $\pi\colon G\to\GL_A(V)$ is a group homomorphism) to distinguish the $A$-module and the group action. We say that an $A[G]$-module is $G$-finite if it is finitely generated as an $A[G]$-module. Moreover, an $A[G]$-module will be called admissible if for any open compact subgroup $K$ of $G$ the submodule of vectors which are fixed under the action of $K$ is finitely generated as an $A$-module. For a smooth representation $(\pi,V)$ over $A$ we denote the contragredient of $V$ by $(\widetilde{\pi},\widetilde{V})$, which consists of the smooth vectors in $\operatorname{Hom}_A(V,A)$ under the action $\widetilde{\pi}(g)(\lambda)(v)=\lambda(\pi(g^{-1})v)$ for $\lambda\in\widetilde{V},v\in V$.

Let $H$ be a closed subgroup of $G$ and $(\sigma,W)$ a smooth $A[H]$-module. We write $\operatorname{Ind}_H^G(\sigma)$ (respectively $\operatorname{c-Ind}_H^G(\sigma)$) for the space of functions (functions whose support modulo $H$ is compact) $f\colon G\to W$ satisfying $f(hg)=\sigma(h)f(g)$ for all $h\in H,g\in G$ and which are moreover smooth with respect to the action of $G$ via right translation.

For any representation $(\sigma,W)$ of $H$ we define for $g\in G$ the representation $(g\cdot\sigma,W)$ of $\prescript{g}{}{H}=gHg^{-1}$ on $W$ via $(g\cdot\sigma)(x)(w)=\sigma(g^{-1}xg)(w)$ for $x\in\prescript{g}{}{H},w\in W$. 
For a subgroup $H$ of $G$ and a character $\omega\colon H\to A^\times$ we will write $J_{H,\omega}(V)$ for the quotient module $V/V(H,\omega)$ where $V(H,\omega)$ is the $A$-submodule generated by all elements $\sigma(h)v-\omega(h)v$ for $h\in H,v\in V$. If $\omega$ is the trivial character of $H$, we will write $J_H(V)$ for $J_{H,\omega}(V)$. 

From now on we assume that $G$ is the $F$-points of a reductive algebraic group defined over $F$. Then Vigneras (\cite{vigneras1996representations}) constructs a non-zero $R$-valued left Haar measure on $G$, which is normalized over a compact open subgroup of pro-order invertible in $R$. In particular we get an $R$-valued measure on any closed subgroup $H$ of $G$ which we denote by $\mu_H$. For any such $H$ of $G$ we will denote the associated modulus character by $\delta_H$. One can show that $\delta_H$ has values in the units of $R$ (it is always a power of $p$) and is multiplicative.  

We assume that $R$ contains a square-root of $p$ and fix a choice of such. Then if $P=MN$ is a parabolic subgroup of $G$ and $\sigma$ a smooth representation of $M$ we will write $i_P^G(\sigma)$ for the normalized parabolic induction $\Ind_P^G(\delta_P^{1/2}\otimes_R\sigma)$. Moreover, for a  representation $(\pi,V)$ of $G$ we will denote the normalized Jacquet module $(\delta_P^{-1/2}\otimes_R\pi,J_N(V))$ by $r_P^G(\pi)$.

For any commutative ring $A$ let $S_A$ be the multiplicative subset of the ring of Laurent polynomials $A[X^{\pm1}]$ which consists of Laurent polynomials whose leading and trailing coefficients are units in $A$.

\section{Background in Representation Theory}	
In this section we review some standard results about the representations of reductive $p$-adic groups in the setting of families. 
	\subsection{Hecke Algebras}
	Recall that $G$ are the $F$-points of a reductive gorup defined over $F$ and let $H$ be any closed subgroup of $G$. Then the $A$-valued Hecke algebra $\mathcal H(H,A)$ is the $A$-module
	$$\{f\colon H\to A\mid f\text{ is locally constant and has compact support}\},$$
	where multiplication is given by the convolution product
	$$(\alpha *\beta)(h)=\int_H\alpha(x)\beta(x^{-1}h)dx,$$
	for $\alpha,\beta\in\mathcal H(H,A)$. 
 
	Recall that there is a standard correspondence between smooth left $A[G]$-modules and nondegenerate left $\mathcal H(G,A)$-modules. We can define a right $\mathcal H(G,A)$-module structure on a smooth left $A[G]$-module $(\pi,V)$ via 
	$$vf=\int_Gf(g)\pi(g^{-1})vdg,$$
	where $v\in V$ and $f\in\mathcal H(G,A)$. The following result is straightforward to see.
	\begin{proposition}
		Let $\pi$ be a smooth left $A[G]$-module, then
		$$\mathcal H(G,A)\otimes_{\mathcal H(G,A)} \pi\cong \pi\otimes_{\mathcal H(G,A)}\mathcal H(G,A)\cong\pi.$$
	\end{proposition}
	If $H$ is a closed subgroup of $G$, the Hecke algebra $\mathcal H(H,A)$ is not necessarily a subalgebra of $\mathcal H(G,A)$. However, one can define a right $\mathcal H(H,A)$-module structure on $\mathcal H(G,A)$ via 
 \begin{equation*}
     (\alpha\cdot\eta)(g)\coloneqq\int_H\alpha\left(gh^{-1}\right)\eta(h)dh,
 \end{equation*}
 where $\alpha\in\mathcal H(G,A)$ and $\eta\in\mathcal H(H,A)$. Then $\mathcal H(G,A)$ becomes a $(\mathcal H(G,A),\mathcal H(H,A))$-bimodule. Again the following result is straightforward to see. 
	\begin{proposition}\label{assocheck}
		Let $\pi$ (respectively $\pi'$) be a smooth left $A[G]$-module (respectively $A[H]$-module). We have
		$$\pi\otimes_{\mathcal H(G,A)}(\mathcal H(G,A)\otimes_{\mathcal H(H,A)}\pi')\cong (\pi\otimes_{\mathcal H(G,A)}\mathcal H(G,A))\otimes_{\mathcal H(H,A)}\pi'\cong \pi\otimes_{\mathcal H(H,A)}\pi'.$$
	\end{proposition}
	For a smooth left $A[H]$-module $(\pi,V)$ we can form the tensor product $\mathcal H(G,A)\otimes_{\mathcal H(H,A)}V$ and let
	$$\Phi\colon\mathcal H(G,A)\otimes_{\mathcal H(H,A)}V\to\operatorname{c-Ind}_H^G(V\otimes\delta_H)$$ 
	be the map defined by
	$$\Phi(f\otimes v)(g)=\int_Hf(g^{-1}h^{-1})\pi(h^{-1})vdh.$$
	We have the following result.
	\begin{proposition}\label{comphecke} For any closed subgroup $H$ of $G$ the above defined map $\Phi$ is well-defined and an isomorphism of $\mathcal H(G,A)$-modules. 

	\end{proposition}

	\begin{proof}
		It is straightforward to check that the map is well-defined and an $\mathcal H(G,A)$-module homomorphism. Hence it remains to show that $\Phi$ is bijective. Since $G$ are the $F$-points of a reductive $p$-adic group we can choose a neighbourhood basis $\{K_i\}_{i\in I}$ of compact open subgroups of the identity consisting of pro-$p$ subgroups. In particular this implies that the volume $\mu_H(gK_ig^{-1}\cap H)$ is invertible in $A$ for all $g\in G,i\in I$. The bijectivity of $\Phi$ follows if we show that for all small enough open compact subgroups $K$ in the neighbourhood basis $\{K_i\}_{i\in I}$, the restriction of $\Phi$ to the $K$-invariant elements is bijective. 
        
		In the following $K$ is always an element of $\{K_i\}_{i\in I}$. We choose a set of representatives $\Gamma$ for the double quotient $H\backslash G/K$, which allows us to write
		$$\operatorname{c-Ind}_H^G(V\otimes\delta_H)^K\cong\bigoplus_{\gamma\in\Gamma}(V\otimes\delta_H)^{H\cap\gamma K\gamma^{-1}}.$$
		We first deal with surjectivity. Fix $\gamma\in\Gamma$ and $v\in(V\otimes\delta_H)^{H\cap\gamma K\gamma^{-1}}$ and consider at the function $\varphi_{\gamma,v}\in\operatorname{c-Ind}_H^G(V\otimes\delta_H)^K$ which is defined via
		$$\varphi_{\gamma,v}(g)=\begin{cases*}\delta_H(h)\pi(h)v&\text{ if }$g=h\gamma k\in H\gamma K$,\\
			0&\text{else.}
		\end{cases*}$$
		Then it is enough to show that all of these functions are in the image of $\Phi$. For that, let $f_{\gamma}$ be the characteristic function of $K\gamma^{-1}$ times $\mu_H(H\cap\gamma K\gamma^{-1})^{-1}$. We claim that 
		$$\Phi(f_\gamma\otimes v)=\varphi_{\gamma,v}.$$
		Note that
		$$\Phi(f_\gamma\otimes v)(g)=\int_Hf_\gamma(g^{-1}h^{-1})\pi(h^{-1})vdh$$
		and this integral can only be nonzero if $g^{-1}h^{-1}$ is an element of $K\gamma^{-1}$ for at least one $h\in H$, which happens if and only if $g\in H\gamma K$. If $g\in H\gamma K$ there are $h'\in H$ and $k\in K$ such that $g=h'\gamma k$. Then we have
		$$\Phi(f_\gamma\otimes v)(h'\gamma k)=\delta_H(h')\pi(h')\int_Hf_\gamma(k^{-1}\gamma^{-1}h^{-1})\pi(h^{-1})vdh$$
		and $k^{-1}\gamma^{-1}h^{-1}\in K\gamma^{-1}$ if and only if $h\in\gamma K\gamma^{-1}$. Since $v\in(V\otimes\delta_H)^{H\cap\gamma K\gamma^{-1}}$ the integral equals $v$ and we obtain $\Phi(f_\gamma\otimes v)=\varphi_{\gamma,v}$.
  
		We now prove that $\Phi$ is injective. Firstly, note that if $\alpha\in\mathcal H(G,A)$ and  $K'$ is small enough open compact subgroup of $G$ such that $\alpha(gk)=\alpha(g)$ for all $g\in G,k\in K'$, then we have for all $h\in H$ that
		\begin{equation}\label{eq:tens}
			(\alpha\cdot e_{K'h\cap H})(g)=\int_H\alpha(gx^{-1})e_{K'h\cap H}(x)dx=\alpha(gh^{-1}).
		\end{equation}
		Let $\sum_{j=1}^n\alpha_j\otimes v_j$ be a nonzero element of $\mathcal H(G,A)\otimes_{\mathcal H(H,A)}V$. Choose an open compact subgroup $K$ of $G$ small enough such that $\alpha_j(kg)=\alpha_j(g)$ for all $g\in G,k\in K$ and $j=1,\dotsc,n$. Each $\alpha_j$ is a finite linear combination of functions $g\mapsto e_{K\gamma^{-1}}(gh)$ where $h\in H$ and $\gamma\in\Gamma$. Hence by Equation (\ref{eq:tens}) we can write 
		$$\sum_{j=1}^n\alpha_j\otimes v_j=\sum_{l=1}^me_{K\gamma_l^{-1}}\otimes v_l',$$
		where $\{\gamma_1,\dotsc,\gamma_m\}$ is a finite subset of $\Gamma$ and at least one $v_l'$ is nonzero. We write $\beta_l$ for $e_{K\gamma_l^{-1}}$. The support of $\Phi(\beta_l\otimes v_l')$ is contained in $H\gamma_lK$ and hence it is enough to show that at least one $\Phi(\beta_l\otimes v_l')$ is nonzero.
		Note that $\beta_l$ is right invariant under $\gamma_lK\gamma_l^{-1}$, in particular $\beta_l\cdot e_{\gamma_lK\gamma_l^{-1}\cap H}=\beta_l$, which implies that $\beta_l\otimes v_l'=\beta_l\otimes e_{\gamma_lK\gamma_l^{-1}\cap H}v_l'$. Hence we can assume that $v_l'\in V^{\gamma_lK\gamma_l^{-1}\cap H}$. We obtain
		$$\Phi(\beta_l\otimes v_l')(\gamma_l)=\int_Hf_l(\gamma_l^{-1}h^{-1})\pi(h^{-1})v_l'dh=\frac{\mu_H(\gamma_lK\gamma_l^{-1}\cap H)}{\mu_G(K)}v_l',$$
		which shows injectivity and concludes the proof.

	\end{proof}
        Suppose we have two smooth (left) $A[H]$-modules $\pi_1,\pi_2$. Then we have a diagonal action of $H$ on $\pi_1\otimes_A\pi_2$. The coinvariants under this action
	$$\pi_1\otimes_A \pi_2/\langle v\otimes v'-hv\otimes hv'\mid v\in \pi_1,v'\in \pi_2,h\in H\rangle$$
	are isomorphic to the tensor product $\pi_1\otimes_{A[H]}\pi_2$, where $\pi_1$ becomes a right $A[H]$-module via $v\cdot h=h^{-1}v$ for $h\in H,v\in \pi_1$. We will need the following result.
	\begin{lemma}\label{tensprod} The $A$-submodule 
		$$\langle h^{-1}v_1\otimes v_2-v_1\otimes hv_2\mid v_1\in \pi_1,v_2\in \pi_2,h\in H\rangle$$ of $\pi_1\otimes_A \pi_2$ equals
  $$\langle v_1f\otimes v_2-v_1\otimes fv_2\mid v_1\in \pi_1,v_2\in \pi_2,f\in\mathcal H(H,A)\rangle.$$
 In particular we can identify $\pi_1\otimes_{A[H]}\pi_2$ with $\pi_1\otimes_{\mathcal H(H,A)}\pi_2$.
	\end{lemma}
\begin{proof}
	Firstly, let $g\in H,v_1\in \pi_1,v_2\in \pi_2$. Choose $K$ to be a small enough open compact subgroup of $H$ such that $g^{-1}v_1\in \pi_1^K$ and $v_2\in \pi_2^K$. Then the function $e_{gK}$ satisfies $v_1e_{gK}=g^{-1}v_1$ and $e_{gK}v_2=gv_2$ and hence $$g^{-1}v_1\otimes v_2-v_1\otimes gv_2=v_1e_{gK}\otimes v_2-v_1\otimes e_{gK}v_2.$$
 
    Now, let $f\in\mathcal H(H,A),v_1\in \pi_1,v_2\in \pi_2$. Choose $K$ to be a small enough open compact subgroup such that $f\in e_K*\mathcal H(H,A)*e_K$, $v_1\in \pi_1^K$ and $v_2\in \pi_2^K$. Then we can find finitely many elements $g_1,\dotsc,g_n$ of $H$ such that $\operatorname{supp}(f)\subseteq\bigcup_{i=1}^n(g_iK\cap Kg_i)$. Moreover, there is a finite refinement $V_j$ of this covering consisting of disjoint open compact sets. Then $f$ is a finite sum of functions each supported in one $V_j$. For such a summand $\alpha$ we have 
	$$v_1\alpha\otimes v_2-v_1\otimes\alpha v_2=\mu_H(\operatorname{supp}(\alpha))\alpha(g_j)(g_j^{-1}v_1\otimes v_2-v_1\otimes g_jv_2)$$
	for some $j$ and the result follows. 
\end{proof}
	
	\subsection{Finiteness Results}
	To prove properties of the zeta integrals we define later we need certain finiteness results on families of representations which we discuss now.
	\begin{proposition}\label{endfin}
		Let $(\pi,V)$ be an admissible $G$-finite representation over a noetherian ring $A$. Then we have
    \begin{enumerate}
        \item The endomorphism algebra $\operatorname{End}_{A[G]}(V)$ is a finitely generated $A$-module.
        \item If $z$ is an element of the center of $G$ there is a polynomial $Q(X)\in S_A$ such that $Q(\pi(z))$ annihilates $V$.
    \end{enumerate}
        \end{proposition}
	\begin{proof}
	ad (1):	By assumption there exist $v_1,\dotsc,v_n$ which generate $V$ as an $A[G]$-module. Let $K$ be a compact open subgroup of $G$ such that $v_i\in V^K$ for all $i=1,\dotsc,n$. Since an element of $\operatorname{End}_{A[G]}(V)$ is completely determined by its values on a generating set and it maps $V^K$ to itself, we have by restriction an embedding of $A$-modules $\operatorname{End}_{A[G]}(V)\hookrightarrow\operatorname{End}_A(V^K)$. By admissibility, $V^K$ is finitely generated as an $A$-module and since $A$ is noetherian we obtain that $\operatorname{End}_A(V^K)$ is a finitely generated $A$-module, which implies the result.
 
 ad (2): Since $z$ is in the center of $G$ we have that $\pi(z)\in\operatorname{End}_{A[G]}(V)$. By part 1) and noetherianity of $A$, we see that $A[\pi(z)]\subseteq\operatorname{End}_{A[G]}(V)$ is a finitely generated $A$-module and hence there is a monic polynomial $Q_1(X)$ such that $Q_1(\pi(z))=0$. Analogously we obtain a monic polynomial $Q_2(X)$ such that $Q_2(\pi(z^{-1}))=0$. Then we have that $Q(X)=Q_1(X)+Q_2(X^{-1})\in S_A$ and $Q(\pi(z))=0$. By multiplying with a large enough power of $X$ we can assume that $Q(X)\in A[X]$.
	\end{proof}
        The next two results imply that the conditions in the above proposition are preserved under Jacquet functors. The following is proven in  II.2.1 of \cite{vigneras1996representations}.
	\begin{proposition}\label{jacfin}
		Suppose that $\pi$ is finitely generated as an $A[G]$-module and let $P=MN$ be a parabolic subgroup of $G$. Then $r_P^G(\pi)$ is finitely generated as an $A[M]$-module.
	\end{proposition}
We will need the following recent deep result of Dat-Helm-Kurinczuk-Moss which was previously proven by Dat in special cases for classical groups (see \cite[Prop 6.7]{dat2005nu} and \cite[Cor 1.6]{dat2009finitude}).
	\begin{theorem}[Cor 1.5 in \cite{dat2022finiteness}]\label{dathelm}
		Let $G$ be the $F$-points of a connected reductive group and $A$ a noetherian ring. Then the Jacquet functor preserves admissibility for all parabolic subgroups.
	\end{theorem}
	
	\begin{remark}
 \begin{enumerate}
     \item Later we will also have to consider the $F$-points of certain disconnected reductive groups (for example orthogonal groups). However, if $G$ is not connected, but the above statement is known for $F$-points of the neutral component $G^0$ of $G$, it also holds for $G$. Namely, if $P=MN$ is a parabolic subgroup of $G$, then $P^0=M^0N$, where $M^0=M\cap G^0$, is a parabolic subgroup of $G^0$. Hence the parabolic restrictions $r^G_P(\pi)$ and $r^{G^0}_{P^0}(\pi|_{G^0})$ have the same underlying set of elements $J_N(\pi)$. Since $G^0$ is open in $G$ we have that $M^0$ is open in $M$. For any compact open subgroup $K$ of $M$ we can find an open compact subgroup $K^0$ of $M^0$ that is contained in $K$. The invariants $J_N(\pi)^K$ are then contained in $J_N(\pi)^{K^0}$ which is finitely generated as an $A$-module. 
     \item Note that $G$-finiteness and admissibility are preserved under twisting by smooth characters. In particular the above results also hold for the unnormalized Jacquet functor. 
 \end{enumerate}
 
	\end{remark}

 \subsection{$\ell$-Sheaves and Cosmooth Modules} We briefly recall the theory of $\ell$-sheaves and cosmooth modules and indicate that the relevant results apply to families of representations. A good introduction is Section 4.3 of \cite{bump1998automorphic}, whose proofs carry over to the setting of general coefficient rings.
	In this section let $\mathbb X$ be a locally compact totally disconnected topological space and $R$ a commutative ring. We denote by $\mathcal I_c$ the collection of compact open subsets of $\mathbb X$, which, by assumption, is a basis of the topology of $\mathbb X$. Let $\mathcal C_c^\infty$ be the sheaf of rings on $\mathbb X$, where for any open subset $U$ of $\mathbb X$, the sections $\mathcal C_c^\infty(U)$ are the locally constant compactly supported $R$-valued functions on $U$ with the obvious restriction maps. The multiplication structure is given by pointwise multiplication of functions. Note that the ring $C_c^\infty(U)$ is unital if and only if $U\in\mathcal I_c$.
 
	Let $\mathcal F$ be a sheaf of $\mathcal C_c^\infty$-modules on $\mathbb X$, where we assume that for $U\in\mathcal I_c$ the $\mathcal C_c^\infty(U)$-module $\mathcal F(U)$ is unital. For any compact open subsets $V\subseteq U$ we can define a map $\iota_{V,U}\colon \mathcal F(V)\to\mathcal F(U)$ by setting $\iota_{V,U}(x)$ to be the unique element in $\mathcal F(U)$ that satisfies $\iota_{V,U}(x)|_V=x$ and $\iota_{V,U}(x)|_{U-V}=0$. We obtain a direct system $(\mathcal F(V),\iota_{V,U})$ and set $$\mathcal F_c=\lim_{\longrightarrow}\mathcal F(V).$$
	Note that one obtains a $\mathcal C_c^\infty(\mathbb X)$-module structure on $\mathcal F_c$. A $\mathcal C_c^\infty(\mathbb X)$-module $M$ is called cosmooth if for every element $x\in M$ there is a compact open subset $U\subseteq\mathbb X$ such that $\mathbf 1_Ux=x$. It is straightforward to see that $\mathcal F_c$ is a cosmooth $\mathcal C_c^\infty(\mathbb X)$-module. We have the following result. 
	\begin{proposition}
		There is an equivalence of categories
		\begin{align*}
			\left\{\text{sheaves of }\mathcal C_c^\infty\text{-modules on }\mathbb X\right\}
			&\Leftrightarrow
			\left\{\text{cosmooth }\mathcal C_c^\infty(\mathbb X)\text{-modules}\right\}\\
			\mathcal F&\mapsto\mathcal F_c\\
			\mathcal M &\mapsfrom M,
		\end{align*}
		where $\mathcal M$ is the sheaf on $\mathbb X$ defined by setting for compact open $V\subseteq U$ that $\mathcal M(U)=\mathbf 1_UM$ and restriction map $x|_V=\mathbf 1_Vx$ for $x\in\mathcal M(U)$. 
	\end{proposition}
	\begin{proof}
		See Section 4.10 of \cite{vigneras1996representations} and Section 4.3 of \cite{bump1998automorphic}.
	\end{proof}
	By using this correspondence the stalks of $\mathcal C_c^\infty$-sheaves can be described in the following way. See Proposition 4.3.13 in \cite{bump1998automorphic} for the result over $\mathbb C$.
	\begin{proposition}\label{stalksheaf}
		Let $M$ be a cosmooth $\mathcal C_c^\infty(\mathbb X)$-module and let $\mathcal M$ be the corresponding $\mathcal C_c^\infty$-sheaf. For any $x\in\mathbb X$ let $M(x)$ be the $R$-submodule of $M$ consisting of elements $m\in M$, for which there is a compact open subset $U\subseteq\mathbb X$ containing $x$ such that $\mathbf 1_Um=0$. Then there is an $R$-module isomorphism between the stalk $\mathcal M_x$ and $M/M(x)$.
		
	\end{proposition}
	\begin{proof}
		For any $U\in\mathcal I_c$ we consider the map $\pi_U\colon\mathcal M(U)\to M/M(x)$, obtained by composing the inclusion $\mathcal M(U)\to M$ with the natural projection  $M\to M/M(x)$. Let $V\subseteq U$ be elements of $\mathcal I_c$ with $x\in V$. Since for any $m\in\mathcal M(U)$ we have that $\mathbf 1_V(\mathbf 1_Vm-m)=0$ and $m|_V=\mathbf 1_Vm$ we see that $\pi_V(m|_V)=\pi_U(m)$ for all $m\in\mathcal M(U)$. By the universal property of the direct limit we obtain an $R$-linear morphism $\sigma\colon\mathcal M_x\to M/M(x)$, such that for any $U\in\mathcal I_c$ with $x\in U$, the canonical map $\mathcal M(U)\to\mathcal M_x$ composed with $\sigma$ equals $\pi_U$.
  
		It is straightforward to see that any $m\in M$ lies in the image of an inclusion $\mathcal M(U)\to M$ where $U\in\mathcal I_c$ and $x\in U$. This implies the surjectivity of $\sigma$.
  
		Now if $\overline{m}$ lies in the kernel of $\sigma$, there is an $U\in\mathcal I_c$ with $x\in U$ such that $m\in\mathcal M(U)$ maps to $\overline{m}$ and $\pi_U(m)=0$. However, then $m$ is an element of $M(x)$ and by definition there is an open compact subset $V$ with $x\in V$ such that $\mathbf 1_Vm=0$. This immediately implies that $\overline{m}=0$.

	\end{proof}
	If $Z$ is a closed subset of $\mathbb X$ and $\mathcal F$ a $\mathcal C_c^\infty$-sheaf on $\mathbb X$ we can consider the inverse image sheaf $\mathcal F_Z$ on $Z$ with respect to the inclusion $Z\hookrightarrow\mathbb X$. Then $\mathcal F_Z$ is a $\mathcal C_c^\infty$-sheaf on $Z$. Moreover, one can construct an $R$-linear map $\mathcal F_c\to(\mathcal F_Z)_c$ in the following way. Suppose $f\in\mathcal F_c$ is represented by an element $f\in\mathcal F(U)$ for some compact open subset $U$ of $\mathbb X$. Note that we have a chain of $R$-linear maps 
	$$\mathcal F(U)\to\varinjlim_{V\supseteq U\cap Z}\mathcal{F}(V)\to\mathcal F_Z(Z\cap U)\to (\mathcal F_Z)_c$$ which defines the image of $f$ in $(\mathcal F_Z)_c$. By Proposition 4.3.14 of \cite{bump1998automorphic} this map is surjective.

	Suppose we have a continuous action of $G$ on $\mathbb X$. Let $\mathcal F$ be a sheaf of $R$-modules on $\mathbb X$ and we say that $G$ acts on the pair $(\mathbb X,\mathcal F)$, if for all $g\in G$ and open subsets $U\subseteq\mathbb X$ we have an $R$-linear isomorphism $$\chi_{U}^g\colon\mathcal F(U)\to\mathcal F(gU)$$ such that for all open $V\subseteq U$ the diagram
	$$\begin{tikzcd}
		\mathcal F(U)\ar[r,"\chi_U^g"]\ar[d,"(-)|_V"]&\mathcal F(gU)\ar[d,"(-)|_{gV}"]\\
		\mathcal F(V)\ar[r,"\chi_V^g"]&\mathcal F(gV)
	\end{tikzcd}$$
	commutes. For any element $f\in\mathcal F(U)$ we will write $g\cdot f$ for $\chi_U^g(f)$. Clearly, $G$ acts on $(\mathbb X,\mathcal C_c^\infty)$ by defining $g\cdot \varphi\in\mathcal C_c^\infty(gU)$ via $g\cdot \varphi(x)=\varphi(g^{-1}x)$ for $x\in\mathbb X$. If $\mathcal F$ is a $\mathcal C_c^\infty$-sheaf and $G$ acts on $(\mathbb X,\mathcal F)$ we moreover assume that $g\cdot(\varphi f)=(g\cdot\varphi)(g\cdot f)$ where $U\subseteq\mathbb X$ is open and $g\in G,\varphi\in\mathcal C_c^\infty(U),f\in\mathcal F(U)$.

	If $\mathcal F$ is a sheaf of $\mathcal C_c^\infty$-modules the action of $G$ gives rise to an action of $G$ on $\mathcal F_c$. The following result over $\mathbb C$ is Proposition 4.3.15 in\cite{bump1998automorphic}. 
	\begin{proposition}\label{propbump}
		Let $\mathbb X$ and $\mathbb Y$ be totally disconnected locally compact spaces and let $\varrho\colon \mathbb X\to\mathbb Y$ be a continuous map. Let $\mathcal F$ be a $\mathcal C_c^\infty$-sheaf on $\mathbb X$. Suppose that $G$ acts on $(\mathbb X,\mathcal F)$ and assume the action satisfies that $\varrho(g\cdot x)=\varrho(x)$ for $g\in G,x\in\mathbb X$. Let $\xi$ be an $R^\times$-valued character of $G$. Suppose that $y\in\mathbb Y$ and let $Z$ be the preimage $\varrho^{-1}(y)$. Then $(\mathcal F_c)_{G,\xi}$ is a cosmooth $\mathcal C^\infty_c(\mathbb Y)$-module. Let $\mathcal G$ be the corresponding sheaf of $\mathcal C_c^\infty$-modules on $\mathbb Y$. For any $y\in\mathbb Y$ the stalk $\mathcal G_y$ is isomorphic to $((\mathcal F_Z)_c)_{G,\xi}$.

		Let $\mathcal F_c(\xi)$ (resp. ($(\mathcal F_Z)_c(\xi)$))) be the submodule of $\mathcal F_c$ (resp. $(\mathcal F_Z)_c$) generated by elements of the form $g\cdot f-\xi(g)^{-1}f$ for $f\in \mathcal F_c$ (resp. $f\in (\mathcal F_Z)_c$) and $g\in G$. \end{proposition}  
	
	\begin{proof} 
		Let $f\in\mathcal F_c$ be represented by $f\in\mathcal F(U)$ for some compact open $U\subseteq\mathbb X$. For $\alpha\in\mathcal C_c^\infty(\mathbb Y)$ we then define $\alpha\cdot f\coloneqq ((\alpha\circ \varrho)\mathbf 1_U )f$. It is not hard to see that this action is well-defined. Now for cosmoothness, note that $\varrho(U)$ is a compact subset of $\mathbb Y$ and hence there is a compact open subset $\overline{U}\subseteq\mathbb Y$ such that $\varrho(U)\subseteq\overline{U}$. Then clearly $$\mathbf 1_{\overline{U}}\cdot f=((\mathbf 1_{\overline{U}}\circ \varrho)\mathbf 1_U )f=\mathbf 1_U f=f.$$
		Since for any $\alpha\in\mathcal C_c^\infty(\mathbb Y)$ and $g\in G$ we have that $g\cdot(\alpha\circ \varrho)=\alpha\circ \varrho$, this gives rise to a cosmooth action of $\mathcal C_c^\infty(\mathbb Y)$ on $(\mathcal F_c)_{G,\xi}$.
  
		Let $L$ be the $A$-submodule of $\mathcal F_c$ generated by elements of the form $(\Phi\circ \varrho)f$ where $\Phi$ is a function in $\mathcal C_c^\infty(\mathbb Y)$ that vanishes at $y$ and $f\in\mathcal F_c$.
		We already mentioned that the canonical map $\mathcal F_c\to (\mathcal F_Z)_c$ is surjective. We will now show that the kernel of this map is precisely $L$. It is clear that $L$ is contained in the kernel.  If $f\in L$ then $\operatorname{supp}(f)$ is compact and disjoint from $Z$. Hence its image under $\varrho$ is a compact subset of $\mathbb Y$ not containing $y$, which implies that there is an open compact subset $U$ of $\mathbb Y$ such that $U\cap \varrho(\operatorname{supp}(f))=\emptyset$. It follows that $f=(\mathbf 1_U\circ \varrho)f$.
  
		Since $(\mathcal F_Z)_c\cong\mathcal F_c/L$ we obtain that 
		$$((\mathcal F_Z)_c)_{G,\xi}\cong\mathcal F_c/(L,\mathcal F_c(G,\xi)).$$
		However, by Proposition \ref{stalksheaf} we have that $\mathcal G_y\cong\mathcal F_c/(L,\mathcal F_c(G,\xi))$, which implies the result.
	\end{proof}
	
	\subsection{$\ell$-Sheaves on $D^r$}

In the following we set $R=\mathbb Z[1/p,\mu_{p^\infty}]$ and let $\psi\colon F\to\mathbb Z[1/p,\mu_{p^\infty}]^\times$ be a nontrivial additive character. Then $\psi_D\coloneqq\psi\circ\operatorname{Trd}$ is a nontrivial additive character of $D$. It is a well known fact that sending $a=(a_1,\dotsc\,a_r)\in D^r$ to the additive character 
\begin{align*}
    \psi_D^a\colon D^r&\to\mathbb Z[1/p,\mu_{p^\infty}]^\times\\
    (x_1,\dotsc,x_r)&\mapsto\psi_D(a_1x_1+\dotsc+a_rx_r)
\end{align*}
 yields an isomorphism between $D^r$ and its group of smooth additive characters.
 \begin{proposition}[Section 3.10 in \cite{vigneras1996representations}]
     The map $\Psi\colon\mathcal H(D^r,\mathbb Z[1/p,\mu_{p^\infty}])\to\mathcal C_c^\infty(D^r)$, given by $$\Psi(f)(a)\coloneqq\int_{D^r}\psi_D^a(x)f(x)\mathrm dx$$ for $f\in\mathcal H(D^r,\mathbb Z[1/p,\mu_{p^\infty}])$ is an $\mathbb Z[1/p,\mu_{p^\infty}]$-algebra isomorphism.
 \end{proposition}
 \begin{proof}
      It is straightforward to check that $\Psi$ is a $\mathbb Z[1/p,\mu_{p^\infty}]$-algebra homomorphism. The map $\Phi\colon\mathcal C_c^\infty(D^r)\to\mathcal H(D^r,\mathbb Z[1/p,\mu_{p^\infty}])$ given by $$\Phi(f)(a)\coloneqq\int_{D^r}\psi_D^x(-a)f(x)\mathrm dx$$ for $f\in\mathcal C_c^\infty(D^r)$ is, up to a scalar, an inverse.
 \end{proof}
	
	This allows us to identify smooth representations of $D^r$ (i.e.\ nondegenerate $\mathcal H(D^r,\mathbb Z[1/p,\mu_{p^\infty}])$-modules) with cosmooth $\mathcal C_c^\infty(D^r)$-modules and subsequently with sheaves of $\mathcal C_c^\infty$-modules on $D^r$.
	\begin{proposition}\label{trivialnil}
		Let $V$ be a smooth $\mathbb Z[1/p,\mu_{p^\infty}][D^r]$-module and $\mathcal V$ the associated sheaf of $\mathcal C_c^\infty$-modules on $D^r$. Then for all $a\in D^r$ the stalk $\mathcal V_a$ is isomorphic to $J_{D^r,\psi_D^{a}}(V)$ as a $\mathbb Z[1/p,\mu_{p^\infty}]$-module.
		Moreover, if $J_{D^r,\psi_D^{a}}(V)=0$ for all $a\not=0$ then $D^r$ acts trivially on $V$.
	\end{proposition}
	\begin{proof}
		An element $v\in V$ lies in $(V,\psi_D^{a})$ if and only if there is a compact open subgroup $K$ of $D^r$ such that $({\psi_D^{a}}^{-1}\cdot e_K)*v=0$. Note that $\Psi({\psi_D^a}^{-1}\cdot e_K)=e_{\hat{K}}$, where $\hat{K}$ is the subset of $D^r$ consisting of elements $y$ such that $\psi_{y-a}$ is trivial on $K$. Now $\hat{K}$ is open and $a\in\hat{K}$. Hence $v\in (V,\psi_D^{a})$ if and only if there is a compact open subgroup $K\subset D^r$ such that $e_{\hat{K}}\cdot v=0$. However, by Proposition \ref{stalksheaf}, we see that $\mathcal V_a\cong V/V(a)$ which implies the result.
  
		Note that $D^r$ acts trivially on $V$ if and only if $J_{D^r,\psi_D^0}(V)=V$. By what we just saw this is equivalent to showing that $V(0)$ is trivial. Suppose now that $v_0\in V(0)$. Since we assume that $J_{D^r,\psi_D^{a}}(V)=0$ for all nonzero $a\in D^r$ we obtain that $v_0\in V(a)$ for all $a\in D^r$ and hence for each $a\in D^r$ there is an open neighbourhood $U_a$ of $a$ such that $\mathbf 1_{U_a}v_0=0$. However, since $V$ is a cosmooth $\mathcal C_c^\infty(D^r)$-module there is an open compact subset $U\subset D^r$ such that $\mathbf 1_Uv_0=v_0$. This immediately implies that $v_0=0$.
	\end{proof}

	\section{Degenerate Whittaker Models}\label{secdegwhit}        In this section we introduce certain degenerate Whittaker models which are crucial for the twisted doubling method. We follow the exposition of Section 2.2.1 in \cite{cai2021twisted}. Recall that we fixed a finite extension $E$ of $\mathbb Q_p$ and a division algebra $D$ over $E$. Moreover, we have an additive character $\psi\colon E\to\mathbb Z[1/p,\mu_{p^\infty}]^\times$ and a noetherian $R$-algebra $B$. In the following, $\theta$ will always be a smooth admissible $B[\GL_r(D)]$-module for some positive integer $r$.
 
	Let $$\mathcal Y\colon Y_0=0\subset Y_1\subset Y_2\subset\dotsc\subset Y_m\subset Y_{m+1}=D^r$$ be a flag of $D$-submodules of $D^r$ and $P(\mathcal Y)$ its stabiliser. We choose a Levi decomposition $P(\mathcal Y)=M(\mathcal Y)N(\mathcal Y)$. For an element $u\in N(\mathcal Y)$ and $0\leq i\leq m$, the endomorphism $u-\id$ maps $Y_{i+1}$ into $Y_i$ and we denote the induced element of $\Hom_D(Y_{i+1}/Y_i,Y_i/Y_{i-1})$ by $u_i$. Then the map
	\begin{align*}
		N(\mathcal Y)^{\text{ab}}&\to\prod_{i=1}^m\Hom_D(Y_{i+1}/Y_i,Y_i/Y_{i-1})\\u&\mapsto (u_1,\dotsc,u_m),
	\end{align*}
	is an isomorphism.
	To define a character of $N(\mathcal Y)$ we choose elements 
	$$\mathcal A=(A_1,\dotsc,A_m)\in\prod_{i=1}^m\Hom_D(Y_i/Y_{i-1},Y_{i+1}/Y_i)$$
	and set $$\psi_\mathcal A(u)\coloneqq\psi\left(\sum_{i=1}^m\operatorname{Trd}(u_i\circ A_i)\right).$$ Whenever we write that $\psi_{\mathcal A}$ is a character of $N(\mathcal Y)$ we refer to this construction for a tuple $\mathcal A=(A_1,\dotsc,A_m)$ as above.
 
	Assume from now on that $r=kn$ for some positive integers $k$ and $n$.
	\begin{definition}\label{orbdef}Let $\lambda_{k,n}=(k,\dotsc,k)$ be the integer partition of $kn$ which consists of $n$-times the summand $k$. For a flag $$\mathcal Y\colon 0\subset Y_1\subset Y_2\subset\dotsc\subset Y_m\subset D^{kn}$$ of $D^{kn}$ and a character $\psi_{\mathcal A}$ of $N(\mathcal Y)$ we say that:
		\begin{enumerate}
			\item 	the pair $(N(\mathcal Y),\psi_\mathcal A)$ lies in the orbit $\lambda_{k,n}$, if 
       \begin{itemize}
       \item $m=k-1$,
       \item the rank of $Y_i$ as a $D$-module is $ni$ for $i=1,\dotsc,k-1$,
       \item $A_i$ is an isomorphism for $i=1,\dotsc,k-1$.
       \end{itemize}
			\item 
			the pair $(N(\mathcal Y),\psi_\mathcal A)$ lies in an orbit higher than $\lambda_{k,n}$ if
    \begin{itemize}
        \item $m\geq k$,
        \item for some $1\leq i\leq m$ we have that 
			$$A_{i+k}\circ\dotsc\circ A_i\not=0.$$
    \end{itemize}
    
		\end{enumerate}
		
	\end{definition}
	\begin{definition}\label{defkn}
		Let $\theta$ be an admissible $\GL_{kn}(D)$-finite $B[\GL_{kn}(D)]$-module that satisfies Schur's lemma. Then $\theta$ is of type $(k,n)$ if 
		\begin{enumerate}
			\item For any pair $(N(\mathcal Y), \psi_\mathcal A)$ that lies in the orbit $\lambda_{k,n}$, the twisted Jacquet module
			$J_{N(\mathcal Y),\psi_\mathcal A}(\theta)$ is free of rank one as a $B$-module.
			\item For any pair $(N(\mathcal Y), \psi_\mathcal A)$ that lies in an orbit higher than $\lambda_{k,n}$, the twisted Jacquet module $J_{N(\mathcal Y),\psi_\mathcal A}(\theta)$ is zero.
		\end{enumerate}
	\end{definition}
	
	If $\theta$ is of type $(k,n)$ and $(N(\mathcal Y),\psi_\mathcal A)$ lies in the orbit $\lambda_{k,n}$ then we have by Frobenius reciprocity that
	$$B\cong\Hom_B(J_{N(\mathcal Y),\psi_\mathcal A}(\theta),B)\cong\Hom_{N(\mathcal Y)}(\theta,\psi_\mathcal A)\cong\Hom_{\GL_{kn}(D)}(\theta,\Ind_{N(\mathcal Y)}^{\GL_{kn}(D)}(\psi_\mathcal A)).$$
    We fix a generator of $\Hom_{\GL_{kn}(D)}(\theta,\Ind_{N(\mathcal Y)}^{\GL_{kn}(D)}(\psi_\mathcal A))$ and denote it by $\operatorname{Wh}_{N(\mathcal Y),\psi_{\mathcal A}}$.

	\subsection{Invariance under stabiliser}
	
	For a pair $(N(\mathcal Y),\psi_\mathcal A)$ let $\operatorname{St}_{(N(\mathcal Y),\psi_\mathcal A)}$ be its stabiliser, i.e.\ elements $g$ in the normalizer $N_{\GL_{kn}(D)}(N(\mathcal Y))$ of $N(\mathcal Y)$ such that $\psi_{\mathcal A}(u)=\psi_{\mathcal A}(gug^{-1})$ for all $u\in N(\mathcal Y)$. If $(N(\mathcal Y),\psi_\mathcal A)$ is a pair in the orbit $\lambda_{k,n}$ then it is straightforward to see that $\operatorname{St}_{(N(\mathcal Y),\psi_\mathcal A)}/N(\mathcal Y)\cong\GL_n(D)$.
	\begin{proposition}[Lemma 2.16 in \cite{cai2021twisted}]\label{invstab}	Let $(N(\mathcal Y),\psi_\mathcal A)$ be a pair in the orbit $\lambda_{k,n}$ and  $\theta$ be a $B[\operatorname{GL}_{kn}(D)]$-module of type $(k,n)$. The stabiliser $\operatorname{St}_{(N(\mathcal Y),\psi_\mathcal A)}/N(\mathcal Y)\cong\GL_n(D)$ acts on $\operatorname{Hom}_{N(\mathcal Y)}(\theta,\psi_\mathcal A)$ via a character $\chi_\theta\colon E^\times\to B^\times$. For any element $f\in\operatorname{Wh}_{N(\mathcal Y),\psi_{\mathcal A}}(\theta)$ we have that 
		$$f(gh)=\chi_\theta(\operatorname{Nrd}(g))f(h)$$ for all $g\in\GL_n(D)\cong\operatorname{St}_{(N(\mathcal Y),\psi_\mathcal A)}/N(\mathcal Y)$ and $h\in\GL_{kn}(D)$.
	\end{proposition}
	\begin{proof} Note that $\operatorname{St}_{(N(\mathcal Y),\psi_\mathcal A)}/N(\mathcal Y)\cong\GL_n(D)$ acts on $\operatorname{Hom}_{N(\mathcal Y)}(\theta,\psi_\mathcal A)$ via $(g\cdot \varphi)(x)=\varphi(gxg^{-1})$, where $x\in\theta, g\in\operatorname{St}_{(N(\mathcal Y),\psi_\mathcal A)}$ and $\varphi\in\operatorname{Hom}_{N(\mathcal Y)}(\theta,\psi_\mathcal A)$. Since $\operatorname{Hom}_{N(\mathcal Y)}(\theta,\psi_\mathcal A)$ is free of rank one as a $B$-module this action is trivial on the commutator subgroup of $\GL_n(D)$. However, the commutator subgroup of $\GL_n(D)$ agrees with the kernel of the reduced norm (since the reduced Whitehead group of $D$ is trivial, see Section 23 of \cite{draxl1983skew}) and hence this action factors through the reduced norm. We obtain that $(g\cdot\varphi)=\chi_\theta(\operatorname{Nrd}(g))\varphi$ for some character $\chi_\theta\colon E^\times\to B^\times$. The second part follows by Frobenius reciprocity.  
	\end{proof}

	\section{Classical Groups}\label{seclass}
 
	In this section we describe the classical groups that will be considered in the rest of this article. Recall that $E$ is a finite extension of $F$ and $D$ a central division algebra over $E$ with an involution $\rho\colon D\to D$. We assume that $F$ is the fixed field when $\rho$ is restricted to $E$.
 
 Let $W$ be a free right $D$-module of rank $n$ and let $\epsilon\in E$ be either $1$ or $-1$. Suppose we have a $F$-bilinear map $h\colon W\times W\to D$ that satisfies
	$$\rho(h(w,w'))=\epsilon h(w',w)$$
	and $$h(wa,w'b)=\rho(a)h(w,w')b$$
	for all $w,w'\in W$ and $a,b\in D$.
        Moreover, we will always assume that $h$ is nondegenerate, i.e.\ that $W^\perp=\{w \in W\mid h(w,w')=0\text{ for all }w'\in W\}$ is trivial. We call $h$ hermitian if $\epsilon=1$ and skew-hermitian if $\epsilon=-1$.
        
	We will consider the following cases for $D,E,W$ and $\rho$:
	\begin{enumerate}
		\item[(I1)] $D=E=F$ and $\rho$ is the identity and $\operatorname{dim}_F(W)$ is even,
		\item[(I2)] $D$ is a nonsplit quaternion algebra over $E=F$ and $\rho$ is the canonical involution of $D$,
		\item[(I3)] $D=E$ is a quadratic extension of $F$ and $\rho$ is the nontrivial element of $\Gal(E/F)$.
	\end{enumerate}
	Let $\End(W,D)$ be the $D$-linear endomorphisms of $W$. The groups we will work with are the isometry groups of a pair $(W,h)$, i.e.\ 
	$$G=\operatorname{Isom}(W,h)=\{g\in\End_D(W)^\times\mid h(gv,gw)=h(v,w)\text{ for all }v,w\in W\}.$$
        In the case (I1) with $\epsilon=-1$ and the cases (I2),(I3), note that $G$ are the $F$-points of a connected reductive group. In the case (I1) with $\epsilon=1$, we have that $G$ are the $F$-points of a disconnected reductive group.
        \begin{remark}
We exclude the odd orthogonal case as in \cite{cai2021twisted}. The results of this article should also hold in this case, however, one would need slightly different constructions, for example in the definition of the character $\psi^\bullet$ defined in Section \ref{psibullet}.
\end{remark}
	 \begin{proposition}\label{nochar} 
	     In all three cases (I1),(I2),(I3) there are no nontrivial continuous $\mathbb R_{>0}^\times$-valued characters of $G$.
	 \end{proposition}
        \begin{proof}
            For any topological group $H$ let $\prescript{0}{}{H}$ be the subgroup that is generated by all compact subgroups of $H$. Since $\mathbb R_{>0}^\times$ has no nontrivial compact subgroups, if $\prescript{0}{}{H}=H$, there are no nontrivial continuous characters from $H$ to $\mathbb R_{>0}^\times$.
            
            It is well known that if $H$ are the $F$-points of a connected classical reductive group then $\prescript{0}{}{H}=H$, unless $H\cong\operatorname{SO}(1,1)(F)\cong F^\times$. Suppose for now that $G\not\cong\operatorname{O}(1,1)(F)$. The $F$-points of the identity component of $G$ have finite index in $G$ and since $\mathbb R_{>0}^\times$ has no nontrivial finite subgroups the result follows. If $G\cong\operatorname{O}(1,1)(F)$ it is straightforward to see that any smooth character $\operatorname{O}(1,1)(F)\to \mathbb R_{>0}^\times$ is trivial.
        \end{proof}
	\subsection{The reduced adjugate matrix}
		We will need the following fact which is very likely well known, but we were unable to find a reference. Let $\mathbb D$ be a central simple division algebra over a finite field extension $F$ of $\mathbb Q_p$ and let $\mathcal O_{\mathbb D}$ be the unique maximal $\mathcal O_F$-order in $\mathbb D$. For an element $\alpha\in M_r(\mathbb D)$ we have the reduced characteristic polynomial
        $$\operatorname{Prd}_{\alpha}(X)=X^{t}-\operatorname{Trd}(\alpha)X^{t-1}+\dotsc+(-1)^t\operatorname{Nrd}(\alpha)\in F[X].$$
        Note that by the Cayley-Hamilton theorem, $\operatorname{Prd}_\alpha(\alpha)=0$.
        We can define a polynomial $Q_\alpha(X)\in F[X]$ via the relation
        $$XQ_\alpha(X)=(-1)^{t-1}\operatorname{Prd}_\alpha(X)+\operatorname{Nrd}(\alpha)$$
        and then the \emph{reduced adjugate} of $\alpha$ as
        $$\operatorname{adjrd}(\alpha)\coloneqq Q_\alpha(\alpha)\in M_r(\mathbb D).$$
        Note that $\alpha\operatorname{adjrd}(\alpha)=\operatorname{Nrd}(\alpha)1_r$ by the Cayley-Hamilton theorem. In particular, for an element $\alpha$ of $\GL_r(\mathbb D)=\{g\in M_r(\mathbb D)\mid\operatorname{Nrd}(g)\not=0\}$ we have $\operatorname{adjrd}(\alpha)=\operatorname{Nrd}(\alpha)\alpha^{-1}$. 
	\begin{proposition}\label{intredadj}
		 If $\alpha\in M_r(\mathcal O_{\mathbb D})$, then $\operatorname{adjrd}(\alpha)\in M_r(\mathcal O_{\mathbb D})$.
	\end{proposition}
    \begin{proof}
    Theorem 17.3 of \cite{reiner1975maximal} implies that $M_r(\mathcal O_{\mathbb D})$ is a maximal $\mathcal O_F$-order of $M_r(\mathbb D)$. By Theorem 10.1 of \cite{reiner1975maximal} the reduced characteristic polynomial $\operatorname{Prd}_{\alpha}(X)$ has coefficients in $\mathcal O_F$, which implies that $Q_\alpha(X)$ also has coefficients in $\mathcal O_F$. Hence $\operatorname{adjrd}(\alpha)=Q_\alpha(\alpha)$ is an element of $M_r(\mathcal O_{\mathbb D})$.
    \end{proof}

	\section{Setup of the Twisted Doubling Method}\label{setup}

	We will now develop the generalization of the doubling method due to Cai-Friedberg-Ginzburg-Kaplan, called twisted doubling, in the setting of algebraic families. We follow Sections 3 and 5 of \cite{cai2021twisted}. Let $G$ be one of the groups defined in the previous section, i.e.\ $G=\operatorname{Isom}(W,h)$, where $W$ is a free right $D$-module of rank $n$. Moreover, let $\theta$ be a smooth $B[\GL_{kn}(D)]$-module for some positive integer $k$.
 
	Set $W^{\Box,k}=W^{\oplus 2k}$ and write 
	$$W^{\Box,k}=W_{1,+}\oplus W_{2,+}\oplus\dotsc\oplus W_{k,+}\oplus W_{k,-}\oplus\dotsc\oplus W_{2,-}\oplus W_{1,-}$$
	to distinguish the copies of $W$ in $W^{\Box,k}$. We set$$(\mathbf x;\mathbf y)=(x_1,\dotsc,x_k,y_k,\dotsc,y_1)\in W^{\Box,k},$$ where $x_i\in W_{i,+}$ and $y_i\in W_{i,-}$. For $1\leq i\leq k$ and $w\in W$ let $w^{i,\pm}$ be the element $(0,\dotsc,0,w,0,\dotsc,0)\in W_{i,\pm}$ and we denote the subspace $W_{i,+}\oplus W_{i,-}$ of $W^{\Box,k}$ by $W_i^\Box$. Moreover, for $w_1,w_2\in W$ we write $(w_1,w_2)_i$ for the element $w_1^{i,+}+w_2^{i,-}\in W_i^\Box$. For subspaces $U,U'$ of $W$ let $(U,U')_i$ be the subspace $$\{(u,u')_i\mid u\in U,u'\in U'\}$$
 of $W^{\Box,k}$.

 One can define a nondegenerate form $h^{\Box,k}$ on $W^{\Box,k}$ by setting for $(\mathbf x;\mathbf y),(\mathbf x';\mathbf y')\in W^{\Box,k}$ that
	$$h^{\Box,k}((\mathbf x;\mathbf y),(\mathbf x';\mathbf y'))=\sum_{i=1}^k\left(h(x_i,x'_i)-h(y_i,y'_i)\right).$$
 Let $G^{\Box,k}$ be the group of isometries of the pair $(W^{\Box,k},h^{\Box,k})$. We can define an action of $G\times G$ on $W^{\Box,k}$ by setting for $(g_1,g_2)\in G\times G$ and $(\mathbf x;\mathbf y)\in W^{\Box,k}$ that
	$$(g_1,g_2)(\mathbf x;\mathbf y)=(g_1x_1,g_1x_2,\dotsc,g_1x_k,g_1y_k,\dotsc,g_1y_2,g_2y_1).$$
	Then $G\times G$ acts isometrically and we obtain an embedding $\iota\colon G\times G\hookrightarrow G^{\Box,k}$.
 
 Later we need to do some explicit computations so we fix a basis $\overline{w}_1,\dotsc,\overline{w}_n$ of $W$ and we set $S$ to be the Gram matrix $(h(\overline{w}_i,\overline{w}_j))_{i,j}$ of $(W,h)$. We obtain a basis $\mathbf e=(e_1,\dotsc,e_{2kn})$ of $W^{\Box,k}$ by setting
		$$e_{nl+j}=(\overline{w}_j,\overline{w}_j)_{k-l}\in W_{k-l}^\Box,$$
		where $0\leq l\leq k-1$ and $1\leq j\leq n$,
		and 
		$$e_{n(k+l)+j}=(\overline{w}_j,-\overline{w}_j)_l\in W_{l}^\Box,$$
		where $1\leq l\leq k$ and $1\leq j\leq n$. Then, with respect to this basis, the Gram matrix of $h^{\Box,k}$ has the form
				

$$\begin{pNiceArray}{ccc|ccc}[margin]
    \Block{3-3}<\large>{}&&&&&2S\\
    &&&&\Iddots&\\
    &&&2S&&\\
    \hline
    &&2S&\Block{3-3}<\Large>{}&&\\
    &\Iddots&&&&\\
    2S&&&&&
\end{pNiceArray}\in M_{2kn}(D).$$
We will write $J$ for the matrix
$$\begin{pNiceArray}{ccc}
&&2S\\
&\Iddots&\\
2S&&
\end{pNiceArray}\in M_{kn}(D).$$

 \subsection{Parabolic Subgroups}
 For a flag of totally isotropic subspaces $\mathcal F$ of $W^{\Box,k}$ we denote the parabolic subgroup of $G^{\Box,k}$ that stabilizes $\mathcal F$ by $P(\mathcal F)$. Let $N(\mathcal F)$ be the unipotent radical of $P(\mathcal F)$ and we will denote a Levi subgroup of $P(\mathcal F)$ by $M(\mathcal F)$.
 
 Let $W_i^\bigtriangleup$ be the subspace $\{(w,w)_i\in W_i^\Box\mid w\in W\}$ and $W_i^\bigtriangledown$ the subspace $\{(w,-w)_i\mid w\in W\}$ of $W_i^\Box$. We write $W^{\bigtriangleup,k}$ for 
	$$W_1^\bigtriangleup\oplus W_2^\bigtriangleup\oplus\dotsc\oplus W_k^\bigtriangleup$$
	and analogously $W^{\bigtriangledown,k}$ for
	$$W_1^\bigtriangledown\oplus W_2^\bigtriangledown\oplus\dotsc\oplus W_k^\bigtriangledown,$$
	which are both maximal totally isotropic subspaces of $W^{\Box,k}$.
       Of importance to us is  the stabiliser of the flag $0\subset W^{\bigtriangleup,k}\subset W^{\Box,k}$ and we will write $P=MN$ for this parabolic subgroup of $G^{\Box,k}$. 
       We choose $M$ to consist of the elements of $G^{\Box,k}$ that stabilize both $W^{\bigtriangleup,k}$ and $W^{\bigtriangledown,k}$. Note that $x\mapsto x|_{W^{\bigtriangleup,k}}$ yields an isomorphism between $M$ and $\GL_{kn}(D)$.
       
       Let the character $\nu\colon P\to E^\times$ be given by $\nu(x)=\operatorname{Nrd}(x|_{W^{\bigtriangleup,k}})$. Moreover, we define a character $\nu_X$ on $P$, valued in $R[X^{\pm1}]^\times$, by setting
	$$\nu_X(g)=X^{\operatorname{val}_E(\nu(g))},$$
	for $g\in P$.
	For the smooth $B[\GL_{kn}(D)]$-module $\theta$ we obtain a smooth $B[X^{\pm}][G^{\Box,k}]$-module by setting $$I(X,\theta)\coloneqq i_{P}^{G^{\Box,k}}(\theta\otimes_R\nu_X).$$
    \begin{remark}\label{evatqs} Suppose that $R=B=\mathbb C$. We discuss the relationship between the space $I(X,\theta)$ above and the spaces $I(q_E^{-s},\theta)$ (where $s\in\mathbb C$) that are used by Cai-Friedberg-Ginzburg-Kaplan. Note that for any complex number $s\in\mathbb C$ we have a complex character $\nu_s\colon P\to\mathbb C^\times$, by setting $$\nu_s(g)=q_E^{-s\operatorname{val}_E(\nu(g))}$$ 
    for $g\in P$ and note that $I(q_E^{-s},\theta)=i_P^{G^{\Box,k}}(\theta\otimes_{\mathbb C}\nu_s)$. The $\mathbb C[X^{\pm1}]$-module $$\mathbb C_s\coloneqq\mathbb C[X^{\pm1}]/(X-q_E^{-s})$$ has dimension one as a $\mathbb C$-vector space. Clearly, we have a $\mathbb C[P]$-isomorphism $$\nu_X\otimes_{\mathbb C[X^{\pm1}]}\mathbb C_s\cong\nu_s.$$
    Since parabolic induction commutes with base change, we obtain that 
    $$I(X,\theta)\otimes_{\mathbb C[X^{\pm1}]}\mathbb C_s\cong I(q_E^{-s},\theta)$$
    as $\mathbb C[G^{\Box,k}]$-modules.

    \end{remark}

 \subsection{The Character $\psi^\bullet$}\label{psibullet}
 For $i=1,\dotsc,k-1$ we define totally isotropic subspaces $$\mathcal F_{i}^\bullet=\bigoplus_{j=1}^iW_{k+1-j}^\bigtriangledown\subseteq W^{\Box,k}$$
	and consider the flag $$\mathcal F^\bullet\colon 0\subseteq\mathcal F^\bullet_1\subseteq \mathcal F^\bullet_2\subseteq\dotsc\subseteq\mathcal F^\bullet_{k-1}.$$
	We denote the stabiliser of this flag by $P^\bullet$ and the unipotent radical of this parabolic subgroup by $N^\bullet$. Note that any element of $P^\bullet$ automatically stabilizes the flag
	$$0\subseteq\mathcal F^\bullet_1\subseteq \mathcal F^\bullet_2\subseteq\dotsc\subseteq\mathcal F^\bullet_{k-1}\subseteq(\mathcal F^\bullet_{k-1})^\perp\subseteq\dotsc\subseteq (\mathcal F^\bullet_1)^\perp\subseteq W^{\Box,k}.$$
 Let $P(W^{\bigtriangledown,k})$ be the stabiliser of the flag $0\subset W^{\bigtriangledown,k}\subset W^{\Box,k}$, with unipotent radical $N(W^{\bigtriangledown,k})$. Note that $P(W^{\bigtriangledown,k})$ is an opposite parabolic of $P$. Let $N^\circ=N^\bullet\cap N(W^{\bigtriangledown,k})$.
 
In the following we set $\mathcal F_k^\bullet=(\mathcal F_{k-1}^\bullet)^\perp$. For any $i=1,\dotsc, k-1$ and $u\in N^\bullet$, the endomorphism $u-\id\in\End_D(W^{\Box,k})$ maps $\mathcal F_{i+1}^\bullet$ to $\mathcal F_{i}^\bullet$ and hence gives rise to an element of $\Hom_D(\mathcal F_{i+1}^\bullet/\mathcal F_{i}^\bullet,\mathcal F_{i}^\bullet/\mathcal F_{i-1}^\bullet)$, which we denote by $u_i$. By Section 2.2 of \cite{cai2021twisted}, to define a character of $N^\bullet$, one has to choose elements
	$$A_i\in\Hom_D(\mathcal F^\bullet_i/\mathcal F^\bullet_{i-1},\mathcal F^\bullet_{i+1}/\mathcal F^\bullet_i)\cong\operatorname{End}_D(W_{k+1-i}^\bigtriangledown,W_{k-i}^\bigtriangledown)$$
	for $i=1,\dotsc,k-2$ and 
	$$A_{k-1}\in\Hom_D(\mathcal F^\bullet_{k-1}/\mathcal F^\bullet_{k-2},\mathcal (\mathcal F^\bullet_{k-1})^\perp/\mathcal F^\bullet_{k-1})\cong\operatorname{Hom}_D(W_2^\bigtriangledown,W_1^\Box).$$
	If $1\leq i\leq k-2$ let $A_i$ be the map that sends $(w,-w)_{k+1-i}$ to $(w,-w)_{k-i}$ and let $A_{k-1}$ be the map that sends $(w,-w)_2$ to $(2w,0)_1$. These choices of $A_i$ yield a character $\psi^\bullet\colon N^\bullet\to\mathbb Z[1/p,\mu_{p^\infty}]^\times$ by setting for $u\in N^\bullet$ that
 $$\psi^\bullet(u)=\psi\left(\sum_{i=1}^{k-1}\operatorname{Trd}(u_i\circ A_i)\right).$$
Note that $\iota(G\times G)$ lies in the stabiliser of $(N^\bullet,\psi^\bullet)$, i.e.\ $\iota(G\times G)$ normalizes $N^\bullet$ and $\psi^{\bullet}(xux^{-1})=\psi^\bullet(u)$ for all $x\in\iota(G\times G),u\in N^\bullet$.

	\section{The Twisted Doubling Zeta Integrals}\label{secintegrals}
	In the setting over the complex numbers, as in \cite{cai2021twisted},\cite{cai2019doubling},\cite{caidoubling} and \cite{gourevitch2022multiplicity}, the twisted doubling zeta integrals are functions in a complex variable $s$, defined via an integral that in general only converges if the real part of $s$ is large enough. Since in our algebraic setting we have no notion of convergence we will remedy this by showing that the zeta integrals in question can be interpreted as a Laurent series. To do so we introduce the notion of coefficient compact functions.
	\subsection{Coefficient Compact Functions}
	\begin{definition} Let $\mathbb G$ be a locally profinite group and assume there exists a nonzero $A$-valued Haar measure on $\mathbb G$ that is normalized on an open compact subgroup of $\mathbb G$.
		\begin{enumerate}
			\item For any $i\in\mathbb Z$ and any locally constant function $f\colon\mathbb G\to A[X^{\pm1}]$ let $\mathcal S_i(f)$ be the support of the $A$-valued function $\operatorname{
				Coeff}_{X^i}(f)$, i.e.
			$$\mathcal S_i(f)\coloneqq\{g\in\mathbb G\mid\operatorname{Coeff}_{X^i}f(g)\not=0\}.$$
			\item Let $\mathcal{CC}_A(\mathbb G)$ be the space of $A$-valued \emph{coefficient compact} functions on $\mathbb G$, i.e.\ the $A[X^{\pm1}]$-module consisting of locally constant functions $f\colon\mathbb G\to A[X^{\pm1}]$ that satisfy that
			\begin{itemize}
				\item there is an integer $M_f$ such that $\mathcal S_i(f)=\emptyset$ for all $i\leq M_f$,
				\item for any $i\in\mathbb Z$ the set $\mathcal S_i(f)$ is compact.
			\end{itemize}
			
			\item For any  function $f\in\mathcal{CC}_A(\mathbb G)$ and integer $i\in\mathbb Z$ the locally constant function $$\operatorname{Coeff}_{X^{i}}(f(-))\colon\mathbb G\to A$$ has compact support and we can define
			$$a_i(f)=\int_{\mathbb G}\operatorname{Coeff}_{X^{i}}(f(g))dg=\operatorname{Coeff}_{X^i}\left(\int_{\mathcal S_i(f)}f(g)dg\right)$$ and a Laurent Series
			$$\int_{\mathbb G}f(g)dg\coloneqq\sum_{i=-\infty}^\infty a_i(f)X^i\in A[[X]][X^{\pm1}].$$
		\end{enumerate}
		
	\end{definition}
Note that the properties of the Haar measure extend readily to the above notion, i.e.\ associating a Laurent series to an element of $\mathcal{CC}_A(\mathbb G)$ is $A[X^{\pm1}]$-linear.

\subsection{Twisted Doubling Zeta Integrals}	
	
	Let $G$ be any of the classical groups defined in Section \ref{secdegwhit} and $(\pi,V)$ a smooth $A[G]$-module. Moreover, assume that $\theta$ is a $B[\GL_{kn}(D)]$-module of type $(k,n)$. We fix a pair $(N(\mathcal Y), \psi_\mathcal A)$ that lies in the orbit $\lambda_{k,n}$ and a generator $$\operatorname{Wh}_{N(\mathcal Y),\psi_{\mathcal A}}\in\Hom_{\GL_{kn}(D)}(\theta,\Ind_{N(\mathcal Y)}^{\GL_{kn}(D)}(\psi_\mathcal A)).$$ In particular by composing with $\operatorname{Wh}_{N(\mathcal Y),\psi_{\mathcal A}}$ we can view $I(X,\theta)=i_{P}^{G^{\Box,k}}(\theta\otimes\nu_X)$ as functions on $G^{\Box,k}$ with values in $\operatorname{Wh}_{N(\mathcal Y),\psi_{\mathcal A}}(\theta)\otimes_B B[X^{\pm1}]$. For an element in $\operatorname{Wh}_{N(\mathcal Y),\psi_{\mathcal A}}(\theta)$, we will write $\operatorname{ev}_1$ for the evaluation on the identity of $\GL_{kn}(D)$. This clearly gives rise to a $B[X^{\pm1}]$-linear map $$\operatorname{ev}_1\colon \operatorname{Wh}_{N(\mathcal Y),\psi_{\mathcal A}}(\theta)\otimes_B B[X^{\pm1}]\to B[X^{\pm1}].$$ For the remainder of this section we assume that the residue characteristic of $F$ is odd. To define the twisted doubling zeta integral if the residue characteristic of $F$ is even we refer to Remark \ref{defeven}. 
 
 Let $\mathcal L$ be the free $\mathcal O_D$-module in $W^{\Box,k}$ that is spanned by the basis $\mathbf e$ defined in Section \ref{setup}. We set $\mathcal L'$ to be the image of $\mathcal L$ under 
            $$\begin{pmatrix}
                1_{kn}&\\&J^{-1}
            \end{pmatrix}.$$
            Let $K$ be the compact open subgroup of $G^{\Box,k}$ that stabilizes $\mathcal L'$ and note that by Proposition \ref{iwasawadec} we have that $P\cdot K=G^{\Box,k}$. Since $\operatorname{val}_E\circ\nu$ is trivial on $P\cap K$ we can use this decomposition to extend $\operatorname{val}_E\circ\nu$ to a right $K$-invariant function on $G^{\Box,k}$ for which we have the following result.

	\begin{proposition}\label{compactness} Suppose that the residue characteristic of $F$ is odd. 
		\begin{enumerate}
			\item The integer-valued map $(g,u)\mapsto\operatorname{val}_E(\nu(u\iota(g,1)))$ on $G\times N^\circ$ is locally constant, has compact fibers and image in the non-negative integers.
			\item For any $f\in I(X,\theta)$ the function $(g,u)\mapsto\operatorname{ev}_1f(u\iota(g,1))$ on $G\times N^\circ$ is coefficient compact, i.e. lies in $\mathcal{CC}_B(G\times N^\circ)$.
		\end{enumerate}
		
	\end{proposition}
	
	\begin{proof} ad (1):
 For $k=1$ this is Lemma 8.4 in \cite{kakuhama2019local}. We write $\mathcal O_D$ for the ring of integers of $D$. Recall that we fixed a basis $w_1,\dotsc,w_n$ of $W$, where, by scaling, we may assume that the Gram matrix $S=(h(w_i,w_j))_{i,j}$ satisfies that $(2S)^{-1}\in M_n(\mathcal O_D)$. 
    With respect to the basis $\mathbf e=(e_1,\dotsc, e_{2kn})$ of $W^{\Box,k}$ which was chosen in Section \ref{setup}, elements of $M$ are of the form
		$$\begin{pmatrix}
			a&\\&J^{-1}{a^*}^{-1}J
		\end{pmatrix},$$
		where $a$ is any element of $\GL_{kn}(D)$ and matrices in $N$ have the form
		$$\begin{pmatrix}
			1_{kn}&X\\&1_{kn}
		\end{pmatrix},$$
		where $X\in M_{kn}(D)$ satisfies $JX+X^*J=0$. Note that $K$ equals
		$$\left\{g\in G^{\Box,k}\mid \begin{pmatrix}
			1_{kn}&\\&J 
		\end{pmatrix}g\begin{pmatrix}
			1_{kn}&\\&J^{-1} 
		\end{pmatrix}\in\GL_{2kn}(\mathcal O_D)\right\}.$$
  
  By Proposition \ref{iwasawadec} we have that $P\cdot K=G^{\Box,k}$. Hence for $g\in G,u\in N^\circ$ we can find $a\in\
GL_{kn}(D), X\in M_{kn}(D)$ and $k_1\in K$ such that 
		\begin{equation}\label{eqlongproof1}u\iota(g,1)=\begin{pmatrix}
				a&\\&J^{-1}{a^*}^{-1}J
			\end{pmatrix}\begin{pmatrix}
				1_{kn}&X\\&1_{kn}
			\end{pmatrix}k_1.\end{equation}
It is then straightforward to see that $\operatorname{val}_E(\nu(u\iota(g,1))=\operatorname{val}_E(\operatorname{Nrd}(a))$. In the following we view endomorphisms of $W$ as endomorphisms of $W_i^\bigtriangleup$ (respectively $W_i^\bigtriangledown$) via the identification $(w,w)_i\mapsto w$ (respectively $(w,-w)_i\mapsto w$). With respect to $\mathbf e$ we have that
		
		
		$$
            \iota(g,1)=\begin{pNiceArray}{cccc|cccc}[margin]
                g&&&&&&&\\
                &\Ddots&&&&&&\\
                &&g&&&&&\\
                &&&\frac{g+1}{2}&\frac{g-1}{2}&&&\\
                \hline
                &&&\frac{g-1}{2}&\frac{g+1}{2}&&&\\
                &&&&&g&&\\
               &&&&&&\Ddots&\\
                &&&&&&&g
            \end{pNiceArray}
            $$
        and 
	$$u=\begin{pNiceArray}{cc|ccc}[margin]
	    \Block{2-2}<\Large>{1_{kn}}&&\Block{2-3}<\Large>{0_{kn}}&&\\
            &&&&\\
            \hline
            y_1&0_n& \Block{2-3}<\Large>{1_{kn}}&&\\
            y_3&y_2&&&
	\end{pNiceArray},$$
        where $y_1\in M_{n\times (k-1)n}(D),y_2\in M_{(k-1)n\times n}(D)$ and $y_3\in M_{(k-1)n}(D)$ and they satisfy certain conditions that ensure that $u$ is an element of $G^{\Box,k}$. Note that we may view
        $y_1$ as an element of $\Hom_D(W_2^\bigtriangleup\oplus\dotsc\oplus W_k^\bigtriangleup,W_1^\bigtriangledown)$ and similarly
        \begin{align*}
            y_2&\in\Hom_D(W_1^\bigtriangleup, W_2^\bigtriangledown\oplus\dotsc\oplus W_k^\bigtriangledown),\\
            y_3&\in\Hom_D(W_2^\bigtriangleup\oplus\dotsc\oplus W_k^\bigtriangleup, W_2^\bigtriangledown\oplus\dotsc\oplus W_k^\bigtriangledown).
        \end{align*}

		We set
		$$\phi=\begin{pmatrix}
			1_{kn}&\\&J
		\end{pmatrix}u\iota(g,1)\begin{pmatrix}
			1_{kn}&\\&J^{-1}
		\end{pmatrix}.$$
  Let $J'\in M_{(k-1)n}(D)$ be the matrix such that $$J=\begin{pmatrix}
      &J'\\
      2S&
  \end{pmatrix}.$$
We can view $J'$ as the bijective linear map $W_2^\bigtriangledown\oplus\dotsc\oplus W_k^\bigtriangledown\to W_1^\bigtriangledown\oplus\dotsc\oplus W_{k-1}^\bigtriangledown$ that sends $(w,-w)_i\mapsto (2S)^{-1}(w,-w)_{k-i+1}$ for $2\leq i\leq k$. By using the explicit matrix representatives for $\iota(g,1)$ and $u$ we can compute 
  \begin{equation}\label{explcomp}
      \phi=\begin{pNiceArray}{cc|cc}[margin]
\Block{2-2}{*} & &\Block{2-2}{*} &  \\
 & &&  \\
 \hline
J'y_3g& J'y_2(g+1)/2&  J'gJ'^{-1}& J'y_2(g-1)/2(2S)^{-1}\\

 (2S)y_1g& S(g-1)&0 &  S(g+1)/2S^{-1}
 \end{pNiceArray}.
  \end{equation}
Equation (\ref{eqlongproof1}) implies that
		\begin{equation}\label{eqlongproof2}\phi=\begin{pmatrix}
				a&\\
				&{a^*}^{-1}
			\end{pmatrix}\begin{pmatrix}
				1_{kn}&XJ^{-1}\\&1_{kn}
			\end{pmatrix}k',
		\end{equation}
		where $k'\in\GL_{2kn}(\mathcal O_D)$, i.e.\ $k'$ stabilizes $\mathcal L$.
  
		We will consider the image of $\mathcal L$ under $\phi$ composed with the projection map $$\Psi\colon W^{\Box,k}\twoheadrightarrow W^{\bigtriangledown,k}.$$
  Moreover, let $\Psi'\colon W^{\bigtriangledown,k}\twoheadrightarrow W_1^\bigtriangledown\oplus\dotsc\oplus W_{k-1}^\bigtriangledown$ and $\Psi_k\colon W^{\bigtriangledown,k}\twoheadrightarrow W_k^\bigtriangledown$ be the canonical projections. 
For $1\leq i\leq k$ we define $\mathcal O_D$-lattices
          \begin{align*}
            \mathcal L_{i}^\bigtriangleup&=\langle e_{(k-i)n+1},\dotsc,e_{(k-i+1)n}\rangle\subseteq W_{i}^\bigtriangleup,\\
            \mathcal L_i^\bigtriangledown&=\langle e_{(k+i-1)n+1},\dotsc,e_{(k+i)n}\rangle\subseteq W_i^\bigtriangledown.
        \end{align*}
 We set $\mathcal L^\bigtriangledown=\mathcal L_1^\bigtriangledown\oplus\dotsc\oplus\mathcal L_{k}^\bigtriangledown$ and clearly have that $\Psi(\mathcal L)=\mathcal L^\bigtriangledown$. Since $k'(\mathcal L)=\mathcal L$ we obtain by Equation (\ref{eqlongproof2}) that \begin{equation}\label{easylcomp}
     (\Psi\circ\phi)(\mathcal L)={a^*}^{-1}(\mathcal L^\bigtriangledown).
 \end{equation}
Equation (\ref{explcomp}) implies for $1\leq i\leq k-1$ and $w\in\mathcal L_{i}^\bigtriangledown$ that
$$(\Psi\circ\phi)(w)=SgS^{-1}(w)={g^*}^{-1}(w)\in\mathcal L_i^\bigtriangledown.$$
For $(k-1)n+1\leq j\leq kn$ we have that $e_j\in\mathcal L_1^\bigtriangleup$, and again by Equation (\ref{explcomp}) we see that  
$$(\Psi\circ\phi)(e_j)=(J'y_2(g+1)/2)(e_j)+S(g-1)(e_{j+kn}).$$
Moreover, $e_{j+kn}\in\mathcal L_{k}^\bigtriangledown$ and one obtains that 
$$(\Psi\circ\phi)(e_{j+kn})=J'y_2(g-1)/2(2S)^{-1}(e_j)+S(g+1)/2S^{-1}e_{j+kn}$$
Since we assume that $(2S)^{-1}\in M_{n}(\mathcal O_D)$ we have for $w\in\mathcal L_k^\bigtriangleup$ that $(2S)^{-1}w\in\mathcal L_k^\bigtriangleup$.
In particular this implies that
\begin{equation}\label{multim}
    (\Psi\circ\phi)((2S)^{-1}e_{j}+e_{j+kn})=J'y_2g(2S)^{-1}(e_j)+SgS^{-1}e_{j+kn}\in (\Psi\circ\phi)(\mathcal L).
\end{equation}
Let
$$m=\begin{pNiceArray}{ccc|c}[margin]
{g^*}^{-1} & & & \Block{3-1}{J'y_2(2S)^{-1}t} \\
 & \Ddots&&  \\
& &{g^*}^{-1} & \\
\hline
 & &  & {g^*}^{-1}
 \end{pNiceArray}$$ 
 and note that by Proposition \ref{nochar} we have that $\operatorname{val}_E(\operatorname{Nrd}(m))=-k\operatorname{val}_E(\operatorname{Nrd}({g^*}))=0$. Now Equations (\ref{easylcomp}) and (\ref{multim}) imply that
$${a^*}^{-1}(\mathcal L^\bigtriangledown)\supseteq m(\mathcal L^\bigtriangledown),$$
and hence that $a^*m\in M_{kn}(\mathcal O_D)$ and in particular $\operatorname{val}_E(\operatorname{Nrd}(a^*m))\geq 0$. We obtain
$$\operatorname{val}_E(\nu(u\iota(g,1)))=\operatorname{val}_E(\operatorname{Nrd}(a))=\operatorname{val}_E(\operatorname{Nrd}(a^*m))\geq 0.$$

Suppose now that $\operatorname{val}_E(\nu(u\iota(g,1)))= j$, which implies that
$$\varpi^{-j}(\mathcal L^\bigtriangledown)\supseteq\operatorname{adjrd}(m^{-1}{a^*}^{-1})(\mathcal L^\bigtriangledown)$$
and hence $\varpi^{j}\operatorname{adjrd}(m^{-1}{a^*}^{-1})\in M_n(\mathcal O_D)$. By Proposition \ref{intredadj} we obtain that 
$$\operatorname{adjrd}(\varpi^{j}\operatorname{adjrd}(m^{-1}{a^*}^{-1}))=\varpi^jm^{-1}{a^*}^{-1}\in M_{kn}(\mathcal O_D)$$
and hence $m^{-1}{a^*}^{-1}\in M_{kn}(\varpi^{-j}\mathcal O_D)$. By Equation (\ref{explcomp}) we have for $w\in\mathcal L_k^\bigtriangledown$ that $$(1-g^{*})(S^{-1}w)=(\Psi_k\circ m^{-1}\circ\Psi\circ\phi)(S^{-1}w)=\Psi_k(m^{-1}{a^*}^{-1})(S^{-1}w)\in\varpi^{-j}(\mathcal L_k^{\bigtriangledown}).$$
Hence we obtain that $(1-g^*)S^{-1}\in M_{n}(\varpi^{-j}\mathcal O_D)$. Note that $G$ is a closed subset of $M_n(D)$ and hence $G\cap M_n(\varpi^\alpha\mathcal O_D)$ is a compact subset of $G$ for any $\alpha\in\mathbb Z$. .

Let $C_j=\{(g,u)\in G\times N^\circ\mid\operatorname{val}_E(\nu(u\iota(g,1))\leq j)\}$, which by continuity is open and closed. By what we just showed, the image of the projection $C_j\to G$ is compact and hence
there are integers $l$ and $l'$ such that $g\mathcal L_i^\bigtriangledown\supseteq\varpi^l\mathcal L_i^\bigtriangledown$ and $g^*\mathcal L_i^\bigtriangledown\subseteq\varpi^{l'}\mathcal L_i^\bigtriangledown$ for any $(g,u)\in C_j$ and any $1\leq i\leq k$.
For $w\in\mathcal L_i^\bigtriangleup$, where $2\leq i\leq k$, we have by Equation (\ref{explcomp}) that 
$$(\Psi_k\circ m^{-1}\circ\Psi\circ\phi)(w)=({g^*}^{-1}t2Sy_1g)(w)\in\varpi^{-j}\mathcal (\mathcal L^\bigtriangledown_k),$$
where we set $t\colon W_1^\bigtriangledown\to W_k^\bigtriangledown$ to be the map defined by $t((x,-x)_1)=(x,-x)_k$. This implies that $2Sy_1(w)\in\varpi^{-l-l'-j}(\mathcal L_1^\bigtriangledown)$ and hence $y_1$ lies in the compact set $(2S)^{-1}M_{n\times (k-1)n}(\varpi^{-l-l'-j}\mathcal O_D)$. Since $u\in G^{\Box,k}$ we have that $J'y_2+y_1^*(2S)=0$, which shows that $y_2$ is also contained in a compact set. Now Equation (\ref{explcomp}) implies for $w\in\mathcal L_2^\bigtriangleup\oplus\dotsc\oplus\mathcal L_k^\bigtriangleup$ that
$$(\Psi'\circ m^{-1}\circ\Psi\circ\phi)(w)=J'y_3g(w)+(J'y_2(2S)^{-1}t'(2S)y_1g)(w)\in\varpi^{-j}(\mathcal L_1^\bigtriangledown\oplus\dotsc\oplus\mathcal L_{k-1}^\bigtriangledown),$$
where we set $t'\colon W_1^\bigtriangledown\to W_1^\bigtriangleup$ such that $t'((x,-x)_1)=(x,x)_1$. By what we already showed there is some $\alpha\in\mathbb Z$ such that for all $w\in\mathcal L_2^\bigtriangleup\oplus\dotsc\oplus\mathcal L_k^\bigtriangleup$ we have that $$(J'y_2(2S)^{-1}t'(2S)y_1g)(w)\in\varpi^{-\alpha}(\mathcal L_1^\bigtriangledown\oplus\dotsc\oplus\mathcal L_{k-1}^\bigtriangledown).$$ 
Hence $J'y_3g(w)\in\varpi^{-\max{(\alpha,j)}}(\mathcal L_1^\bigtriangledown\oplus\dotsc\oplus\mathcal L_{k-1}^\bigtriangledown)$, which implies that 
$$y_3\in J'^{-1}M_{(k-1)n}(\varpi^{-l-\max{(\alpha,j)}}\mathcal O_D)$$ and the result follows.

		ad (2):
		We will write $\xi\colon G\times N^\circ\to B[X^{\pm1}]$ for the function $(g,u)\mapsto\operatorname{ev}_1f(u\iota(g,1))$. For $g\in G,u\in N^\circ$ let $f(u\iota(g,1))=sk_1$ where $s\in P$ and $k_1\in K$. We have \begin{equation}\label{eq:parcomp}
			f(u\iota(g,1))=\delta_P(s)^{1/2}X^{\operatorname{val}_E(\nu(u\iota(g,1)))}\theta(s)f(k_1).
		\end{equation} Since $f$ is locally constant, its restriction to $K$ admits only finitely many different values in $\operatorname{Wh}_{N(\mathcal Y),\psi_{\mathcal A}}(\theta)\otimes B[X^{\pm1}]$. However, since by part (1) the integer $\operatorname{val}_E(\nu(u\iota(g,1)))$ is nonnegative, we immediately obtain that there is an $m\in\mathbb Z$ such that for any integer $i\leq m$ we have that $\operatorname{Coeff}_{X^i}(\xi(g,u))=0$ for all $g\in G, u\in N^\circ$.
  
		By part (1), for any $j\in\mathbb Z$, the set
$$C_j=\{(g,u)\in G\times N^\circ\mid\operatorname{val}_E(\nu(u\iota(g,1)))\leq j \},$$ 
is compact. Now for any $i\in\mathbb Z$, by Equation (\ref{eq:parcomp}), if $(g,u)\notin C_{i-m}$ we obtain that
		$$\operatorname{Coeff}_{X_i}\xi(g,u)=0.$$
		Hence $$\{(g,u)\in G\times N^\circ\mid\operatorname{Coeff}_{X^i}\xi(g,u)\not=0\}$$
  is contained in $C_{i-m}$ and as a closed subset of a compact set, compact.
	\end{proof}
	We just showed that for any $f\in I(X,\theta)$, the function $(g,u)\mapsto \operatorname{ev_1}f(u\iota(g,1))$ lies in $\mathcal{CC}_B(G\times N^\circ)$. Then clearly for any $v\in V,\lambda\in\widetilde{V}$ the function 
	$$(g,u)\mapsto \lambda(\pi(g)v)\otimes\operatorname{ev_1}f(u\iota(g,1))\psi^\bullet(u),$$
	is an element of $\mathcal{CC}_{A\otimes B}(G\times N^\circ)$. Hence we can consider for any $i\in\mathbb Z$ the integral
	$$a_i(v,\lambda,f)=\int_{G\times N^\circ}\operatorname{Coeff}_{X^i}[\lambda(\pi(g)v)\otimes\operatorname{ev_1}f(u\iota(g,1))\psi^\bullet(u)]d(g,u),$$
	which yields an element of $A\otimes B$ and satisfies that $a_j(v,\lambda,f)=0$ for $j$ small enough. This allows us to define a Laurent series
	$$Z(X,v,\lambda,f)\coloneqq\sum_{i=-\infty}^\infty a_i(v,\lambda,f)X^i\in (A\otimes B)[[X]][X^{-1}].$$
	\begin{definition}
		For $v\in V,\lambda\in\widetilde{V}$ and $f\in I(X,\theta)$ we call the Laurent series $Z(X,v,\lambda,f)\in (A\otimes B)[[X]][X^{-1}]$ the \emph{twisted doubling zeta integral} associated to $v,\lambda$ and $f$.
	\end{definition}
	It is clear that the twisted doubling zeta integral is $(A\otimes B)[X^{\pm1}]$-linear and hence gives rise to an element in 
	$$\Hom_{(A\otimes B)[X^{\pm 1}]}\left((V\otimes_{A}\widetilde{V})\otimes_R I(X,\theta),(A\otimes B)[[X]][X^{-1}])\right).$$
	
	As in the classical doubling method the twisted doubling zeta integral behaves nicely under certain group actions on the input data.
	
	\begin{proposition}\label{inttransf} Let $v\in V,\lambda\in\widetilde{V}$ and $f\in I(X,\theta)$. 
		\begin{enumerate}
			\item For all $(g_1,g_2)\in G\times G$ we have that
			$$Z(X,\pi(g_1)v,\widetilde{\pi}(g_2)\lambda,\iota(g_1,g_2)f)=\chi_\theta(\operatorname{Nrd}(g_2))Z(X,v,\lambda,f),$$
			where we view $g_2$ as an element of $\End_{D}(W)$ and define $\operatorname{Nrd}(g_2)$ that way.
			\item 	The equation
			$$Z(X,v,\lambda,x\cdot f)=\psi^\bullet(x)^{-1}Z(X,v,\lambda,f)$$ holds for all $x\in N^\bullet$.
		\end{enumerate}
	\end{proposition}
	\begin{proof} ad (1): We show that the Laurent series on each side have the same coefficient. For any $i\in\mathbb Z$ let 
		$a_i$ be the $i$-th coefficient of $Z(X,\pi(g_1)v,\widetilde{\pi}(g_2)\lambda,\iota(g_1,g_2)f)$. By definition,

			$$a_i=\int_{G\times N^\circ}\operatorname{Coeff}_{X^i}[\lambda(\pi(g_2^{-1}gg_1)v)\otimes\operatorname{ev}_1f(u\iota(gg_1,g_2))\psi^\bullet(u)]d(g,u)$$
  and hence $a_i$ also equals
  $$\int_{G\times N^\circ}\operatorname{Coeff}_{X^i}[\lambda(\pi(g_2^{-1}gg_1)v)\otimes\operatorname{ev}_1f(\iota(g_2,g_2)\iota(g_2^{-1},g_2^{-1})u\iota(g_2,g_2)\iota(g_2^{-1}gg_1,1))\psi^\bullet(u)]d(g,u).$$
		Since $G\times G$ normalizes $N^\circ$ with trivial associated unimodular character (see Lemma 5.2 of \cite{cai2021twisted}) and $G\times G$ is in the stabiliser of $\psi^\bullet$ we obtain
		$$
		a_i=\int_{G\times N^\circ}\operatorname{Coeff}_{X^i}[\lambda\left(\pi(g_2^{-1}gg_1)v\right)\otimes\operatorname{ev}_1f\left(\iota(g_2,g_2)u\iota(g_2^{-1}gg_1,1)\right)\psi^\bullet(u)]d(g,u).$$
		The unimodularity of $G$ implies that 
		$$
		a_i=\int_{G\times N^\circ}\operatorname{Coeff}_{X^i}[\lambda\left(\pi(g)v\right)\otimes\operatorname{ev}_1f\left(\iota(g_2,g_2)u\iota(g,1)\right)\psi^\bullet(u)]d(g,u).$$
		Now $\iota(g_2,g_2)$ is an element of $P$, but moreover lies in $\operatorname{St}_{N(\mathcal Y),\psi_{\mathcal A}}$ and hence by Lemma \ref{invstab} acts via the character $\chi_\theta\circ\operatorname{Nrd}$. We obtain that
		$$
		a_i=\chi_\theta(\operatorname{Nrd}(g_2))\int_{G\times N^\circ}\operatorname{Coeff}_{X^i}[\lambda\left(\pi(g)v\right)\otimes\operatorname{ev}_1f\left(u\iota(g,1)\right)\psi^\bullet(u)]d(g,u).$$
		Since $$\operatorname{Coeff}_{X^i}[Z(X,v,\lambda,f)]=\int_{G\times N^\circ}\operatorname{Coeff}_{X^i}[\lambda\left(\pi(g)v\right)\otimes\operatorname{ev}_1f\left(u\iota(g,1)\right)\psi^\bullet(u)]d(g,u),$$
		we see that 
		$$Z(X,\pi(g_1)v,\widetilde{\pi}(g_2)\lambda,\iota(g_1,g_2)f)=\chi_\theta(\operatorname{Nrd}(g_2))Z(X,v,\lambda,f).$$
		
		ad (2): Let $b_i$ be the coefficient of $Z(X,v,\lambda,x\cdot f)$. Then by definition we have
		$$b_i=\int_{G\times N^\circ}\operatorname{Coeff}_{X^i}\left[\lambda(\pi(g)v)\otimes\operatorname{ev_1}(x\cdot f)(u\iota(g,1))\psi^\bullet(u)\right]d(g,u).$$
		Fubini's theorem yields that
		\begin{align*}
			b_i&=\int_{G}\lambda(\pi(g)v)\otimes\left(\int_{N^\circ}\operatorname{Coeff}_{X^i}[\operatorname{ev}_1f(u\iota(g,1)x)\psi^\bullet(u)]du\right)dg\\
			&=\int_{G}\lambda(\pi(g)v)\otimes\left(\int_{N^\circ}\operatorname{Coeff}_{X^i}[\operatorname{ev}_1f(u\iota(g,1)x\iota(g^{-1},1)\iota(g,1))\psi^\bullet(u)]du\right)dg.
		\end{align*}
		Note that $N^\bullet=(M\cap N^\bullet)N^\circ$ and since $\iota(G\times G)$ normalizes $N^\bullet$ we can write $\iota(g,1)x\iota(g^{-1},1)=x_1^gx_0^g$ where $x_1^g\in M
		\cap N^\bullet$ and $x_0^g\in N^\circ$. Now as $M\cap N^\bullet$ normalizes $N^\circ$ and as the associated modulus character is trivial we see that 
		$$b_i=\int_{G}\lambda(\pi(g)v)\otimes\left(\int_{N^\circ}\operatorname{Coeff}_{X^i}[\operatorname{ev}_1f(x_1^gux_0^g\iota(g,1))\psi^\bullet(u)]du\right)dg.$$
		After the variable change $u\mapsto u{x_0^g}^{-1}$ we obtain
		$$b_i=\int_{G}\lambda(\pi(g)v)\otimes\left(\int_{N^\circ}\operatorname{Coeff}_{X^i}[\operatorname{ev}_1f(x_1^gu\iota(g,1))\psi^\bullet(u{n_0^g}^{-1})]du\right)dg.$$
		Since $x_1^g\in M$ and $\nu_X(x_1^g)=1$ (as $x_1^g$ is a compact element), a short computation gives
		$$\operatorname{ev}_1f(x_1^gu\iota(g,1))=\operatorname{ev}_{x_1^g}f(u\iota(g,1))=\psi^\bullet(x_1^g)^{-1}\operatorname{ev}_1f(u\iota(g,1)).$$
		As $\iota(G\times G)$ lies in the stabiliser of $N^\bullet$ we have that  $\psi^\bullet(x^g_1x_0^g)=\psi^\bullet(x)$ which implies
		$$b_i=\psi^\bullet(x)^{-1}\int_{G}\lambda(\pi(g)v)\otimes\left(\int_{N^\circ}\operatorname{Coeff}_{X^i}[\operatorname{ev}_1f(u\iota(g,1))\psi^\bullet(u)]du\right)dg.$$
		However, note by Fubini's theorem that
		$$\operatorname{Coeff}_{X^i}Z(X,v,\lambda,f)=\int_{G}\lambda(\pi(g)v\otimes\left(\int_{N^\circ}\operatorname{Coeff}_{X^i}[\operatorname{ev}_1f(u\iota(g,1))\psi^\bullet(u)]du\right)dg$$ and we obtain
		$$Z(X,v,\lambda,x\cdot f)=\psi^\bullet(x)^{-1}Z(X,v,\lambda,f).$$
		
	\end{proof} 
	Let $\kappa\colon G\times G\to B^\times$ be given by $\kappa(g_1,g_2)=\chi_{\theta}(\operatorname{Nrd}(g_2))$. We will write $M(\pi)^\kappa$ for the smooth $(A\otimes B)[G\times G]$-module $(V\otimes_{A}\widetilde{V})\otimes_R\kappa^{-1}$. Then the above proposition shows that the twisted doubling zeta integral can be viewed as an element of
	$$\Hom_{(A\otimes B)[X^{\pm1}]}\left(M(\pi)^\kappa\otimes_{B[G\times G]}J_{N^\bullet,{\psi^\bullet}^{-1}}(I(X,\theta)),(A\otimes B)[[X]][X^{-1}]\right)$$
	where $J_{N^\bullet,{\psi^\bullet}^{-1}}$ is the unnormalized twisted Jacquet functor.

	\section{Rationality of the twisted doubling zeta integral}\label{secratio}
	In this section we prove that the twisted doubling zeta integral is rational in the following sense. Recall that $S_{A\otimes B}$ is the mutliplicative subset of $(A\otimes B)[X^{\pm1}]$ consisting of Laurent polynomials with leading and trailing coefficient a unit in $A\otimes B$. Then we prove that for any input data the twisted doubling zeta integral lies in the localization $S_{A\otimes B}^{-1}\cdot (A\otimes B)[X^{\pm1}]$. More concretely, we show the following result.
    \begin{theorem}\label{rattheo}
        Let $(\pi,V)$ be an $A[G]$-module and assume that the contragredient $(\widetilde{\pi},\widetilde{V})$ is admissible and $G$-finite. Let $\theta$ be a $B[\GL_{kn}(D)]$-module of type $(k,n)$. Then there is a polynomial in $\mathcal P\in S_{A\otimes B}$ such that
        $$\mathcal P\cdot Z(X,v,\lambda,f)\in (A\otimes B)[X^{\pm1}]$$
           for all $v\in V,\lambda\in\widetilde{V}$ and $f\in I(X,\theta)$. 
        \end{theorem}
This result will follow from Propositions \ref{firstorbred}, \ref{bottompoly1} and \ref{annihipoly} below (whose proofs will appear at the end of this section). To state these results we need to introduce more notation.

Recall that $P$ is the stabiliser of the maximal isotropic subspace $W^{\bigtriangleup,k}$ and hence by Witt's theorem (Section 7.3 in \cite{garrett1997buildings}) we can identify $P\backslash G^{\Box,k}$ with the set of maximal isotropic subspaces in $W^{\Box,k}$. The results in Section 6.2 of \cite{cai2021twisted} imply that $P\cdot P^\bullet$ corresponds exactly to those maximal isotropic subspaces $L$ of $W^{\Box,k}$ that satisfy 
	$$L\cap (W_2^{\bigtriangledown}\oplus W_3^\bigtriangledown\oplus\dotsc\oplus W_k^\bigtriangledown)=\{0\}.$$ 
 We will denote the collection of these subspaces with $\tilde{\Omega}$. We fix a maximal split torus $A_0$ of $G^{\Box,k}$ that is contained in $M$. Moreover, let $\prescript{P}{}{W}^{P^\bullet}\subseteq W_{G^{\Box,k}}$ be a set of representatives for $W_M\backslash W_{G^{\Box,k}}/W_M$, where $W_H=\operatorname{Norm}_H(A_0)/Z_H(A_0)$ for any subgroup $H$ of $G^{\Box,k}$. The Bruhat decomposition $G^{\Box,k}=\bigcup_{w\in \prescript{P}{}{W}^{P^\bullet}} PwP^\bullet$ induces a filtration by support
 \begin{equation}\label{eq:fil}
     	I^{(0)}(X,\theta)\subseteq \dotsc \subseteq I^{(d)}(X,\theta)=I(X,\theta),
 \end{equation}
where
	$$I^{(m)}(X,\theta)=\{f\in I(X,\theta)\mid\operatorname{supp}(f)\subseteq\bigcup_{\mathclap{\substack{w\in \prescript{P}{}{W}^{P^\bullet}\\\dim (P\backslash PwP^\bullet)\geq m}}}PwP^\bullet\}.$$ 
	The bottom element $I^{(0)}(X,\theta)$ consists of functions whose support is contained in the open orbit $P\cdot P^\bullet$. Note that $\iota(G\times G)$ lies in the stabiliser of $\psi^\bullet$ and hence acts on $J_{N^\bullet,{\psi^\bullet}^{-1}}(I(X,\theta))$. We have the following result whose proof appears at the end of this section.

	\begin{proposition}[c.f. Section 7.1 in \cite{cai2021twisted}]\label{firstorbred}
        Suppose that $\theta$ is of type $(k,n)$. Then the natural inclusion $I^{(0)}(X,\theta)\hookrightarrow I(X,\theta)$ induces an isomorphism of $B[X^{\pm1}][\iota(G\times G)N^\bullet]$-modules
		\begin{equation}\label{firstiso}
			J_{N^\bullet,{\psi^\bullet}^{-1}}(I(X,\theta))\cong J_{N^\bullet,{\psi^\bullet}^{-1}}(I^{(0)}(X,\theta)).
		\end{equation}
	\end{proposition}
	Note that $\iota(G\times G)N^\bullet$ acts on $P\backslash P\cdot P^\bullet$ on the right and we have the following description of the corresponding orbits. For $L\in\tilde{\Omega}$ set $L^+=L\cap((W,0)_1\oplus W_2^\bigtriangledown\oplus\dotsc\oplus W_k^\bigtriangledown)$ and consider the image of the map
  $$L^+\rightarrow((W,0)_1\oplus W_2^\bigtriangledown\oplus\dotsc\oplus W_k^\bigtriangledown)/(W_2^\bigtriangledown\oplus\dotsc\oplus W_k^\bigtriangledown)\rightarrow (W,0)_1.$$
  This yields a totally isotropic subspace of $(W,0)_1$ and we set $\varkappa(L)$ to be its rank as a $D$-module. We have the following result. 
	\begin{lemma}[\cite{cai2021twisted}, Lemma 6.4]
		Two subspaces $L,L'\in\tilde{\Omega}$ lie in the same $\iota(G\times G)N^\bullet$-orbit if and only if $\varkappa(L)=\varkappa(L')$.
	\end{lemma}
    Let $r_0$ be the Witt index of $(W,h)$, i.e.\ the rank of a maximal totally isotropic subspace in $W$, and choose elements $\varepsilon_j\in P\cdot P^\bullet$ such that the corresponding totally isotropic subspace $\varepsilon_j^{-1}(W^{\bigtriangleup,k})$ satisfies $\varkappa(\varepsilon_j^{-1}(W^{\bigtriangleup,k}))=j$, for $0\leq j\leq r_0$. Note that we can choose $\varepsilon_0$ to be the identity. Then the above result implies an orbit decomposition $$P\cdot P^\bullet=\bigcup_{j=0}^{r_0}P\varepsilon_j\iota(G\times G)N^\bullet,$$ 
    where for all $0\leq i\leq r_0$ we have that $\bigcup_{j=0}^{i}P\varepsilon_j\iota(G\times G)N^\bullet$
    is open in $G^{\Box,k}$. In particular the orbit of the identity $P\iota(G\times G)N^\bullet$ is open in $G^{\Box,k}$. We obtain a filtration by support on $J_{N^\bullet,{\psi^\bullet}^{-1}}(I^{(0)}(X,\theta))$:
	$$0\subseteq J^{(0)}(X,\theta)\subseteq J^{(1)}(X,\theta)\subseteq\dotsc\subseteq J^{(r_0)}(X,\theta)=J_{N^\bullet,{\psi^\bullet}^{-1}}(I^{(0)}(X,\theta)),$$
	where $J^{(i)}(X,\theta)$ consists of the image under the twisted Jacquet functor of those functions in $I^{(0)}(X,\theta)$, whose support is contained in $\bigcup_{0\leq j\leq i}P\varepsilon_j\iota(G\times G)N^\bullet$. We have the following two results whose proofs appear at the end of this section.
    \begin{proposition}\label{bottompoly1}
		Suppose that $f\in I(X,\theta)$ is supported in $P\iota(G\times G)N^\bullet$.
  \begin{enumerate}
      \item For any integer $i\in\mathbb Z$ the set $$\mathcal C_i=\{(g,u)\in G\times N^\circ\mid\operatorname{Coeff}_{X^i}(f(u\iota(g,1)))\not=0\}$$
is compact. Moreover, there are only finitely many integers $i$ such that $\mathcal C_i$ is nonempty.
\item If $F$ has odd residue characteristic we have that
  $$Z(X,v,\lambda,f)\in (A\otimes B)[X^{\pm1}]$$ for all $v\in V,\lambda\in\widetilde{V}$.
  \end{enumerate}

	\end{proposition}
    \begin{proposition}\label{annihipoly} Let $(\pi,V)$ be a smooth $A[G]$-module and assume that the contragredient $(\widetilde{\pi},\widetilde{V})$ is admissible and $G$-finite. Moreover, let $\theta$ be a $B[\GL_{kn}(F)]$-module of type $(k,n)$. Then for any integer $1\leq i\leq r_0$ there is a polynomial in $S_{A\otimes B}$ that annihilates 
  $$M(\pi)^\kappa\otimes_{B[G\times G]}\left(J^{(i)}(X,\theta)/J^{(i-1)}(X,\theta)\right).$$
  \end{proposition}
  By assuming the above Propositions \ref{firstorbred}, \ref{bottompoly1} and \ref{annihipoly} we are able to prove Theorem \ref{rattheo}.
\begin{proof}[Proof of Theorem \ref{rattheo}]

 Proposition \ref{firstorbred} states that $$J^{(r_0)}(X,\theta)=J_{N^\bullet,{\psi^\bullet}^{-1}}(I^{(0)}(X,\theta))\cong J_{N^\bullet,{\psi^\bullet}^{-1}}(I(X,\theta)).$$
This together with Proposition \ref{annihipoly} implies that there is a polynomial $\mathcal P\in S_{A\otimes B}$ that annihilates 
 $$M(\pi)^\kappa\otimes_{B[G\times G]}(J_{N^\bullet,{\psi^\bullet}^{-1}}(I(X,\theta))/J^{(0)}(X,\theta)).$$
 Hence for any $f\in I(X,\theta),v\in v,\lambda\in\widetilde{V}$ there exists an $x_0\in M(\pi)^\kappa\otimes_{B[G\times G]}J^{(0)}(X,\theta)$ such that 
 $$\mathcal P(v\otimes\lambda\otimes f)=x_0$$
 in $M(\pi)^\kappa\otimes_{B[G\times G]}J_{N^\bullet,{\psi^\bullet}^{-1}}(I(X,\theta))$.

 However, Proposition \ref{bottompoly1} implies that the twisted doubling zeta integral has values in 
 $(A\otimes B)[X^{\pm1}]$ on $M(\pi)^\kappa\otimes_{B[G\times G]}J^{(0)}(I(X,\theta))$, i.e.\ it is an element of
 $$\Hom_{(A\otimes B)[X^{\pm1}]}(M(\pi)^\kappa\otimes_{B[G\times G]}J^{(0)}(I(X,\theta)),(A\otimes B)[X^{\pm1}]).$$
 Hence
 $$\mathcal P Z(X,v,\lambda,f)=Z(X,\mathcal P(v\otimes\lambda\otimes f))=Z(X,x_0)\in (A\otimes B)[X^{\pm1}],$$
 which finishes the proof.

\end{proof}

\begin{remark}\label{defeven}
As we will see below, the proofs of Propositions \ref{firstorbred}, part 1) of \ref{bottompoly1} and \ref{annihipoly} are independent of the definition of the twisted doubling zeta integral. By following the above proof, we can use them to \emph{define} the twisted doubling zeta integral directly as an element of $S_{A\otimes B}^{-1}\cdot(A\otimes B)[X^{\pm1}]$. This circumvents the definition as a Laurent series and the results needed to achieve that (specifically Proposition \ref{compactness}). In particular, this allows us to define the twisted doubling zeta integral when $F$ has even residue characteristic. 

We will write $\mathcal J$ for the $B[X^{\pm1}]$-submodule of $I(X,\theta)$ which consists of functions supported in $P\iota(G\times G)N^\bullet$. If $f$ is an element of $\mathcal J$, part 1) of Proposition \ref{bottompoly1} implies that the $(A\otimes B)[X^{\pm1}]$-valued function
$$(g,u)\mapsto \lambda(\pi(g)v)\operatorname{ev}_1f(u\iota(g,1))\psi^\bullet(u)$$
on $G\times N^\circ$ has compact support for any $v\in V,\lambda\in\widetilde{V}$. Hence the integral 
$$\int_{G\times N^\circ}\lambda(\pi(g)v)\operatorname{ev}_1f(u\iota(g,1))\psi^\bullet(u)d(g,u)\in (A\otimes B)[X^{\pm1}]$$
is well-defined. The same proof as Proposition \ref{inttransf} (note here that $\mathcal J$ is stable under the action of $\iota(G\times G)$ and $N^\bullet$) shows that this integral gives rise to an element of
$$\Hom_{(A\otimes B)[X^{\pm1}]}(M(\pi)^\kappa\otimes_{B[G\times G]}J^{(0)}(X,\theta),(A\otimes B)[X^{\pm1}]),$$
which we denote by $\mathcal Z$. Now by Proposition \ref{annihipoly} there is a polynomial $\mathcal P\in S_{A\otimes B}$ such that multiplication by $\mathcal P$ yields a map $$M(\pi)^\kappa\otimes_{B[G\times G]}J_{N^\bullet,{\psi^\bullet}^{-1}}(I(X,\theta))\to M(\pi)^\kappa\otimes_{B[G\times G]}J^{(0)}(X,\theta).$$
For any $f\in I(X,\theta),v\in V,\lambda\in\widetilde{V}$ we can consider the element $v\otimes\lambda\otimes f\in M(\pi)^\kappa\otimes_{B[G\times G]}J_{N^\bullet,{\psi^\bullet}^{-1}}(I(X,\theta))$ and define the twisted doubling zeta integral associated to $v,\lambda$ and $f$ as
$$Z(X,v,\lambda,f)\coloneqq\mathcal P^{-1}\mathcal Z(\mathcal P(v\otimes\lambda\otimes f))\in S_{A\otimes B}^{-1}\cdot (A\otimes B)[X^{\pm1}].$$
It is straightforward to see that if $F$ has odd residue characteristic, this definition agrees with the one in Section \ref{secintegrals} when viewed as an element of $S_{A\otimes B}^{-1}\cdot (A\otimes B)[X^{\pm1}]$.
\end{remark}

 \subsection{Proofs of auxiliary results}\label{auxil}
 We now give the proofs of Propositions \ref{firstorbred}, \ref{bottompoly1} and \ref{annihipoly}.
	\begin{proof}[Proof of Lemma \ref{firstorbred}] For $w\in {}^PW^{P^\bullet}$ let $I_w$ be the smooth $(B[X^{\pm1}])[P^\bullet]$-module which consists of locally constant functions $f\colon PwP^\bullet\to \theta\otimes_B (B[X^{\pm1}])$ such that $$f(xg)=\delta_P(x)^{1/2}\nu_X(x)\theta(x)f(g),$$ where $x\in P,g\in PwP^\bullet$ and $P\backslash\operatorname{supp}(f)$ is compact. Each successive quotient in the filtration (\ref{eq:fil}) is a finite direct sum of spaces $I_w$. To show the isomorphism in Equation (\ref{firstiso}) it is, by the exactness of the twisted Jacquet functor, enough to show that $J_{N^\bullet,{\psi^\bullet}^{-1}}(I_w)=0$ for $w\notin P\cdot P^\bullet$.
 
		To prove this we will utilize Proposition \ref{propbump}. Consider the topological spaces
		$$\mathbb X=P\backslash PwP^\bullet$$
		and
		$$\mathbb Y=\mathbb X/N^\bullet$$
		and the continuous projection $\varrho\colon \mathbb X\to\mathbb Y$. Note that $\mathbb X\cong (w^{-1}Pw\cap P^\bullet)\backslash P^\bullet$. Let $\mathcal F$ be the sheaf of locally constant $\theta\otimes_BB[X^{\pm1}]$-valued functions on $\mathbb X$. Then $I_w$ can be identified with the associated cosmooth $\mathcal C_c^\infty(\mathbb X)$-module $\mathcal F_c$.
		By Proposition \ref{propbump} we obtain that $(\mathcal F_c)_{N^\bullet,{\psi^\bullet}^{-1}}$ is a cosmooth $\mathcal C_c^\infty(\mathbb Y)$-module and let $\mathcal G$ be the associated $\mathcal C^\infty_c$-sheaf. To prove that $J_{N^\bullet,{\psi^\bullet}^{-1}}(I_w)=0$ it is by Proposition \ref{propbump} enough to show that $\mathcal G_{\overline{y}}\cong(\mathcal F_{\varrho^{-1}(\overline{y})})_c/(\mathcal F_{\varrho^{-1}(\overline{y})})_c(\psi^\bullet)$ is zero for all $\overline{y}\in\mathbb Y$. Let $y\in\mathbb X$ be an element in the preimage of $\overline{y}\in\mathbb Y$ under $\varrho$. Note that 
		$$\varrho^{-1}(\overline{y})\cong (w^{-1}Pw\cap N^\bullet)\backslash y N^\bullet\cong((wy)^{-1}Pwy\cap N^\bullet)\backslash N^\bullet$$
		and hence by Proposition 7.8 of \cite{cai2021twisted} we have that 
  $$(\mathcal F_{\varrho^{-1}(\overline{y})})_c\cong\cInd_{(wy)^{-1}Pwy\cap N^\bullet}^{N^\bullet}((wy)^{-1}(\theta\otimes_B\nu_X)).$$
		Since $P\cap wyN^\bullet(wy)^{-1}$ consists of compact elements, $\nu_X$ is trivial on this group which implies that it is enough to show that 
		$$J_{N^\bullet,{\psi^\bullet}^{-1}}(\operatorname{c-Ind}_{N^\bullet\cap y^{-1}w^{-1}P wy}^{N^\bullet}((wy)^{-1}(\theta)))\otimes_BB[X^{\pm1}]=0.$$
        By abuse of notation we write $\psi^\bullet$ for the left $R[N^\bullet]$-module, which is free of rank one as an $A$-module, and where $x\in N^\bullet$ acts via $x\cdot 1=\psi^\bullet(x)\cdot 1$. To have the same setup as in Lemma \ref{tensprod} we write $(\psi^\bullet)^\dagger$ for the associated right $R[N^\bullet]$-module, i.e.\ where $x\in N^\bullet$ acts via $1\cdot x=\psi^{\bullet}(x^{-1})\cdot 1$.
		We have that 
		$$J_{N^\bullet,{\psi^\bullet}^{-1}}(\operatorname{c-Ind}_{N^\bullet\cap (wy)^{-1}P wy}^{N^\bullet}((wy)^{-1}(\theta)))\cong{(\psi^\bullet)^\dagger}\otimes_{R[N^\bullet]}\left(\operatorname{c-Ind}_{N^\bullet\cap (wy)^{-1}P wy}^{N^\bullet}((wy)^{-1}(\theta))\right)$$
  and write $\mathcal N$ for this $R$-module.
	Proposition \ref{comphecke} and Lemma \ref{tensprod} imply that
		\begin{align*}
			\mathcal N &\cong {\psi^\bullet}\otimes_{\mathcal H(N^\bullet,R)}\left(\operatorname{c-Ind}_{N^\bullet\cap (wy)^{-1}P wy}^{N^\bullet}((wy)^{-1}(\theta))\right)\\
			&\cong {\psi^\bullet}\otimes_{\mathcal H(N^\bullet, R)}\left(\mathcal H(N^\bullet, R)\otimes_{\mathcal H(N^\bullet\cap (wy)^{-1}P wy,R)}((wy)^{-1}(\theta))\right)\\
			&\cong\left({\psi^\bullet}\otimes_{\mathcal H(N^\bullet, R)}\mathcal H(N^\bullet, R)\right)\otimes_{\mathcal H(N^\bullet\cap (wy)^{-1}P wy,R)}((wy)^{-1}(\theta))\\
			&\cong{\psi^\bullet}\otimes_{\mathcal H(N^\bullet\cap (wy)^{-1}P wy,R)}((wy)^{-1}(\theta))\\
			&\cong ({\psi^\bullet})^\dagger\otimes_{R[N^\bullet\cap (wy)^{-1}P wy]}((wy)^{-1}(\theta)).
		\end{align*}
		Firstly, assume there is $x\in N^\bullet\cap (wy)^{-1}N wy$ such that $\psi^\bullet(x)\not=1$. Then for any $c\otimes v\in (\psi^\bullet)^\dagger\otimes_{N^\bullet\cap (wy)^{-1}P wy}((wy)^{-1}(\theta))$ we have that 
		$${\psi^{\bullet}}(x^{-1})(c\otimes v)=c\otimes (\theta(wyn(wy)^{-1})v)=c\otimes v.$$
       Note that $\psi^\bullet(x^{-1})$ equals a nontrival $p$-th power root of unity $\zeta$, i.e.\ there is some positive integer $\alpha$ such that $$\zeta^{p^\alpha-1}+\zeta^{p^\alpha-2}+\dotsc+1=0,$$
       which implies that
       $$(1-\zeta^{p^\alpha-1})+(1-\zeta^{p^\alpha-2})+\dotsc+(1-\zeta)=p^\alpha.$$
		Hence $1-{\psi^{\bullet}}(x^{-1})\in\mathbb Z[1/p,\mu_{p^\infty}]^\times$ and we see that
		$$({\psi^\bullet})^\dagger\otimes_{R[N^\bullet\cap (wy)^{-1}P wy]}((wy)^{-1}(\theta))=0.$$
		If $\psi^{\bullet}$ restricted to $N^\bullet\cap (wy)^{-1}N wy$ is trivial, then by Lemma 6.1 of \cite{cai2021twisted}, the pair $(N^\bullet\cap (wy)^{-1}Mwy,{\psi^\bullet}^{-1})$ lies in an orbit higher than $\lambda_{k,n}$ (c.f. Section \ref{secdegwhit}), which shows that 
		\begin{align*}
			({\psi^\bullet})^\dagger\otimes_{R[N^\bullet\cap (wy)^{-1}P wy]}(wy)^{-1}(\theta)&\cong({\psi^\bullet})^\dagger\otimes_{R[N^\bullet\cap (wy)^{-1}M wy]}\left((wy)^{-1}(\theta)\right)
			\\
			&\cong J_{N^\bullet\cap (wy)^{-1}M wy,{\psi^\bullet}^{-1}}(\theta)\\&\cong 0.
		\end{align*}
		Hence the stalks $\mathcal G_{\overline{y}}$ are zero for all $\overline{y}\in\mathbb Y$ which implies that $\mathcal G$ is the zero sheaf and hence $J_{N^\bullet,{\psi^\bullet}^{-1}}(I_w)=0$ for any $w\not\in P\cdot P^\bullet$.
	\end{proof}
Next we prove Proposition \ref{bottompoly1}.
\begin{proof}[Proof of Proposition \ref{bottompoly1}]
 Note that $\operatorname{supp}(f)\subseteq G^{\Box,k}$ is open and closed and since the map $G^{\Box,k}\to P\backslash G^{\Box,k}$ is open one obtains that $P\backslash\operatorname{supp}(f)$ is closed in $P\backslash G^{\Box,k}$. Since $P\backslash G^{\Box,k}$ is compact, we see that $P\backslash\operatorname{supp}(f)$ is compact. By assumption $P\backslash\operatorname{supp}(f)$ is contained in $P\backslash P\iota(G\times G)N^\bullet$ and by Lemma \ref{homeobottom} the map 
    \begin{align*}
             \Phi\colon N^\circ\iota(G\times 1)&\to P\backslash P\iota(G\times G)N^\bullet\\
             x&\mapsto Px
    \end{align*}
    is a homeomorphism. We have that
    $$\Phi^{-1}(P\backslash\operatorname{supp}(f))=\operatorname{supp}(f|_{N^\circ\iota(G\times 1)})$$
    and hence 
  $f$ restricted to $N^\circ\iota(G\times 1)$ has compact support. This immediately implies that the sets $\mathcal C_i$ are compact and nonempty for only finitely many integers $i$.
  
  Part 2) then immediately follows from the definition of the twisted doubling zeta integral since there is an integer $M_0$ such that $\operatorname{supp}(\operatorname{Coeff}_{X^m}(f|_{N^\circ\iota(G\times 1)}))=\emptyset$ for all $m\geq M_0$.
	\end{proof}

To prove Proposition \ref{annihipoly} we will first introduce certain representatives for the $\iota(G\times G)N^\bullet$-orbits of $P\cdot P^\bullet$. One can define an injective group homomorphism $$\xi_k\colon G^{\Box,1}\hookrightarrow G^{\Box,k},$$ by letting an element $g\in G^{\Box,1}$ act as $g$ on $W_1^{\Box}$ and as the identity on $W_2^\Box\oplus\dotsc\oplus W_k^\Box$. Recall that $r_0$ is the Witt index of $(W,h)$ and let $0=U_0\subseteq U_1\subseteq\dotsc\subseteq U_{r_0}$ be a collection of totally isotropic subspaces of $W$ such that $U_i$ has rank $i$ for $0\leq i\leq r_0$. Set $\tilde{L_i}$ to be $$\{(w_1,w_2)_1\in (U_i^\perp, U_i^\perp)_1\mid w_1-w_2\in U_i\}\subseteq W_1^\Box.$$ Then for any integer $0\leq i\leq r_0$ the subspace $$L_i\coloneqq\tilde{L_i}\oplus W_2^\bigtriangleup\oplus\dotsc\oplus W_k^\bigtriangleup$$ is an element of $\tilde{\Omega}$ which satisfies that $\varkappa(L_i)=i$. By Witt's Theorem there is an isometry $\tilde{\varepsilon}_i\in G^{\Box,1}$ such that $\tilde{\varepsilon}_i(\tilde{L}_i)=W^{\bigtriangleup}_1$. We choose $\tilde{\varepsilon}_0$ to be the identity in $G^{\Box,1}$. We define elements $\varepsilon_i=\xi_k(\tilde{\varepsilon_i})\in P\cdot P^\bullet$ which then satisfy that $P\varepsilon_i$ corresponds to $L_i$, i.e.\ $\varepsilon_i^{-1}(W^{\bigtriangleup,k})=L_i$.

We also need to introduce further subgroups of $G^{\Box,k}$. Set $\tilde{L}_i^-=\tilde{\varepsilon}_i^{-1}(W_1^\bigtriangledown)$ and note that $W_1^\Box=\tilde{L}_i\oplus\tilde{L}_i^-$. Let $P_i'$ be the parabolic subgroup of $G^{\Box,1}$ that stabilizes the flag of totally isotropic subspaces $0\subseteq(0,U_i)_1\subseteq\tilde{L_i}$. We set $Q_i$ to be the parabolic subgroup of $G^{\Box,1}$ that stabilizes the flag $0\subseteq\tilde{L}_i^-$. Let $N_i'$ be the unipotent radical of $P_i'$ and define $\mathfrak a_i=\xi_k(Q_i\cap N_i')$. We also have that $\mathfrak a_i=\xi_k(N_i')\cap\varepsilon^{-1}_iM\varepsilon_i$. The map $u\mapsto u-\id$ yields a group isomorphism between $\mathfrak a_i$ and $\Hom_D(\tilde{L}_i/(0,U_i)_1,(0,U_i)_1)$. In particular $\mathfrak a_i\cong D^{(2n-i)i}$.  For $0\leq i\leq r_0$ consider the flag $$\mathcal Y_i\colon0\subseteq(0,U_i)_1\subseteq \tilde{L_i}\subseteq\tilde{L_i}\oplus W_2^\bigtriangleup\subseteq\tilde{L_i}\oplus W_2^\bigtriangleup\oplus W_3^\bigtriangleup\subseteq\dotsc\subseteq L_i$$ and the associated parabolic subgroup $P(\mathcal Y_i)=M(\mathcal Y_i)N(\mathcal Y_i)$ of $G^{\Box,k}$. We set $\overline{N}_i=N(\mathcal Y_i)\cap\varepsilon_i^{-1}M\varepsilon_i$.

As mentioned before the action of $\iota(G\times G)N^\bullet$ on $P\cdot P^\bullet$ induces a filtration by support of $J_{N^\bullet,{\psi^\bullet}^{-1}}(I^{(0)}(X,\theta))$:
	$$0\subseteq J^{(0)}(X,\theta)\subseteq J^{(1)}(X,\theta)\subseteq\dotsc\subseteq J^{(r_0)}(X,\theta)=J_{N^\bullet,{\psi^\bullet}^{-1}}(I^{(0)}(X,\theta)),$$
	where $J^{(i)}(X,\theta)$ consists of the image under the twisted Jacquet functor of those functions in $I(X,\theta)$, whose support is contained in $\bigcup_{0\leq j\leq i}P\varepsilon_j\iota(G\times G)N^\bullet$. Note that $\varepsilon_0$ is the identity in $G^{\Box,k}$.

 We set $$Q^{(i)}(X,\theta)\coloneqq J^{(i)}(X,\theta)/J^{(i-1)}(X,\theta)$$
 for $0\leq i\leq r_0$. As an $B[X^{\pm1}][P^\bullet]$-module, $I^{(0)}(X,\theta)$ can be identified with $\cInd_{P\cap P^\bullet}^{P^\bullet}(\delta_P^{1/2}\otimes \theta\otimes\nu_X)$ by sending a function in $I^{(0)}(X,\theta)$ to its restriction to $P^\bullet$.
        By the geometric Lemma (c.f. Theorem 5.2 of \cite{bernstein1977induced} for complex representations and Appendix A of \cite{cui2020modulo} for general coefficient rings) we see that 
		\begin{multline}\label{eq:geom}
		   	    Q^{(i)}(X,\theta)\cong\\\cInd_{(G\times G)\cap \varepsilon_i^{-1}P\varepsilon_i}^{G\times G}\left(\delta_{N^\bullet\cap\varepsilon_i^{-1}P\varepsilon_i}^{-1}\otimes(\varepsilon_i^{-1}\cdot J_{P\cap\varepsilon_iN^\bullet\varepsilon_i^{-1},\varepsilon_i\cdot{\psi^\bullet}^{-1}}(\delta_P^{1/2}\otimes\theta\otimes\nu_X))\right)  
		\end{multline}
	for $0\leq i\leq r_0$. For convenience we will sometimes write $(G\times G)\cap \varepsilon_i^{-1}P\varepsilon_i$ for the subgroup $\iota^{-1}(\iota(G\times G)\cap\varepsilon^{-1}P\varepsilon_i)$ of $G\times G$.

We now describe the stabiliser $\iota(G\times G)\cap\varepsilon_i^{-1}P\varepsilon_i$ of $L_i$ in $\iota(G\times G)$ in more detail. Let $P_i$ be the parabolic subgroup of $G$ that stabilizes the flag $0\subseteq U_i\subseteq W$ and choose a Levi decomposition $P_i=M_iN_i$. Note that $M_i$ is isomorphic to the product $\GL_D(U_i)\times\operatorname{Isom}(U_i^\perp/U_i,h)$ by sending $m$ to $(m|_{U_i},m|_{U_i^\perp/U_i})$. For $x\in\GL_D(U_i)$ let $m(x)$ be the element in $M_i$ that corresponds to $(x,\id)$ under this isomorphism. We then have the following result. 
\begin{proposition}\label{stabiliser} For $0\leq i\leq r_0$ we have
\begin{enumerate}
    \item $\iota(G\times G)\cap\varepsilon_i^{-1}P\varepsilon_i\subseteq\iota(P_i\times P_i)$,
    \item $\iota(N_i\times N_i)\subseteq\iota(G\times G)\cap\varepsilon_i^{-1}P\varepsilon_i$,
    \item for any $x,y\in\GL_D(U_i)$ the element $\iota(m(x),m(y))$ lies in $\iota(G\times G)\cap\varepsilon_i^{-1}P\varepsilon_i$.
\end{enumerate}
	\end{proposition}
	\begin{proof}
		The first and second statement follow from \cite[Prop 2.1]{piatetski1987functions}. For the third statement note that $\iota(G\times G)$ stabilizes $W_2^\bigtriangleup\oplus\dotsc\oplus W_k^\bigtriangleup$. Hence it is enough to show that $\iota(m(x),m(y))$ stabilizes $$\tilde{L}_i=\{(w_1,w_2)\in U_i^\perp\times U_i^\perp\mid w_1-w_2\in U_i\}\subseteq W_1^\Box$$
  for any $x,y\in\GL_D(U_i)$. However, this follows immediately since $m(x)$ and $m(y)$ act as the identity on $U_i^\perp/U_i$.
	\end{proof}

 We now describe the successive quotients of the induced filtration on $M(\pi)^\kappa\otimes_{B[G\times G]}J_{N^\bullet,{\psi^\bullet}^{-1}}(I^{(0)}(X,\theta))$ in more detail.
\begin{lemma}\label{heckecomp1}
    For $1\leq i\leq r_0$ we have an isomorphism of $(A\otimes B)[X^{\pm1}]$-modules
    $$M(\pi)^\kappa\otimes_{B[G\times G]}Q^{(i)}(X,\theta)\cong J_{1\times N_i}(M(\pi)^\kappa)\otimes_{B[P_i\times P_i]}\left(\cInd_{(G\times G)\cap \varepsilon_i^{-1}P\varepsilon_i}^{P_i\times P_i}(\Theta_i)\otimes\delta_{P_i\times P_i}^{-1}\right),$$
    where $\Theta_i$ is the smooth $B[X^{\pm1}]$-representation $$\delta_{N^\bullet\cap\varepsilon_i^{-1}P\varepsilon_i}^{-1}\otimes\left(\varepsilon_i^{-1}\cdot J_{P\cap\varepsilon_iN^\bullet\varepsilon_i^{-1},\varepsilon_i\cdot{\psi^\bullet}^{-1}}(\delta_P^{1/2}\otimes\theta\otimes\nu_X)\right)$$ of $(G\times G)\cap \varepsilon_i^{-1}P\varepsilon_i$.
    \end{lemma}
	\begin{proof}
    Note that Equation (\ref{eq:geom}) states that $Q^{(i)}(X,\theta)\cong\cInd_{(G\times G)\cap\varepsilon_i^{-1}P\varepsilon_i}^{G\times G}(\Theta_i)$.
	 Since $\nu_X$ and $\delta_P$ are trivial on compact elements we have an isomorphism of $B[X^{\pm1}][P\cap\varepsilon_i\iota(G\times G)\varepsilon_i^{-1}]$-modules
		$$J_{P\cap\varepsilon_iN^\bullet\varepsilon_i^{-1},\varepsilon_i\cdot{\psi^\bullet}^{-1}}(\delta_P^{1/2}\otimes\theta\otimes\nu_X)\cong J_{P\cap\varepsilon_iN^\bullet\varepsilon_i^{-1},\varepsilon_i\cdot{\psi^\bullet}^{-1}}(\theta)\otimes\delta_P^{1/2}\otimes\nu_X.$$
	By transitivity of compact induction and Proposition \ref{stabiliser} we see that $$Q^{(i)}(X,\theta)\cong\cInd_{P_i\times P_i}^{G\times G}\cInd_{(G\times G)\cap \varepsilon_i^{-1}P\varepsilon_i}^{P_i\times P_i}(\Theta_i).$$
		We can compute by using Propositions \ref{assocheck}, \ref{comphecke} and Lemma \ref{tensprod}: 
		\begin{align*}
			&M(\pi)^\kappa\otimes_{B[G\times G]}\cInd_{P_i\times P_i}^{G\times G}\cInd_{(G\times G)\cap \varepsilon_i^{-1}P\varepsilon_i}^{P_i\times P_i}(\Theta_i\otimes\delta_{P_i\times P_i}^{-1})\\
			\cong&M(\pi)^\kappa\otimes_{\mathcal H(G\times G,B)}\cInd_{P_i\times P_i}^{G\times G}\cInd_{(G\times G)\cap \varepsilon_i^{-1}P\varepsilon_i}^{P_i\times P_i}(\Theta_i\otimes\delta_{P_i\times P_i}^{-1})\\
			\cong&M(\pi)^\kappa\otimes_{\mathcal H(G\times G,B)}\left(\mathcal H(G\times G,B)\otimes_{\mathcal H(P_i\times P_i,B)}\left(\cInd_{(G\times G)\cap \varepsilon_i^{-1}P\varepsilon_i}^{P_i\times P_i}(\Theta_i)\otimes\delta_{P_i\times P_i}^{-1})\right)\right)\\
			\cong &M(\pi)^\kappa\otimes_{\mathcal H(P_i\times P_i,B)}\left(\cInd_{(G\times G)\cap \varepsilon_i^{-1}P\varepsilon_i}^{P_i\times P_i}(\Theta_i)\otimes\delta_{P_i\times P_i}^{-1}\right)\\
			\cong&M(\pi)^\kappa\otimes_{B[P_i\times P_i]}\left(\cInd_{(G\times G)\cap \varepsilon_i^{-1}P\varepsilon_i}^{P_i\times P_i}(\Theta_i)\otimes\delta_{P_i\times P_i}^{-1}\right).
		\end{align*}

	As a next step we will prove that $1\times N_i$ acts trivially on $$\Theta_i\cong\delta_{N^\bullet\cap\varepsilon_i^{-1}P\varepsilon_i}^{-1}\otimes\left(\varepsilon_i^{-1}\cdot J_{P\cap\varepsilon_iN^\bullet\varepsilon_i^{-1},\varepsilon_i\cdot{\psi^\bullet}^{-1}}(\theta)\otimes\varepsilon_i^{-1}\cdot(\delta_P^{1/2}\otimes \nu_X)\right).$$
        Since $1\times N_i$ consists of compact elements it is trivial on $\delta_{N^\bullet\cap\varepsilon_i^{-1}P\varepsilon_i}$ and on $\varepsilon_i^{-1}\cdot(\delta_P^{1/2}\otimes \nu_X)$. Hence it is enough to show that $1\times N_i$ acts trivially on $$\varepsilon_i^{-1}\cdot J_{P\cap\varepsilon_iN^\bullet\varepsilon_i^{-1},\varepsilon_i\cdot{\psi^\bullet}^{-1}}(\theta)\cong J_{\varepsilon_i^{-1}P\varepsilon_i\cap N^\bullet,{\psi^\bullet}^{-1}}(\varepsilon_i^{-1}\cdot\theta).$$

        By Proposition \ref{semidir} we have that $\overline{N}_i=(\varepsilon_i^{-1}P\varepsilon_i\cap N^\bullet)\rtimes\mathfrak a_i$. Moreover, Lemma \ref{semidirchar} implies that the action of $\mathfrak a_i$ on $\varepsilon_i^{-1}\cdot\theta$ gives rise to a well-defined action of $\mathfrak a_i$ on $J_{\varepsilon_i^{-1}P\varepsilon_i\cap N^\bullet,{\psi^\bullet}^{-1}}(\varepsilon_i^{-1}\cdot\theta)$.
        Let $\psi'\colon \mathfrak a_i\to\mathbb Z[1/p,\mu_{p^\infty}]^\times$ be a nontrivial character. Then by Proposition \ref{semidir} and Lemma \ref{semidirchar} we can define a character ${\psi^\bullet}^{-1}\rtimes\psi'$ on $\overline{N}_i$ by setting $({\psi^\bullet}^{-1}\rtimes\psi')(ux)\coloneqq{\psi^\bullet}^{-1}(u)\psi'(x)$ for $u\in\mathfrak a_i$ and $x\in(\varepsilon_i^{-1}P\varepsilon_i\cap N^\bullet)$.
		We then have
		$$J_{\mathfrak a_i,\psi'}\left(J_{\varepsilon_i^{-1}P\varepsilon_i\cap N^\bullet,{\psi^\bullet}^{-1}}(\varepsilon_i^{-1}\cdot\theta)\right)\cong J_{\overline{N}_i,{\psi^\bullet}^{-1}\rtimes\psi'}(\varepsilon_i^{-1}\cdot\theta).$$
        As shown in Proposition \ref{higherorb} the pair $(\overline{N}_i,{\psi^\bullet}^{-1}\rtimes\psi')$ lies in an orbit higher than $\lambda_k^n$, which, since $\theta$ is of type $(k,n)$, implies that $J_{\overline{N}_i,{\psi^\bullet}^{-1}\rtimes\psi'}(\varepsilon_i^{-1}\cdot \theta)=0$. Since $\mathfrak a_i$ and $D^{(2n-i)i}$ are isomorphic as topological groups, we can apply Proposition \ref{trivialnil} which shows that $\mathfrak a_i$ acts trivially on $J_{\varepsilon_i^{-1}P\varepsilon_i\cap N^\bullet,{\psi^\bullet}^{-1}}(\varepsilon_i^{-1}\cdot \theta)$.

		Note that $\iota(1\times N_i)$ is contained in $\varepsilon_i^{-1}P\varepsilon_i$ and hence we can write for $x\in\iota(1\times N_i)$ that $x=m_xn_x$, where $m_x\in \varepsilon_i^{-1}M\varepsilon_i$ and $n_x\in \varepsilon_i^{-1}N\varepsilon_i$. Clearly, $n_x$ acts trivially on $\Theta_i$. However for all $x\in\iota(1\times N_i)$ we have that $m_x\in \mathfrak a_i$ which by what we just showed also acts trivially on $\Theta_i$. This immediately implies that $1\times N_i$ acts trivially on $\cInd_{(G\times G)\cap \varepsilon_i^{-1}P\varepsilon_i}^{P_i\times P_i}(\Theta_i)\otimes\delta_{P_i\times P_i}^{-1}$.
		 We obtain that
    \begin{multline*}
			M(\pi)^\kappa\otimes_{B[P_i\times P_i]}\left(\cInd_{(G\times G)\cap \varepsilon_i^{-1}P\varepsilon_i}^{P_i\times P_i}(\Theta_i)\otimes\delta_{P_i\times P_i}^{-1}\right)\\
			\cong J_{1\times N_i}(M(\pi)^\kappa)\otimes_{B[P_i\times M_i]}\left(\cInd_{(G\times G)\cap \varepsilon_i^{-1}P\varepsilon_i}^{P_i\times P_i}(\Theta_i)\otimes\delta_{P_i\times P_i}^{-1}\right).
		\end{multline*}
  \end{proof}
  By using the above lemma we are now in a position to prove Proposition \ref{annihipoly}.

\begin{proof}[Proof of Proposition \ref{annihipoly}]

    By Lemma \ref{heckecomp1} we have that 
    $$M(\pi)^\kappa\otimes_{B[G\times G]}Q^{(i)}(X,\theta)\cong J_{1\times N_i}(M(\pi)^\kappa)\otimes_{B[P_i\times M_i]}\left(\cInd_{(G\times G)\cap \varepsilon_i^{-1}P\varepsilon_i}^{P_i\times P_i}(\Theta_i)\otimes\delta_{P_i\times P_i}^{-1}\right)$$
    and hence it is enough to show that the right hand side of the above equation is annihilated by an element in $S_{A\otimes B}$. 
    
    We will first construct a certain element of $B[X^{\pm1}][(G\times G)\cap\varepsilon_iP\varepsilon_i^{-1}]$ that annihilates $\Theta_i$. In the proof of Lemma \ref{heckecomp1} we saw that 
$$\Theta_i\cong\delta_{N^\bullet\cap\varepsilon_i^{-1}P\varepsilon_i}^{-1}\otimes\left(\varepsilon_i^{-1}\cdot J_{P\cap\varepsilon_iN^\bullet\varepsilon_i^{-1},\varepsilon_i\cdot{\psi^\bullet}^{-1}}(\theta)\otimes\varepsilon_i^{-1}\cdot(\delta_P^{1/2}\otimes \nu_X)\right)$$
as $B[X^{\pm1}][(G\times G)\cap\varepsilon_iP\varepsilon_i^{-1}]$-modules. By Proposition \ref{semidir} we have that $\overline{N}_i=(\varepsilon_iP\varepsilon_i^{-1}\cap N^\bullet)\rtimes\mathfrak a_i$ and by Lemma \ref{semidirchar} we can extend the character $\psi^\bullet$ to $\overline{N}_i$ by setting it to be trivial on $\mathfrak a_i$. Moreover, these results imply that $\mathfrak a_i$ acts on $J_{\varepsilon_i^{-1}P\varepsilon_i\cap N^\bullet,{\psi^\bullet}^{-1}}(\varepsilon_i^{-1}\cdot \theta)$. However, in the proof of Lemma \ref{heckecomp1} we saw this action of $\mathfrak a_i$ is trivial. We obtain
		\begin{equation}\label{eq:jacquet1}
		    J_{\varepsilon_i^{-1}P\varepsilon_i\cap N^\bullet,{\psi^\bullet}^{-1}}(\varepsilon_i^{-1}\cdot \theta)= J_{\mathfrak a_i}\left(J_{\varepsilon_i^{-1}P\varepsilon_i\cap N^\bullet,{\psi^\bullet}^{-1}}(\varepsilon_i^{-1}\cdot \theta)\right)=J_{\overline{N}_i,{\psi^\bullet}^{-1}}(\varepsilon_i^{-1}\cdot\theta).
		\end{equation}

Let $z_i'\in GL_D(U_i)$ be the central element that has $\varpi_E$ as diagonal entries, where $\varpi_E$ is a uniformizer of the ring of integers of $E$. By Proposition $\ref{stabiliser}$ we see that $z_i=\iota(1,m(z_i'))$ lies in $\iota(G\times G)\cap \varepsilon_i^{-1}P\varepsilon_i$. Let $\overline{z}_i$ be its image under the projection $\varepsilon_i^{-1}P\varepsilon_i\twoheadrightarrow\varepsilon_i^{-1}M\varepsilon_i,x\mapsto x|_{L_i}$. It is straightforward to see that $\overline{z}_i$ normalizes $\mathfrak a_i$ and $\varepsilon_i^{-1}P\varepsilon_i\cap N^\bullet$ and by using Lemma \ref{diffformula} that $\psi^{\bullet}(\overline{z_i}^{-1}uz_i)=\psi^\bullet(u)$ for all $u\in \varepsilon_i^{-1}P\varepsilon_i\cap N^\bullet$. This implies that $\overline{z}_i$ acts on $J_{\overline{N}_i,{\psi^\bullet}^{-1}}(\varepsilon_i^{-1}\cdot\theta).$

Consider the flag $$\mathcal Z_i\colon 0\subseteq (0,U_i)_1\subseteq L_i\subseteq W^{\Box,k}$$ and the associated parabolic $P(\mathcal Z_i)\subseteq G^{\Box,k}$ with unipotent radical $N(\mathcal Z_i)$. Lemma \ref{bulltriv} yields that $N(\mathcal Z_i)\cap\varepsilon_i^{-1}M\varepsilon_i$ is a normal subgroup of $\overline{N}_i$ and $\psi^{\bullet}$ is trivial on $N(\mathcal Z_i)\cap\varepsilon_i^{-1}M\varepsilon_i$. We obtain that 
\begin{equation}\label{eq:jacquet2}
    J_{\overline{N}_i,{\psi^\bullet}^{-1}}(\varepsilon_i^{-1}\cdot\theta)= J_{\overline{N}_i,{\psi^\bullet}^{-1}}(J_{N(\mathcal Z_i)\cap\varepsilon_i^{-1}M\varepsilon_i}(\varepsilon_i^{-1}\cdot\theta)).
\end{equation}

Note that $\overline{z}_i$ is an element of $P(\mathcal Z_i)\cap \varepsilon_i^{-1}M\varepsilon_i$ and we choose a Levi subgroup $M'(\mathcal Z_i)$ of the parabolic subgroup $P(\mathcal Z_i)\cap \varepsilon_i^{-1}M\varepsilon_i$ of $\varepsilon_i^{-1}M\varepsilon_i$ such that $M'(\mathcal Z_i)$ contains $\overline{z}_i$. Then $\overline{z}_i$ lies in the center of $M'(\mathcal Z_i)$. Since we assume that $\theta$ is admissible and $\GL_{kn}(D)$-finite we obtain by Proposition \ref{jacfin}, Theorem \ref{dathelm} and part 2) of Proposition \ref{endfin} that there is a polynomial
    $$Q(Y)=\sum_{j=0}^{m}b_jX^j\in B[Y]$$
    in $S_B$ (i.e.\ $b_0,b_m\in B^\times$) such that $Q(\overline{z}_i)$ annihilates $J_{N(\mathcal Z_i)\cap\varepsilon_i^{-1}M\varepsilon_i}(\varepsilon_i^{-1}\cdot\theta)$. Then by Equations (\ref{eq:jacquet1}) and (\ref{eq:jacquet2}) we see that $Q(\varepsilon_iz_i\varepsilon_i^{-1})$ annihilates $J_{P\cap\varepsilon_iN^\bullet\varepsilon_i^{-1},\varepsilon_i\cdot{\psi^\bullet}^{-1}}(\theta)$.

    Let $f\in \cInd_{(G\times G)\cap \varepsilon_i^{-1}P\varepsilon_i}^{P_i\times P_i}(\Theta_i)$ and $p_1,p_2\in P_i$, where we write $p_2=n_2m_2$ for $m_2\in M_i$ and $n_2\in N_i$.
    Since $m_2$ and $m(z_i')$ commute and $1\times N_i$ acts trivially on $\Theta_i$ we obtain that $$(z_i \cdot f)(p_1,p_2)=f(p_1,n_2m_2m(z_i'))=\Theta_i(\iota(1,m(z_i')))f(p_1,p_2).$$
   Note that
     \begin{equation*}\label{actiononcind}
        \Theta_i(\iota(1,m(z_i')))f(p_1,p_2)=\delta^{-1}_{N^\bullet\cap\varepsilon^{-1}P\varepsilon_i}(z_i)\delta_P^{1/2}(\varepsilon_iz_i\varepsilon_i^{-1})\nu_X(\varepsilon_iz_i\varepsilon_i^{-1})\theta(\varepsilon_iz_i\varepsilon_i^{-1})f(p_1,p_2).
    \end{equation*}
    It is straightforward to see that $\nu_X(\varepsilon_i z_i\varepsilon_i^{-1})=X^{ri}$, where $r=\dim_E(D)$. This and the existence of $Q(Y)\in S_B$ such that $Q(\varepsilon_iz_i\varepsilon_i^{-1})$ annihilates $J_{P\cap\varepsilon_iN^\bullet\varepsilon_i^{-1},\varepsilon_i\cdot{\psi^\bullet}^{-1}}(\theta)$ implies that there is a polynomial
    $$Q'(Y)=\sum_{j=0}^mb_j'X^{ri(m-j)}Y^j\in B[X^{\pm1}][Y]$$
    where $b_j'\in B$ such that $Q'(z_i)$ annihilates $\cInd_{(G\times G)\cap \varepsilon_i^{-1}P\varepsilon_i}^{P_i\times P_i}(\Theta_i)\otimes\delta_{P_i\times P_i}^{-1}$. By scaling we can assume that $b_m'=1$ and note that $b_0'\in B^\times$.

    We will now use the polynomial $Q'(Y)$ to prove the result. There is a faithfully flat extension $B\hookrightarrow B'$ such that the polynomial $\sum_{j=0}^mb_j'Y^j$
    splits over $B'$, i.e.\
    $$\sum_{j=0}^mb_j'Y^j=\prod_{j=1}^m(Y-\beta_j),$$
    where $\beta_j\in B'^\times$.
    We can write 
    $$Q'(Y)=X^{rim}\prod_{j=1}^m(YX^{-ri}-\beta_j)=\prod_{j=1}^m(Y-\beta_jX^{ri})$$
    where $\beta_j\in B'^\times$. Then
    $$\prod_{j=0}^m(z_i-\beta_jX^{ri})$$
    annihilates $(\cInd_{(G\times G)\cap \varepsilon_i^{-1}P\varepsilon_i}^{P_i\times P_i}(\Theta_i)\otimes\delta_{P_i\times P_i}^{-1})\otimes_B B'$.
     We have an isomorphism $J_{1\times N_i}(M(\pi)^{\kappa})\cong (V\otimes J_{N_i}(\widetilde{V})\otimes\kappa)$ of $(A\otimes B)[P_i\times M_i]$-modules and since by Proposition \ref{jacfin} and Theorem \ref{dathelm} we see that $J_{N_i}(\widetilde{V})$ is an admissible and $M_i$-finite $A[M_i]$-module there is by Proposition \ref{endfin} a polynomial $T'\in S_A$ such that $T'(z_i'^{-1})$ annihilates $J_{N_i}(\widetilde{V})$. It is then straightforward to see that there is a polynomial 
        $T\in S_{A\otimes B}$ such that $T(z_i^{-1})$ annihilates $J_{1\times N_i}(M(\pi)^{\kappa})$. 
        
    We will now prove that 
    $$\prod_{j=0}^mT(\beta_jX^{ri})\in (A\otimes B')[X^{\pm1}]$$
    annihilates 
    \begin{equation}\label{annind}
        J_{1\times N_i}(M(\pi)^{\kappa})\otimes_{B[P_i\times M_i]}\left(\cInd_{(G\times G)\cap \varepsilon_i^{-1}P\varepsilon_i}^{P_i\times P_i}(\Theta_i)\otimes\delta_{P_i\times P_i}^{-1}\right)\otimes_B B'.
    \end{equation}

    For $0\leq t\leq m$ let $S_i^t$ be the $B'[X^{\pm1}][P_i\times P_i]$-submodule of $$\left(\cInd_{(G\times G)\cap \varepsilon_i^{-1}P\varepsilon_i}^{P_i\times P_i}(\Theta_i)\otimes\delta_{P_i\times P_i}^{-1}\right)\otimes_B B'$$ that consists of elements that are annihilated by $\prod_{j=1}^t(z_i-\beta_jX^{ri})$. Note that $S_i^0=\{0\}$ and $S_i^m=\left(\cInd_{(G\times G)\cap \varepsilon_i^{-1}P\varepsilon_i}^{P_i\times P_i}(\Theta_i)\otimes\delta_{P_i\times P_i}^{-1}\right)\otimes_B B'$.
    
    We will prove inductively that $\prod_{j=1}^tT(\beta_jX^{ri})$ annihilates $J_{1\times N_i}(M(\pi)^{\kappa})\otimes_{B[P_i\times M_i]}S_i^t$. If $t=0$ the statement is clearly true, so let $1\leq t\leq m$ and suppose it holds for $t-1$. Let $\varphi\in J_{1\times N_i}(M(\pi)^{\kappa})$ and $f\in S_i^{t}$. Then we can write $(z_i-\beta_{t}X^{-ri})f=f_0$, where $f_0\in S_{i}^{t-1}$.
    Since by induction hypothesis we have that $\prod_{j=1}^{t-1}T(\beta_jX^{ri})(\varphi\otimes f_0)=0$ we obtain for any $\alpha\in A\otimes B$ that
    \begin{align*}
       (\alpha\beta_lX^{ri})\prod_{j=1}^{t-1}T(\beta_jX^{ri})(\varphi\otimes f)&=\prod_{j=1}^{t-1}T(\beta_jX^{ri})(\alpha\varphi\otimes\beta_lX^{ri}f)\\
       &=\prod_{j=1}^{t-1}T(\beta_jX^{ri})(\alpha\varphi\otimes (z_if-f_0))\\
       &=\prod_{j=1}^{t-1}T(\beta_jX^{ri})(\alpha\varphi\otimes z_if)\\
       &=\prod_{j=1}^{t-1}T(\beta_jX^{ri})(\alpha z_i^{-1}\varphi\otimes f).
    \end{align*}
   This immediately implies that
    $$\prod_{j=1}^{t}T(\beta_jX^{ri})(\varphi\otimes f)=\prod_{j=1}^{t-1}T(\beta_jX^{ri})(T(z_i^{-1})\varphi\otimes f)=0.$$
    Induction yields that $\prod_{j=1}^tT(\beta_jX^{ri})\in (A\otimes B')[X^{\pm1}]$ annihilates the $(A\otimes B')[X^{\pm1}]$-module (\ref{annind}). However, note that the coefficients of $\prod_{j=1}^tT(\beta_jX^{ri})$ are symmetric polynomials in the $\beta_l$ which implies that $\prod_{j=1}^tT(\beta_jX^{ri})$ actually lies in $(A\otimes B)[X^{\pm1}]$. Moreover, since $\prod_{j=1}^t\beta_j=b_0'\in B^\times$ we see that $\prod_{j=1}^tT(\beta_jX^{ri})$ is an element of $S_{A\otimes B}$. Since $B\hookrightarrow B'$ is faithfully flat we have an embedding 
        $$M(\pi)^\kappa\otimes_{B[G\times G]}Q^{(i)}(X,\theta)\hookrightarrow \left(M(\pi)^\kappa\otimes_{B[G\times G]}Q^{(i)}(X,\theta)\right)\otimes_B B'$$
        and since we just showed that the right hand side is annihilated by a polynomial in $S_{A\otimes B}$ the same is true for $M(\pi)^\kappa\otimes_{B[G\times G]}Q^{(i)}(X,\theta)$.

	\end{proof}

\section{Intertwining Operator}
To state the functional equation that is satisfied by the twisted doubling zeta integrals we need a certain intertwining operator which we construct in this section. For the corresponding results over the complex numbers see \cite{waldspurger2003formule}. Dat ``algebraized'' this construction in \cite{dat2005nu} and we follow his ideas. 

Let $w_{kn}$ be the element of $G^{\Box,k}$ that sends for any $1\leq j\leq k$ and $w_1,w_2\in W$ the element $(w_1,w_2)_j$ to $(w_1,-w_2)_j$.
Note that since $M=w_{kn}Mw_{kn}$ we can define for any smooth $R[M]$-module $(\varrho,V_{\varrho})$ an $R[M]$-module $(\varrho^{w_{kn}},V_{\varrho})$ by setting that $\varrho^{w_{kn}}(m)v=\varrho(w_{kn}mw_{kn})v$ for $m\in M,v\in V_{\varrho}$. We fix an $R$-valued Haar measure on $N$ that is normalized on a compact open subgroup (which via the isomorphism $N\to\overline{N},x\mapsto w_{kn}xw_{kn}$ defines a Haar measure on $\overline{N}$).

\begin{proposition}\label{intertwining}
    There exists an intertwining operator $$\mathbb M\colon I(X,\theta)\to I(X^{-1},\theta^{w_{kn}})\otimes_{B[X^{\pm1}]}S_B^{-1}\cdot B[X^{\pm1}]$$ such that for a function $f\in I(X,\theta)$ with $\operatorname{supp}(f)\subseteq Pw_{kn}P$ the equation
    $$\mathbb M(f)(1)=\int_{N}f(w_{kn}x)dx$$
    holds. Moreover, there is a polynomial $\mathcal Q\in S_B$ such that $$\mathcal Q\cdot\mathbb M\colon I(X,\theta)\to I(X^{-1},\theta^{w_{kn}}).$$
\end{proposition}
\begin{proof}
We write $\overline{P}$ for the parabolic of $G^{\Box,k}$ that stabilizes the flag $0\subset W^{\bigtriangledown,k}\subset W^{\Box,k}$ and $\overline{N}$ for its unipotent radical. Note that $\overline{P}$ is an opposite parabolic of $P$ and that $\overline{P}=M\overline{N}$. Recall that we have an isomorphism $M\cong\GL_{kn}(D)$ by sending $m\in M$ to $m|_{W^{\bigtriangleup,k}}$. We start by constructing a nontrivial element of 
$$\Hom_{B[X^{\pm1}][G^{\Box,k}]}\left(i_P^{G^{\Box,k}}(\theta\otimes\nu_X),i_{\overline{P}}^{G^{\Box,k}}(\theta\otimes\nu_X)\right),$$
which by Frobenius reciprocity is equivalent to constructing a nontrivial element of 
$$\Hom_{B[X^{\pm1}][M]}\left(r_{\overline{P}}^{G^{\Box,k}}i_P^{G^{\Box,k}}(\theta\otimes\nu_X),\theta\otimes\nu_X\right).$$
We write $\mathbf{\Lambda}$ for the $B[X^{\pm1}][M]$-module $r_{\overline{P}}^{G^{\Box,k}}i_P^{G^{\Box,k}}(\theta\otimes\nu_X)$ and will use the geometric lemma to study $\mathbf{\Lambda}$. We could work abstractly (as it is done in \cite{dat2005nu}), however in the spirit of the proofs in Section \ref{auxil} we can give a rather explicit description of the double cosets $P\backslash G^{\Box,k}/\overline{P}$.
Firstly, recall that the cosets $P\backslash G^{\Box,k}$ correspond to maximal totally isotropic subspaces of $W^{\Box,k}$. Suppose that $Px$ and $Px'$ correspond to maximal totally isotropic subspaces $L$ and $L'$ respectively. Then by \cite[Lemma 2.4 (1)]{cai2021twisted} there is an element $y\in\overline{P}$ such that $Px=Px'y$ if and only if $\operatorname{dim}_D(L\cap W^{\bigtriangledown,k})=\operatorname{dim}_D(L'\cap W^{\bigtriangledown,k})$. For each integer $0\leq j\leq kn$ let $w_j$ be an element of $G^{\Box,k}$ such that $\dim_D(w_j^{-1}(W^{\bigtriangleup,k})\cap W^{\bigtriangledown,k})=j$.
Let $V_j^\bigtriangledown$ be a subspace of $W^{\bigtriangledown,k}$ with $\dim_D(V_j^\bigtriangledown)$ and let $V_j^\bigtriangleup$ be the $n-j$-dimensional subspace of $W^{\bigtriangleup,k}$ that is the image of the composition
$$(V_j^\bigtriangledown)^\perp\to W^{\Box,k}/W^{\bigtriangledown,k}\twoheadrightarrow W^{\bigtriangleup,k}.$$
Moreover, choose subspaces $U_j^\bigtriangleup$ and $U_j^\bigtriangledown$ such that $U_j^\bigtriangleup\oplus V_j^\bigtriangleup=W^{\bigtriangleup,k}$ and $U_j^\bigtriangledown\oplus V_j^\bigtriangledown=W^{\bigtriangledown,k}$. By Witt's theorem there is an element $w_j\in G^{\Box,k}$ such that $w_j^{-1}(W^{\bigtriangleup,k})=V_j^\bigtriangleup\oplus V_j^\bigtriangledown$ and by \cite[Lemma 2.4 (2)]{cai2021twisted} we can moreover assume that $w_j^{-1}(W^{\bigtriangledown,k})=U_j^{\bigtriangleup}\oplus U_j^{\bigtriangledown}$. Note that we can choose $w_0$ to be the identity and $w_{kn}$ is the already defined element in $G^{\Box,k}$ that sends for any $1\leq j\leq k$ and $w_1,w_2\in W$ the element $(w_1,w_2)_j$ to $(w_1,-w_2)_j$.

We then have a Bruhat decomposition
$$G^{\Box,k}=\bigcup_{j=0}^{kn}Pw_j\overline{P}$$
and we can apply the geometric Lemma to analyse $\mathbf{\Lambda}$. We obtain a filtration
$$0\subseteq\Lambda_0\subseteq\Lambda_1\subseteq\dotsc\subseteq\Lambda_{kn}=\mathbf{\Lambda},$$
where $\Lambda_i$ consists of the image in $\mathbf{\Lambda}$ of those functions in $i_P^{G^{\Box,k}}(\theta\otimes\nu_X)$ whose support is contained in $\bigcup_{j=0}^iPw_j\overline{P}$. The geometric Lemma states that
$$\Lambda_i/\Lambda_{i-1}\cong i_{M\cap w_iPw_i^{-1}}^M\left(w_i\cdot r^M_{M\cap w_i^{-1}\overline{P}w_i}(\theta\otimes\nu_X)\right).$$
We will now construct a certain endomorphism of $\mathbf{\Lambda}$ that stabilizes the $\Lambda_i$ and annihilates each quotient $\Lambda_i/\Lambda_{i-1}$ for $1\leq i\leq kn$. Moreover, it acts as multiplication by a polynomial in $S_B$ on $\Lambda_0$. The existence of such a map follows from the proof of Theorem IV.1.1 in \cite{waldspurger2003formule}. However, by using our above description of the double cosets $P\backslash G^{\Box,k}/\overline{P}$ we can give a more explicit proof.

Let $a\in M$ be the element of $G^{\Box,k}$ that maps any $w\in W^{\bigtriangleup,k}$ to $\varpi_Ew$. Then $a$ lies in the center of $M$ and its action on $W^{\bigtriangledown,k}$ is given by $w\mapsto\varpi_E^{-1}w$ for $w\in W^{\bigtriangledown,k}$. Moreover, by how we have chosen the $w_i$ for $0\leq i\leq kn$ we see that $a\in M\cap w_iMw_i^{-1}$ and that
$$\nu_X(w_i^{-1}aw_i)=X^{r(kn-2j)},$$
where $r^2=\dim_E(D)$.

Since $\nu_X$ is trivial on compact elements we have that $$r^M_{M\cap w_i^{-1}\overline{P}w_i}(\theta\otimes\nu_X)\cong r^M_{M\cap w_i^{-1}\overline{P}w_i}(\theta)\otimes \nu_X$$ as $B[X^{\pm1}][M\cap w_i^{-1}Mw_i]$-modules. For any $1\leq j\leq kn$ note that $w_i^{-1}aw_i$ lies in the center of $M\cap w_i^{-1} Mw_i$ and in particular gives rise to an element in 
$$\End_{B[M\cap w_i^{-1}Mw_i]}\left(r^M_{M\cap w_i^{-1}\overline{P}w_i}(\theta)\right).$$
Since by assumption $\theta$ is admissible and finitely generated as an $B[M]$-module there is by Propositions \ref{endfin} and \ref{jacfin} and by Theorem \ref{dathelm} a nontrivial polynomial 
$$Q_i'(Y)=\sum_{j=0}^{m_i}\beta_j' Y^j\in B[Y]$$
in $S_B$ such that $Q_i'(w_i^{-1}aw_i)$ annihilates $r^M_{M\cap w_i^{-1}\overline{P}w_i}(\theta)$. Then the element 
$$\sum_{j=0}^{m_i}\beta_j'X^{r(m_i-j)(kn-2i)}(w_i^{-1}aw_i)^j\in B[X^{\pm1}][M\cap w_i^{-1}Mw_i]$$
annihilates $r^M_{M\cap w_i^{-1}\overline{P}w_i}(\theta)\otimes\nu_X$. It is straightforward to see that there is some polynomial 
$$\hat{Q}_i(Y)=\sum_{j=0}^{m_i}\beta_jX^{r(m_i-j)(kn-2i)}Y^j\in B[X^{\pm1}][Y],$$
where $\beta_j\in B$ and $\beta_0,\beta_{m_i}\in B^\times$, such that $\hat{Q}_i(a)\in B[X^{\pm1}][M\cap w_iMw_i^{-1}]$ annihilates $$\delta_{M\cap w_iPw_i^{-1}}^{1/2}\otimes\left(w_i\cdot r^M_{M\cap w_i^{-1}\overline{P}w_i}(\theta\otimes\nu_X)\right).$$
Then $\hat{Q}_i(a)$ annihilates $\Lambda_i/\Lambda_{i-1}$.

Note that $\Lambda_0\cong\theta\otimes\nu_X$. Since we assume that $\theta$ satisfies Schur's Lemma, the center of $M$ (which is isomorphic to $E^\times$) acts via a character $\omega_\theta\colon E^\times\to B^\times$. This implies that $\hat{Q}_i(a)$ acts on $\Lambda_0$ via multiplication by the polynomial
$$Q_i(X)=\sum_{j=0}^{m_i}\beta_j\omega_\theta(a)^jX^{jkn+(m_i-j)(kn-2i))}\in B[X^{\pm1}].$$
It is straightforward to see that the above polynomial lies in $S_B$.
We set
$$\hat{\mathcal Q}=\prod_{i=1}^{kn}\hat{Q}_i(a)\in B[X^{\pm1}][M],$$
and by construction, acting by $\hat{\mathcal Q}$ yields a map from $\mathbf{\Lambda}$ to $\Lambda_0$. On $\Lambda_0$ it acts via multiplication by the polynomial
$$\mathcal Q=\prod_{i=1}^{kn}Q_i(X)\in S_B.$$
Note that the map
\begin{align*}
    \phi\colon\overline{N}&\to P\backslash P\overline{P}\\
    x&\mapsto Px
\end{align*}
is a homeomorphism. Moreover, for any $f\in I(X,\theta)$, whose support is contained in $P\overline{P}$, we have that
$$\phi^{-1}(P\backslash\operatorname{supp}(f))=\operatorname{supp}(f|_{\overline{N}}).$$
Since $P\backslash P\overline{P}$ is compact this implies that for any such function $f\in I(X,\theta)$, where $\operatorname{supp}(f)\subseteq P\overline{P}$, the integral 
$$\int_{\overline{N}}f(x)dx$$
yields a well-defined element in $\theta\otimes\nu_X$. It is straightforward to check that we obtain a $B[X^{\pm1}][M]$-module morphism
\begin{align*}
    \Phi\colon\Lambda_0&\to\theta\otimes\nu_X\\
    f&\mapsto\int_{\overline{N}}f(x)dx.
\end{align*}
Overall, $\Phi\circ\hat{\mathcal Q}$ is then an element of
$$\Hom_{B[X^{\pm1}][M]}(\mathbf{\Lambda},\theta\otimes\nu_X)$$
and Frobenius reciprocity yields an element of 
$$\Hom_{B[X^{\pm1}][G^{\Box,k}]}(I(X,\theta),i_{\overline{P}}^{G^{\Box,k}}(\theta\otimes\nu_X))$$
that is characterized by
$$f\mapsto(g\mapsto \Phi(\hat{\mathcal Q}(g\cdot f)).$$
Since $P=w_{kn}\overline{P}w_{kn}$ we can define a map
\begin{align*}
    i_{\overline{P}}^{G^{\Box,k}}(\theta\otimes\nu_X)&\to i_P^{G^{\Box,k}}((\theta\otimes\nu_X)^{w_{kn}})\\
    f&\mapsto (g\mapsto f(w_{kn}g)).
\end{align*}
Let $m\in M$ and write $m_{\mathbf e}\in\GL_{kn}(D)$ for the induced endomorphism of $W^{\bigtriangleup,k}$ with respect to the basis $\mathbf e=(e_1,\dotsc,e_{kn})$ chosen in Section \ref{setup}. Then the matrix of the endomorphism of $W^{\bigtriangleup,k}$ induced by $w_{kn}mw_{kn}$ is given by $J^{-1}(m_{\mathbf e}^*)^{-1}J$. In particular we obtain that $\nu_X^{w_{kn}}(m)=\nu_X(m)^{-1}$ for any $m\in M$ and hence $i_P^{G^{\Box,k}}((\theta\otimes\nu_X)^{w_{kn}})\cong I(X^{-1},\theta^{w_{kn}})$.

We can then define 
$$\mathbb M(f)(g)\coloneqq\mathcal Q^{-1}\Phi(\hat{\mathcal Q}(w_{kn}g\cdot f))$$
and it is straightforward to see from our construction that if $\operatorname{supp}(f)\subseteq Pw_{kn}P$ then
$$\mathbb M(f)(1)=\int_Nf(w_{kn}x)dx.$$

\end{proof}

\section{Functional Equation}

The goal of this section is to prove the following functional equation.

\begin{theorem}\label{funcequ} Let $(\pi,V)$ be a smooth $A[G]$-module and $\theta$ a smooth $B[\GL_{kn}(D)]$-module.
Suppose that 
\begin{itemize}
    \item $(\widetilde{\pi},\widetilde{V})$ is admissible, $G$-finite and the canonical trace map $V\otimes\widetilde{V}\to A$ is surjective,
    \item $\theta$ is of type $(k,n)$ and the evaluation $\operatorname{ev}_1\colon \operatorname{Wh}_{N(\mathcal Y),\psi_{\mathcal A}}(\theta)\to B$ is surjective,
    \item and $\Hom_{(A\otimes B)[G]}(V\otimes_A\widetilde{V}\otimes B,A\otimes B)\cong A\otimes B,$
    where $G$ acts diagonally on $V\otimes_A\widetilde{V}$.
\end{itemize}
    Then there exists a unique element $\Gamma(X,\pi,\theta,\psi)\in S_{A\otimes B}^{-1}\cdot (A\otimes B)[X^{\pm1}]$ such that 
    $$\Gamma(X,\pi,\theta,\psi)Z(X,v,\lambda,f)=Z(X^{-1},v,\lambda,\mathbb M(f))$$
    for all $v\in V,\lambda\in\widetilde{V}$ and $f\in I(X,\theta)$.
\end{theorem}
\begin{remark}
Note that as we saw in the previous section in general $\mathbb M(f)$ might have a denominator for $f\in I(X,\theta)$, i.e.\ $\mathbb M(f)\in I(X^{-1},\theta^{w_{kn}})\otimes_{B[X^{\pm1}]}S_B^{-1}\cdot B[X^{\pm1}]$. This is not covered by our definition of the twisted doubling zeta integral. However, by Proposition \ref{intertwining} there is a polynomial $\mathcal Q\in S_B$ such that $\mathcal Q\cdot\mathbb M(f)\in I(X^{-1},\theta^{w_{kn}})$ and we might then define
$$Z(X^{-1},v,\lambda,\mathbb M(f))\coloneqq\mathcal Q^{-1} Z(X^{-1},v,\lambda,\mathcal Q\cdot\mathbb M(f))\in S_{A\otimes B}^{-1}\cdot(A\otimes B)[X^{\pm1}]$$
for $v\in V,\lambda\in\widetilde{V}$ and $f\in I(X,\theta)$.
\end{remark}
To prove this result we will need the following two auxiliary results. Propositions \ref{inttransf} and \ref{bottompoly1} show that the twisted doubling zeta can be viewed as an element of 
 $$\Hom_{(A\otimes B)[X^{\pm 1}]}\left(M(\pi)^\kappa\otimes_{B[G\times G]}J^{(0)}(X,\theta),(A\otimes B)[X^{\pm1}]\right).$$
 Under the assumptions of Theorem \ref{funcequ} we can compute this Hom-space rather explicitly.
\begin{proposition}\label{homcomp}
Let $G^\Delta=\{(g,g)\mid g\in G\}$ be the diagonal in $G\times G$. Suppose that $\theta$ is of type $(k,n)$. Then we have an $(A\otimes B)[X^{\pm1}]$-module isomorphism 
$$M(\pi)^\kappa\otimes_{B[G\times G]}J^{(0)}(X,\theta)\cong (V\otimes\widetilde{V}\otimes B)\otimes_{B[G^\Delta]}B[X^{\pm1}],$$
where $G^\Delta$ acts trivially on $B[X^{\pm1}]$.
Moreover, if
\begin{itemize}
    \item the canonical trace map $V\otimes\widetilde{V}\to A$ is surjective
    \item and $\Hom_{(A\otimes B)[G]^\Delta}(V\otimes_A\widetilde{V}\otimes B,A\otimes B)\cong A\otimes B$,
\end{itemize}
then
$$\Hom_{(A\otimes B)[X^{\pm 1}]}\left(M(\pi)^\kappa\otimes_{B[G\times G]}J^{(0)}(X,\theta),(A\otimes B)[X^{\pm1}]\right)\cong (A\otimes B)[X^{\pm1}].$$
\end{proposition}
\begin{proof}
  By Lemma 5.1 of \cite{cai2021twisted} we have that $P\cap\iota(G\times G)=\iota(G^\Delta)$. Then Equation (\ref{eq:geom}) yields that 
$$J^{(0)}(X,\theta)\cong\cInd_{G^\Delta}^{G\times G}\left(\delta_{N^\bullet\cap P}^{-1}\otimes\delta_P^{1/2}\otimes \nu_X\otimes J_{N^\bullet\cap P,{\psi^\bullet}^{-1}}(\theta)\right).$$
By Proposition \ref{nochar} we obtain that $\delta_{N^\bullet\cap P}$ and $\delta_{P}$ and $\nu_X$ are trivial on $\iota(G^\Delta)$. Since by Proposition \ref{higherorb} the pair $(M\cap N^\bullet,{\psi^\bullet}^{-1})$ lies in the orbit $\lambda_{k,n}$ we see that $J_{N^\bullet\cap P,{\psi^\bullet}^{-1}}(\theta)\cong B$ and by Proposition \ref{invstab} we obtain that $G^\Delta$ acts on $J_{N^\bullet\cap P,{\psi^\bullet}^{-1}}(\theta)$ via $\chi_\theta\circ\operatorname{Nrd}_W$. Note that on $G^\Delta$ the character $\chi_\theta\circ\operatorname{Nrd}_W$ agrees with $\kappa$. Hence we have an isomorphism $$J^{(0)}(X,\theta)\cong\cInd_{G^\Delta}^{G\times G}(\kappa\otimes_B B[X^{\pm1}])$$ of $B[X^{\pm1}][G\times G]$-modules. Now by Propositions \ref{assocheck}, \ref{comphecke} and Lemma \ref{tensprod} we have $(A\otimes B)[X^{\pm1}]$-module isomorphisms
\begin{align*}
    M(\pi)^\kappa\otimes_{B[G\times G]}J^{(0)}(X,\theta)&\cong M(\pi)^\kappa\otimes_{B[G\times G]}\cInd_{G^\Delta}^{G\times G}(\kappa\otimes_B B[X^{\pm1}])\\
    &\cong M(\pi)^\kappa\otimes_{\mathcal H(G\times G,B)}(\mathcal H(G\otimes G,B)\otimes_{\mathcal H(G^\Delta,B)}(\kappa\otimes_BB[X^{\pm1}]))\\
    &\cong M(\pi)^\kappa\otimes_{B[G^\Delta]}(\kappa\otimes_BB[X^{\pm1}])\\
    &\cong (V\otimes_A\widetilde{V}\otimes B)\otimes_{B[G^\Delta]}B[X^{\pm1}].
\end{align*}

For convenience we will denote 
 $$\Hom_{(A\otimes B)[X^{\pm 1}]}\left(M(\pi)^\kappa\otimes_{B[G\times G]}J^{(0)}(X,\theta),(A\otimes B)[X^{\pm1}]\right).$$
 by $\mathbf H$. Note that
$$ (V\otimes_A\widetilde{V}\otimes B)\otimes_{B[G^\Delta]}B[X^{\pm1}]\cong  (V\otimes_A\widetilde{V}\otimes B)\otimes_{(A\otimes B)[G^\Delta]}(A\otimes B)[X^{\pm1}].$$
Then tensor-hom adjunction and the above yields that $\mathbf H$ is isomorphic to
$$\Hom_{(A\otimes B)[G\times G]}(V\otimes_A\widetilde{V}\otimes B,(A\otimes B)[X^{\pm1}]).$$

Recall that we assume that the canonical trace map $V\otimes_A\widetilde{V}\to A$ is surjective and that $\Hom_{(A\otimes B)[G]}(V\otimes_A\widetilde{V}\otimes B,A\otimes B)\cong A\otimes B$. This implies that every element of $\Hom_{(A\otimes B)[G]}(V\otimes_A\widetilde{V}\otimes B,A\otimes B)$ is given by
$$w\otimes\lambda\otimes 1\mapsto c\lambda(w),$$
for some $c\in A\otimes B$. For every $i\in\mathbb Z$ we have the $A\otimes B$-linear coefficient map $\alpha_i\colon (A\otimes B)[X^{\pm1}]\to A\otimes B$ which sends a Laurent polynomial to its $i$-th coefficient. Hence for $$\phi\in\Hom_{(A\otimes B)[G^\Delta]}(V\otimes_A\widetilde{V}\otimes B,(A\otimes B)[X^{\pm1}])$$ and $i\in\mathbb Z$ we can define a map $\phi_i\coloneqq\alpha_i\circ\phi$ which is then an element of $$\Hom_{(A\otimes B)[G]}(V\otimes_A\widetilde{V}\otimes B,A\otimes B)$$ and hence satisfies $$\phi_i(v\otimes\lambda\otimes 1)=c_i\lambda(v)$$
for all $v\in V,\lambda\in\widetilde{V}$ and some $c_i\in A\otimes B$. Since $\phi=\sum_{i=-\infty}^{\infty}\phi_iX^i$ this proves that $$\phi(v\otimes\lambda\otimes 1)=\left(\sum_{i=-\infty}^\infty c_iX^i\right)\lambda(v)$$
for all $v\in V,\lambda\in\widetilde{V}$. Now since we assume that the canonical trace map $V\otimes_A\widetilde{V}\to A$ is surjective we obtain that only finitely many of the $c_i$ are nonzero which proves that 
$$\Hom_{(A\otimes B)[G^\Delta]}(V\otimes_A\widetilde{V}\otimes B,(A\otimes B)[X^{\pm1}])\cong (A\otimes B)[X^{\pm1}]$$
and hence also 
$$\mathbf H\cong (A\otimes B)[X^{\pm1}].$$
\end{proof}
In the following we need to assume that there is an element in $\alpha_0\in\operatorname{Wh}_{N(\mathcal Y),\psi_{\mathcal A}}(\theta)\subseteq\Ind_{N(\mathcal Y)}^{\GL_{kn}(D)}(\psi_{\mathcal A})$ such that $\operatorname{ev}_1(\alpha_0)=1$. Under this assumption we are able to show that the twisted doubling zeta integral can be made constant.
\begin{proposition}\label{zetaconst}
Suppose that
\begin{itemize}
    \item the canonical trace map $V\otimes_A\widetilde{V}\to A$
    \item and evaluation at the identity $\operatorname{ev}_1\colon \operatorname{Wh}_{N(\mathcal Y),\psi_{\mathcal A}}(\theta)\to B$
\end{itemize}
are surjective. Then there are elements $v\in V,\lambda\in\widetilde{V}$ and $f\in I(X,\theta)$ with $\operatorname{supp}(f)\subseteq P\iota(G\times G)N^\bullet$ such that 
    $$Z(X,v,\lambda,f)=1.$$
\end{proposition}
\begin{proof}
    Let $\varphi=\sum_{i=1}^lv_i\otimes\lambda_i$ be an element of $V\otimes_A\widetilde{V}$ such that $\sum_{i=1}^l\lambda_i(v_i)=1$. Moreover, choose an element $\alpha_0\in\operatorname{Wh}_{N(\mathcal Y),\psi_{\mathcal A}}(\theta)$ such that $\operatorname{ev}_1(\alpha_0)=1$. Let $K$ be a compact open subgroup of $G$ such that $v_i\in V^K$ for $1\leq i\leq n$ and such that the volume $\mu_G(K)$ of $K$ is invertible in $R$. 
    By Proposition \ref{homeobottom} we have a homeomorphism $\Phi\colon N^\circ\iota(G\times 1)\to P\backslash PN^\bullet\iota(G\times G) $. Choose an open compact subgroup $N_0$ of $N^\circ$ such that the restriction of $\psi^{\bullet}$ to $N_0$ is trivial. Let $f_0\in I(X,\theta)$ be the function whose support is contained in $PN^\bullet\iota(G\times G)$ and which satisfies that
    $$\Phi^{-1}(\operatorname{supp}(f_0))=\operatorname{supp}(f_0|_{N^\circ\iota(G\times 1)})=N_0K$$
    and $f_0(x)=\alpha_0$ for $x\in N_0K$. Then if $i\not=0$ we clearly have for all $u\in N^\circ, g\in G$ that
    $$\operatorname{Coeff}_{X^i}(\operatorname{ev}_1f_0(u\iota(g,1)))=0$$
     and hence

    \begin{align*}
        Z(X,\varphi,f_0)&=\int_{G\times N^\circ}\operatorname{Coeff}_{X^0}\left[\sum_{i=1}^l\lambda_i(\pi(g)v_i)\otimes\operatorname{ev_1}f_0(u\iota(g,1))\psi^\bullet(u)\right]d(g,u)\\
        &=\int_{K\times N_0}\sum_{i=1}^l\lambda_i(\pi(g)v_i)\otimes\operatorname{ev_1}f_0(u\iota(g,1))\psi^\bullet(u)d(g,u)\\
        &=\int_{K\times N_0}\sum_{i=1}^l\lambda_i(v_i)\otimes\operatorname{ev_1}(\alpha_0)d(g,u)\\
        &=\mu_G(K)\mu_{N^\circ}(N_0).
    \end{align*}
     Since $\mu_{N^\circ}(N_0)$ is invertible in $\mathbb Z[1/p,\mu_{p^\infty}]$ (it is a power of $p$) the result follows.
\end{proof}
By using the two above results we are now in a position to prove Theorem \ref{funcequ}.
\begin{proof}[Proof of \ref{funcequ}]
As already mentioned the twisted doubling zeta integral can be viewed as an element of
\begin{equation}\label{zetahom}
    \Hom_{(A\otimes B)[X^{\pm 1}]}\left(M(\pi)^\kappa\otimes_{B[G\times G]}J^{(0)}(X,\theta),(A\otimes B)[X^{\pm1}]\right),
\end{equation}
which by Proposition \ref{homcomp} is free of rank one as an $(A\otimes B)[X^{\pm1}]$-module. Moreover, Proposition \ref{zetaconst} implies that it is actually a generator of the above Hom-space.

We will now construct a second element in this $\Hom$-space which will imply the functional equation. By Proposition \ref{intertwining} we have an intertwining operator
$$\mathcal Q\cdot\mathbb M\colon I(X,\theta)\to I(X^{-1},\theta^{w_{kn}}),$$
which gives rise to an $(A\otimes B)[X^{\pm1}]$-linear map
\begin{equation}\label{inducedmorp}
    \mathcal Q\cdot\mathbb M\colon J^{(0)}(X,\theta)\to J_{N^{\bullet},{\psi^\bullet}^{-1}}(I(X^{-1},\theta^{w_{kn}})).
\end{equation}
Note that $\theta^{w_{kn}}$ is also of type $(k,n)$. Now given a pair $(N(\mathcal Y),\psi_{\mathcal A})$ in the orbit $\lambda_{k,n}$, recall that we chose a generator $\operatorname{Wh}_{N(\mathcal Y),\psi_{\mathcal A}}$ of $\Hom_{\GL_{kn}(D)}(\theta,\Ind_{N(\mathcal Y)}^{\GL_{kn}(D)}(\psi_{\mathcal A}))\cong B$. This gives rise to a nontrivial map in $\Hom_{\GL_{kn}(D)}(\theta^{w_{kn}},\Ind_{N(\mathcal Y)}^{\GL_{kn}(D)}(\psi_{\mathcal A}^{-1}))$ by sending $v\in\theta$ to the map $g\mapsto\operatorname{Wh}_{N(\mathcal Y),\psi_{\mathcal A}}(v)(J^{-1}(g^*)^{-1}J)$. Note that we use here that $\psi_{\mathcal A}(J^{-1}(u^*)^{-1}J)=\psi_{\mathcal A}(u)^{-1}$ for $u\in N(\mathcal Y)$. 

We can now use the results of Section \ref{secratio} with respect to $(\pi,V)$, $I(X^{-1},\theta^{w_{kn}})$ and the map in $\Hom_{\GL_{kn}(D)}(\theta^{w_{kn}},\Ind_{N(\mathcal Y)}^{\GL_{kn}(D)}(\psi_{\mathcal A}^{-1}))$ described above. In particular Proposition \ref{annihipoly} implies that there is a polynomial $\mathcal P'\in S_{A\otimes B}$ such that $\mathcal P'$ annihilates 
$$M(\pi)^{\kappa}\otimes\left(J_{N^{\bullet},{\psi^\bullet}^{-1}}(I(X^{-1},\theta^{w_{kn}}))/J^{(0)}(X^{-1},\theta^{w_{kn}})\right).$$
By this and Equation (\ref{inducedmorp}) we obtain an $(A\otimes B)[X^{\pm1}]$-linear map
$$\mathcal P'(\operatorname{id}\otimes\mathcal Q\cdot\mathbb M)\colon M(\pi)^\kappa\otimes_{B[G\times G]}J^{(0)}(X,\theta)\to M(\pi)^\kappa\otimes_{B[G\times G]}J^{(0)}(X^{-1},\theta^{w_{kn}}).$$
Now Proposition \ref{bottompoly1} yields an $(A\otimes B)[X^{\pm1}]$-linear map
$$ M(\pi)^\kappa\otimes_{B[G\times G]}J^{(0)}(X^{-1},\theta^{w_{kn}})\to (A\otimes B)[X^{\pm1}]$$
and overall we obtain another element in (\ref{zetahom}) which we denote by $Z'$. More concretely, we have that $$Z'(v\otimes\lambda\otimes f)=Z(X^{-1},\mathcal P'\cdot(\operatorname{id}\otimes\mathcal Q\cdot\mathbb M)(v\otimes\lambda\otimes f))$$ for all $v\in V,\lambda\in\widetilde{V}$ and $f\in J^{(0)}(X,\theta)$.

Since the twisted doubling zeta integral is a generator of (\ref{zetahom}) we obtain that there is a unique element $c\in (A\otimes B)[X^{\pm1}]$ such that 
$$Z'(\varphi\otimes f)=cZ(X,\varphi\otimes f)$$
for all $\varphi\otimes f\in M(\pi)^\kappa\otimes_{B[G\times G]}J^{(0)}(X,\theta)$. To extend this result to functions in $I(X,\theta)$ note that by Proposition \ref{annihipoly} there is a polynomial $\mathcal P\in S_{A\otimes B}$ such that multiplication by $\mathcal P$ yields a map
$$M(\pi)^{\kappa}\otimes_{B[G\times G]}J_{N^\bullet,{\psi^\bullet}^{-1}}(I(X,\theta))\to M(\pi)^\kappa\otimes_{B[G\times G]}J^{(0)}(X,\theta).$$
Hence for any $f\in I(X,\theta)$ and $v\in V,\lambda\in\widetilde{V}$ we have that

\begin{align*}
   \mathcal PcZ(X,v,\lambda,f)&=cZ(X,\mathcal P(v\otimes\lambda\otimes f))\\
   &=Z'(\mathcal P(v\otimes\lambda\otimes f))\\
   &=Z(X^{-1},\mathcal P'\cdot(\operatorname{id}\otimes\mathcal Q\cdot\mathbb M)\mathcal P(v\otimes\lambda\otimes f))\\
   &=\mathcal P\mathcal P'Z(X^{-1},v,\lambda,\mathcal Q\cdot\mathbb M(f)).
\end{align*}
Hence if we set 
$$\Gamma(X,\pi,\theta,\psi)\coloneqq {\mathcal P'}^{-1}\mathcal Q^{-1}c\in S_{A\otimes B}^{-1}\cdot (A\otimes B)[X^{\pm1}]$$ we have that 
\begin{equation*}
Z(X^{-1},v,\lambda,\mathbb M(f))=\Gamma(X,\pi,\theta,\psi)Z(X,v,\lambda,f)  
\end{equation*}
for all $v\in V,\lambda\in\widetilde{V}$ and $f\in I(X,\theta)$. 
\end{proof}
\begin{remark}\label{resultoverc}
We will now explain how a multiplicity one theorem of Gourevitch-Kaplan in \cite{gourevitch2022multiplicity} over the complex numbers can be recovered from our results above. Note that when specializing our proofs to the complex numbers they basically reduce to the arguments of \cite{gourevitch2022multiplicity}, however as we are in the setting of \cite{cai2021twisted} we also cover the case of quaternionic unitary groups.

We set $R=A=B=\mathbb C$ and moreover fix an irreducible complex representation $(\pi,V)$ of $G$ and a complex representation $\theta$ of type $(k,n)$ as in Section 1.4 of \cite{gourevitch2022multiplicity}. Recall from Remark \ref{evatqs} that for any complex number $s$ we have that $I(X,\theta)\otimes_{\mathbb C[X^{\pm1}]}\mathbb C_s\cong I(q_E^{-s},\theta)$. The result we want to show (and which Gourevitch and Kaplan show in \cite{gourevitch2022multiplicity} for a class of classical groups) is that outside of a discrete set of values of $s$ we have that
\begin{equation}\label{dimone}
\dim_{\mathbb C}\Hom_{\mathbb C[G\times G]}\left(J_{N^\bullet,{\psi^\bullet}^{-1}}(I(q_E^{-s},\theta)),\widetilde{V}\otimes V\otimes\kappa^{-1}\right)=1.\end{equation}
Note that by the first part of Proposition \ref{homcomp} we have an isomorphism of $\mathbb C[X^{\pm1}]$-modules
$$M(\pi)^\kappa\otimes_{\mathbb C[G\times G]}J^{(0)}(X,\theta)\cong (V\otimes_{\mathbb C}\widetilde{V})\otimes_{\mathbb C[G^\Delta]}\mathbb C[X^{\pm1}]$$
and hence in particular an isomorphism
$$M(\pi)^\kappa\otimes_{\mathbb C[G\times G]}J^{(0)}(X,\theta)\otimes_{\mathbb C[X^{\pm1}]}\mathbb C_s\cong (V\otimes_\mathbb C\widetilde{V})\otimes_{\mathbb C[G^\Delta]}\mathbb C$$
of $\mathbb C$-vector spaces. Hence we can compute for any $s\in\mathbb C$ that 
\begin{align*}
    \Hom_{\mathbb C}(M(\pi)^\kappa\otimes_{\mathbb C[G\times G]}J^{(0)}(X,\theta)\otimes_{\mathbb C[X^{\pm1}]}\mathbb C_s,\mathbb C)&\cong\Hom_{\mathbb C}((V\otimes_\mathbb C\widetilde{V})\otimes_{\mathbb C[G^\Delta]}\mathbb C,\mathbb C)\\
    &\cong\Hom_{\mathbb C[G^\Delta]}(V\otimes\widetilde{V},\mathbb C)\\
    &\cong\Hom_{\mathbb C[G]}(V,\widetilde{\widetilde{V}})\\
    &\cong\Hom_{\mathbb C[G]}(V,V)\\
    &\cong\mathbb C,
\end{align*}
where the last step follows by Schur's lemma. By Propositions \ref{firstorbred} and \ref{annihipoly} there is a polynomial $\mathcal P$ in $S_\mathbb C$ such that multiplication by $\mathcal P$ yields a map
$$M(\pi)^{\kappa}\otimes_{\mathbb C[G\times G]}J_{N^\bullet,{\psi^\bullet}^{-1}}(I(X,\theta))\to M(\pi)^{\kappa}\otimes_{\mathbb C[G\times G]}J^{(0)}(X,\theta).$$
However, if $s\in\mathbb C$ satisfies that $\mathcal P(q_E^{-s})\not=0$ (which is clearly true for $s$ outside a discrete set), this gives rise to a $\mathbb C$-vector space isomorphism
$$M(\pi)^{\kappa}\otimes_{\mathbb C[G\times G]}J_{N^\bullet,{\psi^\bullet}^{-1}}(I(X,\theta))\otimes_{\mathbb C[X^{\pm1}]}\mathbb C_s\cong M(\pi)^{\kappa}\otimes_{\mathbb C[G\times G]}J^{(0)}(X,\theta)\otimes_{\mathbb C[X^{\pm1}]}\mathbb C_s.$$
Since $J_{N^\bullet,{\psi^\bullet}^{-1}}(I(X,\theta))\otimes_{\mathbb C[X^{\pm1}]}\mathbb C_s\cong J_{N^\bullet,{\psi^\bullet}^{-1}}(I(q_E^{-s},\theta))$ we obtain that 
$$\Hom_{\mathbb C}(M(\pi)^{\kappa}\otimes_{\mathbb C[G\times G]}J_{N^\bullet,{\psi^\bullet}^{-1}}(I(q_E^{-s},\theta)),\mathbb C)\cong\mathbb C,$$
outside of a discrete set of values of $s$. Then Equation (\ref{dimone}) follows since we have an isomorphism of $\mathbb C$-vector spaces
$$M(\pi)^{\kappa}\otimes_{\mathbb C[G\times G]}J_{N^\bullet,{\psi^\bullet}^{-1}}(I(q_E^{-s},\theta))\cong J_{N^\bullet,{\psi^\bullet}^{-1}}(I(q_E^{-s},\theta))\otimes_{\mathbb C[G\times G]}M(\pi)^{\kappa},$$
and tensor-hom adjunction.
\end{remark}

\appendix
\section{Results on $G^{\Box,k}$}\label{appendix}
In this appendix we collect some results purely on the structure of the group $G^{\Box,k}$ that are needed throughout this article.

Recall that $P$ is the parabolic subgroup of $G^{\Box,k}$ that stabilizes the maximal totally isotropic subspace $W^{\bigtriangleup,k}$. With respect to the basis $e_1,\dotsc,e_{2kn}$ of $W^{\Box,k}$ one has that 
            $$P=\begin{pmatrix}
                *&*\\0_{kn}&*
            \end{pmatrix}\cap G^{\Box,k}.$$
    Let $K$ be the compact open subgroup
		$$\left\{g\in G^{\Box,k}\mid \begin{pmatrix}
			1_{kn}&\\&J 
		\end{pmatrix}g\begin{pmatrix}
			1_{kn}&\\&J^{-1} 
		\end{pmatrix}\in\GL_{2kn}(\mathcal O_D)\right\}$$
    of $G^{\Box,k}$. We then have the following explicit Iwasawa decomposition and we would like to thank Shaun Stevens for suggesting this proof. 
    \begin{proposition}\label{iwasawadec} Suppose that the residue characteristic of $F$ is odd. Then we have that $P\cdot K=G^{\Box,k}$.
    \end{proposition}
    \begin{proof}
          Let $\mathcal L'$ be the free $\mathcal O_D$-module in $W^{\Box,k}$ spanned by 
            $$e_j'=\begin{pmatrix}
                1_{kn}&\\&J^{-1}
            \end{pmatrix}e_j,$$
            where $1\leq j\leq 2kn$. Note that $K$ is the stabiliser of $\mathcal L'$ in $G^{\Box,k}$. The Gram matrix of $e_1',\dotsc,e_{2kn}'$ equals
            $$\begin{pmatrix}
                &1_{kn}\\
                \epsilon&
            \end{pmatrix}$$
            and in particular $\mathcal L'$ is self-dual.
            
            Now for any element $g\in G^{\Box,k}$ we consider the free $\mathcal O_D$-module $g(\mathcal L')$. Note that $g(\mathcal L')$ is also self-dual. Since $\mathcal O_D$ is hereditary $g(\mathcal L')\cap W^{\bigtriangleup,k}$ is a free $\mathcal O_D$-module and we can find a basis $f_1,\dotsc,f_{kn}$ of $g(\mathcal L')\cap W^{\bigtriangleup,k}$. Since $g(\mathcal L')/g(\mathcal L')\cap W^{\bigtriangleup,k}$ is also free we can find elements $f_{kn+1},\dotsc,f_{2kn}$ such that $\mathcal F=\{f_1,\dotsc,f_{2kn}\}$ is an $\mathcal O_D$-basis of $g(\mathcal L')$. The Gram matrix of $\mathcal F$ has the form
            $$\begin{pmatrix}
                &Y\\
                \epsilon Y^*&Z
            \end{pmatrix},$$
            where $Y\in\GL_{kn}(\mathcal O_D), Z\in M_{kn}(\mathcal O_D)$ and $Z=\epsilon Z^*$. Now 
            $$\mathcal F'=\begin{pmatrix}
                1_{kn}&\\
                &Y^{-1}
            \end{pmatrix}\mathcal F=\left(\begin{pmatrix}
                1_{kn}&\\
                &Y^{-1}
            \end{pmatrix}f_1,\dotsc,\begin{pmatrix}
                1_{kn}&\\
                &Y^{-1}
            \end{pmatrix}f_{2kn}\right)$$
            is an $\mathcal O_D$-basis of $g(\mathcal L')$ with Gram matrix
            $$\begin{pmatrix}
                &1_{kn}\\
                \epsilon& {Y^*}^{-1}ZY^{-1}
            \end{pmatrix}.$$
            Let $Z'=-\epsilon/2\cdot {Y^*}^{-1}ZY^{-1}$. Since we assume that the residue characteristic of $F$ is odd we have that $Z'\in M_n(\mathcal O_D)$. Then
            $$\mathcal F''=\begin{pmatrix}
                1_{kn}&Z'\\
                &1_{kn}
            \end{pmatrix}\mathcal F'=(f_1'',\dotsc,f_{2kn}'')$$
            is an $\mathcal O_D$-basis of $g(\mathcal L')$ with Gram matrix
            $$\begin{pmatrix}
                &1_{kn}\\
                \epsilon&
            \end{pmatrix}.$$
            This implies that the automorphism $k_g$ of $W^{\Box,k}$ that maps $f_i''$ to $ge_i$ is an element of $G^{\Box,k}$ and since it stabilizes $g(\mathcal L')$ an element of $gKg^{-1}$. Now note that
            $p_g=k_g^{-1}g$ is an element of $G^{\Box,k}$ that sends $e_i$ to $f_i''$ and hence stabilizes $W^{\bigtriangleup,k}$, which implies that $p_g\in P$. Since $g=p_g(g^{-1}k_gg)$ and $g^{-1}k_gg\in K$ the result follows.
    \end{proof}
  
Recall the definition of the character $\psi^\bullet$ in Section \ref{psibullet}. We will now prove a different formula for the character $\psi^\bullet$. Note that for any $i=1,\dotsc,k-1$ and $u\in N^\bullet$ the endomorphism $u-\id$ maps $(\mathcal F_{i-1}^\bullet)^\perp$ to $(\mathcal F_{i}^\bullet)^\perp$ and hence gives rise to an element in $$\Hom_D((\mathcal F_{i-1}^\bullet)^\perp/(\mathcal F_{i}^\bullet)^\perp,(\mathcal F_{i}^\bullet)^\perp/(\mathcal F_{i+1}^\bullet)^\perp)$$ which we denote by $u_i'$.
For $i=1,\dotsc,k-2$ we define the elements $$A_i'\in\Hom_D((\mathcal F_i^\bullet)^\perp/(\mathcal F_{i+1}^\bullet)^\perp,(\mathcal F_{i-1}^\bullet)^\perp/(\mathcal F_{i}^\bullet)^\perp)\cong\Hom_D(W^\bigtriangleup_{k-i},W^\bigtriangleup_{k+1-i})$$ to be $(w,w)_{k-i}\mapsto (-w,-w)_{k-i+1}$. Moreover, we set $$A_{k-1}'\in\Hom_{D}(\mathcal (\mathcal F^\bullet_{k-1})^\perp/\mathcal F^\bullet_{k-1},(\mathcal F^\bullet_{k-2})^\perp/(\mathcal F^\bullet_{k-1})^\perp)\cong\Hom_D(W^\Box_1,W^\bigtriangleup_2)$$
    to be the map $(w_1,w_2)_1\mapsto (-w_1,-w_1)_2$. We then have the following.
\begin{lemma}\label{diffformula}
    For any $u\in N^\bullet$ we have that
    $$\psi^{\bullet}(u)=\psi\left(\sum_{i=1}^{k-1}\operatorname{Trd}(u_i'\circ A_i')\right).$$
\end{lemma}
\begin{proof}
For a right $D$-module $M$ we denote the left $D$-module $\Hom_D(M,D)$ by $M^\vee$. We can define a right $D$-module structure on $M^\vee$ by setting that $(\varphi\cdot a)(w)=\rho(a)\varphi(w)$ for $a\in D,\varphi\in M^\vee$ and $w\in M$. We denote this right $D$-module as $\prescript{\rho}{}{M}^\vee$. Moreover, for a $D$-linear morphism $\alpha\colon M\to M'$ between to right $D$-modules $M$ and $M'$ we define a $D$-linear map $\alpha^\vee\colon M'^\vee\to M^\vee$ by setting that $\alpha^\vee(f)= f\circ\alpha$ for $f\in M'^\vee$.
For any integer $1\leq i\leq k$ the morphism of right $D$-modules
    \begin{align*}
    \Phi_i\colon(\mathcal F_{i-1}^\bullet)^\perp/(\mathcal F_{i}^\bullet)^\perp&\to\prescript{\rho}{}{(\mathcal F_i^\bullet/\mathcal F_{i-1}^\bullet)}^\vee\\
    x+(\mathcal F_{i}^\bullet)^\perp&\mapsto\lambda_x,
    \end{align*}
    where $\lambda_x(w)\coloneqq 
h^{\Box,k}(x,w)$ for $w\in\mathcal F_i^\bullet/\mathcal F_{i-1}^\bullet$, is an isomorophism by dimension reasons (where we again set $\mathcal F_k^\bullet=(\mathcal F_{k-1}^\bullet)^\perp$). 

Fix an integer $1\leq i\leq k-1$. We claim that for any $u\in N^\bullet$ we have
\begin{equation}\label{commdiag}
    (-u_i)^\vee\circ\Phi_i=\Phi_{i+1}\circ u_i'.
\end{equation} Let $x+(\mathcal F_{i}^\bullet)^\perp\in(\mathcal F_{i-1}^\bullet)^\perp/(\mathcal F_{i}^\bullet)^\perp$ and $w+\mathcal F_{i}^\bullet\in\mathcal F_{i+1}^\bullet/\mathcal F_{i}^\bullet$. We have that $(\Phi_{i+1}\circ u_i')(x)=\lambda_{(u-\id)x}$ and since $(u^{-1}-\id)w-(-u_i)w\in\mathcal F_i^\bullet$ we obtain
$$(\Phi_{i+1}\circ u_i')(x)(w)=\lambda_{(u-\id)x}(w)=h^{\Box,k}(ux-x,w)=h^{\Box,k}(x,(u^{-1}-\id)w)=h^{\Box,k}(x,-u_iw).$$
However, note that
$$((-u_i)^\vee\circ\Phi_i)(x)(w)=\lambda_x(-u_iw)=h^{\Box,k}(x,-u_iw),$$
which implies Equation \ref{commdiag}. It is straightforward to check that 
$$\Phi_{i}\circ A_i'=(-A_i)^\vee\circ\Phi_{i+1}$$
and hence we obtain a commutative diagram
    $$\begin{tikzcd}
       (\mathcal F_{i}^\bullet)^\perp/(\mathcal F_{i+1}^\bullet)^\perp\arrow[r,"A_i'"]\arrow[d]& (\mathcal F_{i-1}^\bullet)^\perp/(\mathcal F_i^\bullet)^\perp\arrow[r,"u_i'"]\arrow[d]&(\mathcal F_{i}^\bullet)^\perp/(\mathcal F_{i+1}^\bullet)^\perp\arrow[d]\\
        \prescript{\rho}{}{(\mathcal F_{i+1}^\bullet/\mathcal F_{i}^\bullet)}^\vee\arrow[r,"(-A_i)^\vee"]&\prescript{\rho}{}{(\mathcal F_i^\bullet/\mathcal F_{i-1}^\bullet)}^\vee\arrow[r,"-u_i^\vee"]&\prescript{\rho}{}{(\mathcal F_{i+1}^\bullet/\mathcal F_{i}^\bullet)}^\vee.
    \end{tikzcd}$$
  In particular this implies that $$\operatorname{Trd}(u_i'\circ A_i')=\operatorname{Trd}((-u_i)^\vee\circ(-A_i)^\vee).$$ Note that by Theorem 4.2 of \cite{knus1998book}, when choosing a basis the matrix associated to $A_i^\vee$ equals $A_i^*$, which is given by taking the transpose and applying the involution $\rho$ to all entries of $A_i$. We obtain that $$\operatorname{Trd}((-u_i)^\vee\circ(-A_i)^\vee)=\operatorname{Trd}((A_i\circ u_i)^\vee)=\operatorname{Trd}(A_i\circ u_i)=\operatorname{Trd}(u_i\circ A_i).$$   
Since by definition $$\psi^\bullet(u)=\psi\left(\sum_{i=1}^{k-1}\operatorname{Trd}(u_i\circ A_i)\right)$$
the result follows.
\end{proof}

Recall the notation from Section \ref{auxil}, in particular the subgroups $\overline{N}_i$ and $\mathfrak a_i$ of $\varepsilon_i^{-1}M\varepsilon_i$, for which we have the following result. 
\begin{proposition}\label{semidir}
The equality $\overline{N}_i=(\varepsilon_i^{-1}P\varepsilon_i\cap N^\bullet)\rtimes\mathfrak a_i$ holds.
\end{proposition}
\begin{proof} We start by showing that both $\varepsilon_i^{-1}P\varepsilon_i\cap N^\bullet$ and $\mathfrak a_i$ are subgroups of $\overline{N}_i$. It is straightforward to see that $\xi_k(N_i')$ is a subgroup of $N(\mathcal Y_i)$ and hence $\mathfrak a_i\subseteq\overline{N}_i$. Moreover, if $g\in \varepsilon_i^{-1}M\varepsilon_i\cap N^\bullet$ we have that $g(L_i)=L_i$ but since $g\in N^\bullet$ we also have $(g-\id)(W_1^\Box)\subseteq W_2^\bigtriangledown\oplus\dotsc\oplus W_k^\bigtriangledown$ and hence $g|_{\tilde{L}_i}=\id$. Since $g\in N^\bullet$ we have for $2\leq j\leq k$ that $$(g-\id)(W_j^\bigtriangleup)\subseteq W_1^{\Box}\oplus\dotsc\oplus W_{j-1}^{\Box}\oplus W_{j}^{\bigtriangledown}\oplus\dotsc\oplus W_k^\bigtriangledown,$$
but since $g(L_i)=L_i$ we obtain that $g\in N(\mathcal Y_i)$.

We proceed by constructing a map $\overline{N}_i\to\mathfrak a_i$ that is the identity on $\mathfrak a_i$. Note that any element $x$ in $P(\mathcal Y_i)\cap\varepsilon_i^{-1}M\varepsilon_i$ induces an endomorphism $\varrho_x$ of $W_1^\Box$ by composing $x$ with the projection $W^{\Box,k}\twoheadrightarrow W^{\Box,k}/(W_2^\Box\oplus\dotsc\oplus W_k^\Box)$. We will now prove that $\varrho_x$ is an element of $G^{\Box,1}$.

Clearly, for $w\in\tilde{L}_i$ we have that $\varrho_x(w)\in\tilde{L}_i$ and $x(w)=\varrho_x(w)$. For $w^-\in\tilde{L}_i^-$ we have since $x\in\varepsilon_i^{-1}M\varepsilon_i$ that 
$x(w^-)=\varrho_x(w^-)+s$ where $s\in W_2^\bigtriangledown\oplus\dotsc\oplus W_k^\bigtriangledown$. Hence for any $w_1,w_2\in W_1^{\Box}$ we can write
   $$x(w_j)=\varrho_x(w_j)+s_j,$$ where 
   $s_j\in W_2^\bigtriangledown\oplus\dotsc\oplus W_k^\bigtriangledown$, for $j=1,2$. Since $$(W_2^\bigtriangledown\oplus\dotsc\oplus W_k^\bigtriangledown)^\perp=W_1^\Box\oplus W_2^\bigtriangledown\oplus\dotsc\oplus W_k^\bigtriangledown,$$ we obtain 
\begin{align*}
    h^{\Box,k}(w_1,w_2)&=h^{\Box,k}(x(w_1),x(w_2))\\
    &=h^{\Box,k}(\varrho_x(w_1)+s_1,\varrho_x(w_2)+s_2)\\
    &=h^{\Box,k}(\varrho_x(w_1),\varrho_x(w_2)).
\end{align*}
For $w_1,w_2\in W_1^\Box$ we clearly have $h^{\Box,k}(w_1,w_2)=h^{\Box,1}(w_1,w_2)$ and since $\varrho_{x^{-1}}=\varrho_x^{-1}$ we see that $\varrho_x$ yields an element of $G^{\Box,1}$.

If $x\in N(\mathcal Y_i)\cap\varepsilon_i^{-1}M\varepsilon_i$, then $x$ acts as the identity on $(0,U_i)_1$ and on $\tilde{L}_i/(0,U_i)_1$, which immediately implies that $\varrho_x$ is in $N_i'$. Moreover, $\varrho_x(\tilde{L}_i^-)=\tilde{L}_i^-$ which shows that $\xi_k(\varrho_x)\in\mathfrak a_i$. Hence we have group homomorphism
$\overline{N}_i\to\mathfrak a_i$ which is clearly the identity on $\mathfrak a_i$.

Since $N^\bullet$ acts as the identity on $(W_1^\Box\oplus W_2^\bigtriangledown\oplus\dotsc\oplus W_k^\bigtriangledown)/(W_2^\bigtriangledown\oplus\dotsc\oplus W_k^\bigtriangledown)$ we see that $\varepsilon_i^{-1}M\varepsilon_i\cap N^\bullet$ is in the kernel of the above morphism. If we have an element $g\in\overline{N}_i$ in this kernel then since $g\in\varepsilon_i^{-1}M\varepsilon_i$ we see that $g$ stabilizes $\varepsilon_i^{-1}(W^{\bigtriangledown}_2\oplus\dotsc\oplus W_k^\bigtriangledown)$ and this together with $g\in N(\mathcal Y_i)$ implies that
$$(g-\id)(W_j^\bigtriangledown)\subseteq W_{j+1}^\bigtriangledown\oplus\dotsc\oplus W_k^\bigtriangledown$$
for $j\geq 2$. Moreover, since we assume that $\overline{g}=\id_{W_1^\Box}$ and $g(\tilde{L}_i)=\tilde{L}_i$ we obtain that
$g|_{\tilde{L}_i}=\id$. Now as $g\in\varepsilon_i^{-1}M\varepsilon_i\cap N(\mathcal Y_i)$ we obtain that $(g-\id)\tilde{\varepsilon}_i^{-1}(W_1^\bigtriangledown)\subseteq W_2^\bigtriangledown\oplus\dotsc\oplus W_k^\bigtriangledown$ which implies that $g\in N^\bullet$.
\end{proof}

\begin{lemma}\label{semidirchar}
    For $x\in(\varepsilon_i^{-1}P\varepsilon_i\cap N^\bullet)$ and $u\in\mathfrak a_i$ we have that 
    $$\psi^\bullet(u^{-1}xu)=\psi^\bullet(x).$$
\end{lemma}
\begin{proof}
In the proof of the above Proposition \ref{semidir} we already saw that $x$ acts as the identity on $\tilde{L}_i$.
Since $u$ acts as the identity on $W_2^\Box\oplus\dotsc\oplus W_k^\Box$ we obtain that $u^{-1}xu$ and $x$ agree on $W_2^\bigtriangledown\oplus\dotsc\oplus W_k^\bigtriangledown$. Hence to prove the result it is enough to show that
$$\operatorname{Trd}\left(A_{k-1}\circ (u^{-1}xu)_{k-1}\right)=\operatorname{Trd}\left(A_{k-1}\circ x_{k-1}\right).$$
Recall that we write $(u^{-1}xu)_{k-1}$, respectively $x_{k-1}$, for the linear morphisms in $\Hom_D(W_1^\Box,W_2^\bigtriangledown)$ that are induced by $u^{-1}xu-\id$ respectively $x-\id$. As the image of $A_{k-1}$ is $(W,0)_1$ we see that 
$$\operatorname{Trd}(A_{k-1}\circ(u^{-1}xu)_{k-1})=\operatorname{Trd}(A_{k-1}\circ(u^{-1}xu)_{k-1}|_{(W,0)_1})$$
and 
$$\operatorname{Trd}(A_{k-1}\circ x_{k-1})=\operatorname{Trd}(A_{k-1}\circ x_{k-1}|_{(W,0)_1}).$$
We will finish the proof by showing that $x_{k-1}|_{(W,0)_1}=(u^{-1}xu)_{k-1}|_{(W,0)_1}$. Note that $(u-\id)((0,U_i)_1^\perp)\subseteq\tilde{L_i}$ and since $(W,0)_1\subseteq (0,U_i)_1^\perp$ we can write for any $w\in (W,0)_1$ that $u(w)=w+l$ for some $l\in\tilde{L_i}$. Now $x(w)=w+w_2+r$ where $w_2\in W_2^\bigtriangledown$ and $r\in W_3^\bigtriangledown\oplus\dotsc\oplus W_k^\bigtriangledown$. Since $x(l)=l$ we see that $xu(w)=w+l+w_2+r$. Now as $u^{-1}(w)=w-u^{-1}(l)$ we see that 
    $$u^{-1}xu(w)=w+w_2+r=x(w),$$
    which finishes proof.
\end{proof}
Let $\psi'$ be a character of $\mathfrak a_i$ and define a map $\psi^\bullet\rtimes\psi'\colon\overline{N}_i\to\mathbb Z[1/p,\mu_{p^\infty}]^\times$ by setting for $u\in (\varepsilon_i^{-1}P\varepsilon_i\cap N^\bullet)$ and $u'\in\mathfrak a_i$ that $(\psi^\bullet\rtimes\psi')(uu')=\psi^\bullet(u)\psi'(u')$. The above lemma implies that this yields a well-defined character of $\overline{N}_i$. By abuse of notation we will write $\psi^\bullet$ for the extension of $\psi^\bullet$ to $\overline{N}_i$ which is trivial on $\mathfrak a_i$.
Note that the map $x\mapsto x|_{L_i}$ yields a group isomorphism between $\varepsilon_i^{-1}M\varepsilon_i$ and $\GL_D(L_i)\cong\GL_{kn}(D)$. We have the following result. 
\begin{proposition}\label{higherorb}
\begin{enumerate}
    \item The pair $(M\cap N^\bullet,{\psi^{\bullet}}^{-1})$ lies in the orbit $\lambda_{k,n}$.
    \item For a nontrivial character $\psi'$ of $\mathfrak a_i$ the pair $(\overline{N}_i,\psi^\bullet\rtimes\psi')$ lies in an orbit higher than $\lambda_{k,n}$.

\end{enumerate}

\end{proposition}
\begin{proof}
\underline{ad 1):} As a subgroup of $M\cong\GL_D(W^{\bigtriangleup,k})$ one has that $M\cap N^\bullet$ stabilizes the flag
$$0\subseteq W^\bigtriangleup_1\subseteq W^\bigtriangleup_1\oplus W^\bigtriangleup_2\subseteq\dotsc\subseteq W^{\bigtriangleup,k}.$$
Note that $\dim_D(W_1^\bigtriangleup\oplus\dotsc\oplus W_j^\bigtriangleup)=nj$.
Recall that for $j=1,\dotsc,k-1$ and $u\in M\cap N^\bullet$ we write $u_{j}'$ for the map in $\Hom_D(W^\bigtriangleup_{k-j+1},W^\bigtriangleup_{k-j})$ that is induced by $u-\id$ (note that since $u\in M$ the image of $u_{k-1}'$ is contained in $W_1^\bigtriangleup$). By Lemma \ref{diffformula} we have that
$$\psi^\bullet(u)=\sum_{j=1}^{k-1}\psi\left(\operatorname{Trd}(u_j'\circ A'_j)\right)$$
for $u\in M\cap N^\bullet$. The maps $A_{j}'\in\Hom_{D}(W_{k-j}^\bigtriangleup,W_{k-j+1}^\bigtriangleup)$ are clearly isomorphisms for $j=1,\dotsc,k-2$. Since the image of $u_{k-1}'$ is contained in $W_1^\bigtriangleup$ we have that 
$$\operatorname{Trd}(u_{k-1}'\circ A'_{k-1})=\operatorname{Trd}(u_{k-1}'\circ (A'_{k-1}|_{W^{\bigtriangleup}_1})).$$
Now $A'_{k-1}|_{W^{\bigtriangleup}_1}\in\Hom_{D}(W_1^\bigtriangleup,W_2^\bigtriangleup)$ is clearly an isomorphism and the result follows.

\underline{ad 2):} When viewed as a subset of $\GL_D(L_i)$ the group $\overline{N}_i$ stabilizes the flag 
$$\mathcal L\colon 0\subseteq (0,U_i)_1\subseteq\tilde{L_i}\subseteq\tilde{L_i}\oplus W_2^\bigtriangleup\subseteq\dotsc\subseteq L_i=D^{kn}.$$
For $j=0,\dotsc,k+1$ we will write $\mathcal L^j$ for the $j$-th term in this flag. Moreover, for $j=1,\dotsc,k$ and $u\in\overline{N}_i$ we will write $u_{k+1-j}'$ for the map in $\Hom_D(\mathcal L^{j+1}/\mathcal L^{j},\mathcal L^j/\mathcal L^{j-1})$ that is induced by $u-\id$. Note that $\psi'$ corresponds to a nontrivial choice of an element $A_k'\in\Hom_D((0,U_i)_1,\tilde{L}_i/(0,U_i)_1)$ such that for $u\in\mathfrak a_i$ we have that $\psi'(u)=\psi(\operatorname{Trd}(u_k'\circ A_k'))$. By Lemma \ref{diffformula} we have for any $u\in\varepsilon_i^{-1}M\varepsilon_i\cap N^\bullet$ that
$$\psi^\bullet(u)=\psi\left(\sum_{j=1}^{k-1}\operatorname{Trd}(u_j'\circ A_j')\right).$$
This shows for $u\in\overline{N}_i$ that
$$\psi^\bullet\rtimes\psi'(u)=\psi\left(\sum_{j=1}^{k}\operatorname{Trd}(u_j'\circ A_j')\right).$$
Since $A_k'$ is nontrivial there is an element $x\in (0,U_i)_1$ such that $A_k'(x)=(x_1,x_2)+(0,U_i)_1$ is nonzero, which immediately implies that $x_1$ is nonzero. Then $(A_k-1'\circ A_k')(x)=-(x_1,x_1)$ is nonzero. Since for $j\leq 2\leq k-1$ the maps $A_j'$ are all isomorphisms we see that
$$A_{1}'\circ\dotsc\circ A_k'\not=0.$$
According to Definition \ref{orbdef} this shows that $(\overline{N}_i,\psi^\bullet\rtimes\psi')$ lies in an orbit higher than $\lambda_{k,n}$.
\end{proof}

  For $0\leq i\leq r_0$ consider the flag $$\mathcal Z_i\colon 0\subseteq (0,U_i)_1\subseteq L_i$$ of totally isotropic subspaces in $(W^{\Box,k},h^{\Box,k})$ and the associated parabolic subgroup $P(\mathcal Z_i)$ of $G^{\Box,k}$ with unipotent radical $N(\mathcal Z_i)$.

  \begin{lemma}\label{bulltriv}
      We have that $N(\mathcal Z_i)\cap\varepsilon_i^{-1}M\varepsilon_i$ is a normal subgroup of $\overline{N}_i$ and 
      $$\psi^\bullet(u)=1$$
      for all $u\in N(\mathcal Z_i)$.
  \end{lemma}
  \begin{proof} Let $x\in\overline{N}_i=N(\mathcal Y_i)\cap\varepsilon_i^{-1}M\varepsilon_i$ and $u\in N(\mathcal Z_i)$. The first statement follows if we show that $(x^{-1}nx-\id)(v)\subseteq (0,U_i)_1$ for all $v\in L_i$. If $v\in L_i$ we have that $u(x(v))=x(v)+r$, where $r\in (0,U_i)_1$, and hence $$x^{-1}nx(v)-v=x^{-1}(r)=r.$$ 
 Now for $u\in N(\mathcal Z_i)$ let $u'$ be its image $\mathfrak a_i$ under the surjection $\overline{N}_i\twoheadrightarrow\mathfrak a_i$ constructed in Proposition \ref{semidir}. We have that $uu'^{-1}\in \varepsilon_i^{-1}P\varepsilon_i\cap N^\bullet$ and since $\psi^\bullet$ is trivial on $\mathfrak a_i$ we have $\psi^\bullet(u)=\psi^\bullet(uu'^{-1})$. Now $u\in N(\mathcal Z_i)$ implies that $(u-\id)(W_2^\bigtriangledown\oplus\dotsc\oplus W_k^\bigtriangledown)\subseteq L_i$. However, since $L_i\cap (W_2^\bigtriangledown\oplus\dotsc\oplus W_k^\bigtriangledown)=\{0\}$ and $u\in\varepsilon_i^{-1}M\varepsilon_i$ we see that $u$ acts as the identity on $W_2^\bigtriangledown\oplus\dotsc\oplus W_k^\bigtriangledown$ and the same is true for $uu'^{-1}$. In particular this shows that 
 $$\operatorname{Trd}((uu'^{-1})_j\circ A_j)=0$$
 for $j=1,\dotsc,k-2$.
      
We proceed by showing that $\operatorname{Trd}((uu'^{-1})_{k-1}\circ A_{k-1})=0$, which, since the image of $A_{k-1}$ is contained in $(W,0)_1$, follows from proving that $uu'^{-1}$ acts as the identity on $(W,0)_1$. Firstly, since $uu'^{-1}$ is in $N^\bullet$ we obtain that $(uu'^{-1}-\id)(W_1^\Box)\subseteq W_2^\bigtriangledown\oplus\dotsc\oplus W_k^\bigtriangledown$. Moreover, note that $(u'^{-1}-\id)(W,0)_1\subseteq\tilde{L}_i$ and $(u-\id)(W,0)_1\subseteq L_i$, which implies that $(uu'^{-1}-\id)(W,0)_1\subseteq L_i$. Since $L_i\cap (W_2^\bigtriangledown\oplus\dotsc\oplus W_k^\bigtriangledown)=\{0\}$ we obtain that $uu'^{-1}$ acts as the identity on $(W,0)_1$ and the claim follows.
      
  \end{proof}
\begin{lemma}\label{homeobottom}
    The map $x\mapsto Px$ yields a homeomorphism
    $$N^\circ\iota(G\times 1)\cong P\backslash PN^\bullet\iota(G\times G).$$
\end{lemma}
\begin{proof} We write $H$ for $N^\bullet\iota(G\times G)$ and note that $\iota(G\times G)\subseteq P^\bullet$ which implies that $\iota(G\times G)$ normalizes $N^\bullet$. Hence $H$ is actually a subgroup of $G^{\Box,k}$ and $H=\iota(G\times G)N^\bullet$. Moreover, note that $N^\bullet\cap\iota(G\times G)=1$ and recall that $N^\circ=N(W^{\bigtriangledown,k})\cap N^\bullet$. We have a homeomorphism $P\backslash PH\cong (P\cap H)\backslash H$.

Every element $h\in H$ gives rise to $\tau^+(h)\in\End_D(W)$ by setting it to be the composition
    $$(W,0)_1\xrightarrow[]{h}W_1^\Box\oplus W_{2}^\bigtriangledown\oplus\dotsc\oplus W_k^\bigtriangledown\twoheadrightarrow (W,0)_1$$
    and similarly to $\tau^-(h)\in\End_D(W)$ defined to be the composition
        $$(0,W)_1\xrightarrow[]{h}W_1^\Box\oplus W_{2}^\bigtriangledown\oplus\dotsc\oplus W_k^\bigtriangledown\twoheadrightarrow (0,W)_1.$$
          Note that for an element $h=n\iota(g_1,g_2)\in H$, where $n\in N^\bullet$ and $g_1,g_2\in G$, we have for any $w_1,w_2\in W$ that $h(w_1,w_2)_1-(g_1w_1,g_2w_2)_1\in W_2^\bigtriangledown\oplus\dotsc\oplus W_k^\bigtriangledown$ and hence $\tau^+(h)=g_1$ and $\tau^-(h)=g_2$. We obtain a continuous surjection $\tau\colon H\twoheadrightarrow G\times G,h\mapsto (\tau^+(h),\tau^-(h))$.
         
          For any $h=u\iota(g_1,g_2)\in H$, where $u\in N^\bullet$ and $g_1,g_2\in G$, we have that
         $$\iota(\tau^-(h),\tau^-(h))^{-1}\cdot h\cdot\iota(\tau^+(h)^{-1}\tau^-(h),1)=\iota(g_2^{-1},g_2^{-1})u\iota(g_2,g_2)$$ lies in $N^\bullet$, which yields a surjective continuous map $H\twoheadrightarrow N^\bullet$. Note that $N^\bullet\subseteq P(W^{\bigtriangledown,k})$ which implies that the Levi decompositon $P(W^{\bigtriangledown,k})=MN(W^{\bigtriangledown,k})$ gives rise to a continuous surjection $\tau'\colon N^\bullet\to M\cap N^\bullet$ (note that $\tau'$ is characterized by $\tau'(u)|_{W^{\bigtriangledown,k}}=u|_{W^{\bigtriangledown,k}}$). Then $u\mapsto \tau'(u)^{-1}n$ is a continuous surjection $N^\bullet\twoheadrightarrow N^\circ$. Overall, we get a continuous surjective map $\tau''\colon H\twoheadrightarrow N^\circ$.
         
         The map $H\twoheadrightarrow\iota(G\times 1),h\mapsto \iota(\tau^-(h)^{-1}\tau^+(h),1)$ is continuous and hence $$h\mapsto\tau''(h)\iota(\tau^-(h)^{-1}\tau^+(h),1)$$ is a continuous surjection $H\twoheadrightarrow N^\circ\iota(G\times 1)$. It is straightforward to check that this map factors through $(P\cap H)\backslash H$ and is an inverse to $x\mapsto Px$.

\end{proof}

	\bibliographystyle{abbrv}
	\bibliography{references}
\end{document}